\documentclass[12pt, letterpaper]{amsart}

\newcommand{\indentalign}{\hspace{0.3in}&\hspace{-0.3in}}
\newcommand{\la}{\langle}
\newcommand{\ra}{\rangle}

\newcommand{\defeq}{\stackrel{\rm{def}}{=}}
\newcommand{\supp}{\operatorname{supp}}

\newcommand{\hi}{\textnormal{hi}}
\newcommand{\lo}{\textnormal{lo}}

\newcommand{\sgn}{\operatorname{sgn}}
\newcommand{\pv}{\operatorname{pv}}
\newcommand{\am}{\mathfrak{a}}
\newcommand{\cm}{\mathfrak{c}}

\usepackage{amssymb}
\usepackage{amsthm}
\usepackage{amsxtra}
\usepackage{graphicx}
\usepackage{xcolor}
\usepackage[T1]{fontenc} 
\usepackage{amsaddr}

\newtheorem{theorem}{Theorem}

\newtheorem{proposition}[theorem]{Proposition}
\newtheorem{lemma}[theorem]{Lemma}
\newtheorem{corollary}[theorem]{Corollary}

\theoremstyle{remark}

\numberwithin{equation}{section}

\numberwithin{theorem}{section}

\numberwithin{table}{section}

\numberwithin{figure}{section}

\ifx\pdfoutput\undefined
  \DeclareGraphicsExtensions{.pstex, .eps}
\else
  \ifx\pdfoutput\relax
    \DeclareGraphicsExtensions{.pstex, .eps}
  \else
    \ifnum\pdfoutput>0
      \DeclareGraphicsExtensions{.pdf}
    \else
      \DeclareGraphicsExtensions{.pstex, .eps}
    \fi
  \fi
\fi

\title[BO soliton dynamics]{Benjamin-Ono soliton dynamics in a \\ slowly varying potential revisited}

\author{Justin Holmer}
\email{justin\_holmer@brown.edu}
\address{Brown University, Box 1917, 151 Thayer St., Providence, RI 02912, USA}

\author{Katherine Zhiyuan Zhang}
\email{zz3463@nyu.edu}
\address{Warren Weaver Hall, Office 519, 251 Mercer St, New York, NY 10012, USA}

\begin{document}

\maketitle

\begin{abstract}
The Benjamin Ono equation with a slowly varying potential is
$$
\text{(pBO)} \qquad u_t + (Hu_x-Vu + \tfrac12 u^2)_x=0
$$
with $V(x)=W(hx)$, $0< h \ll 1$, and $W\in C_c^\infty(\mathbb{R})$, and $H$ denotes the Hilbert transform.  The soliton profile is 
$$Q_{a,c}(x) = cQ(c(x-a)) \,, \text{ where } Q(x) = \frac{4}{1+x^2}$$ 
and $a\in \mathbb{R}$, $c>0$ are parameters.   For initial condition $u_0(x)$ to (pBO) close to $Q_{0,1}(x)$, it was shown in Zhang \cite{Z} that the solution $u(x,t)$ to (pBO) remains close  to $Q_{a(t),c(t)}(x)$ and approximate parameter dynamics for $(a,c)$ were provided, on a dynamically relevant time scale.    In this paper, we prove \emph{exact} $(a,c)$ parameter dynamics.  This is achieved using the basic framework of the paper \cite{Z} but adding a \emph{local virial} estimate for the linearization of (pBO)  around the soliton.  This is a local-in-space estimate averaged in time, often called a \emph{local smoothing} estimate, showing that effectively the remainder function in the perturbation analysis is smaller near the soliton than globally in space.    A weaker version of this estimate is proved in Kenig \& Martel \cite{KM} as part of a ``linear Liouville'' result, and we have adapted and extended their proof for our application.
\end{abstract}

\section{Introduction}

Let $H$ be the Hilbert transform, corresponding to the Fourier multiplier $i\sgn \xi$, so that the operator $D=-\partial_x H$ is the positive operator with Fourier multiplier $|\xi|$.   (For further elaboration on notational conventions, see \S \ref{S:notation}.)  The Benjamin-Ono equation (BO) is
$$
\text{(BO)} \qquad 
\partial_t u = \partial_x (-H\partial_x u - \frac12 u^2)
$$
with $u$ real-valued, on $\mathbb{R}$.  The equation (BO) is a model for 1D long internal waves in a stratified fluid, introduced by Benjamin \cite{Ben} and Ono \cite{Ono}.  By working with the three transformations $u(x,-t)$, $u(-x,t)$, and $-u(x,t)$ we are in fact covering all four sign choices in $\partial_t u = \partial_x (\pm H\partial_x u \pm \frac12 u^2)$, and hence we do not have a distinction between ``focusing" or ``defocusing" problems for this equation.   Moreover, (BO) also satisfies translational invariance in space and has the scaling invariance, for $\lambda>0$,
$$u \text{ solves (BO)} \implies u_\lambda(x,t) = \lambda u(\lambda x, \lambda^2 t) \text{ solves (BO)}$$
(BO) is completely integrable, so it enjoys infinitely many conserved quantities \cite{BK}, the first three of which are
$$
M_0(u) = \frac12 \int u^2\,, \quad E_0(u) = -\frac12\int uHu_x - \frac{1}{6} \int u^3, $$
$$E_1(u) = \frac12\int u^2_x +\frac{3}{8}\int u^2 Hu_x -\frac{1}{16} \int u^4 $$
Tao \cite{Tao} proved local well-posedness of (BO) in $H_x^1$, and global well-posedness follows using the aforementioned conserved quantities.  This result followed several earlier results at higher regularity, including \cite{Saut, Iorio, GV, Ponce, KT, KK}.  The innovation Tao introduced was a gauge transformation to reduce the effective regularity of the nonlinearity. Following \cite{Tao}, there were a few improvements to even lower regularity, using the gauge transformation idea combined with bilinear Strichartz estimates, culminating in the $L^2$ result by Ionescu \& Kenig \cite{IK} and Molinet \& Pilod \cite{MolPil}.

More recently, there have been substantial innovations in the study of (BO) and related equations.  Saut \cite{Saut2019} provides an overview of the derivations from physical models and the mathematical literature.  Mu\~noz \& Ponce \cite{MP2019} and Linares, Mendez, \& Ponce \cite{LMP2021} obtained local $L^\infty$ estimates on an expanding spatial window as $t\to \infty$.  A normal forms procedure in the format of the ``quasilinear modified energy method'' was developed by Ifrim \& Tataru \cite{IT2019} resulting in a new dispersive decay estimate for $L^2$ weighted initial data and its application to a new proof of $L^2$ global well-posedness.  Kim \& Kwon \cite{KK2019} obtained $H^{1/2}$ scattering for defocusing higher-power nonlinearity via monotonicity estimates and the concentration compactness and rigidity method.   A unique continuation result for (BO) was obtained by Kenig, Ponce, \& Vega \cite{KPV2020}.  Deng, Tzvetkov,  \& Visciglia \cite{TV2013, TV2014, TV2015, D2015, DTV2015} constructed invariant measures concentrated on Sobolev spaces $H^s(\mathbb{T})$ and Sy \cite{S2018} constructed a measure concentrated on $C^\infty(\mathbb{T})$.  There have been advances in the integrability and inverse scattering theory associated to (BO).  In particular, G\'erard, Kappeler, and Topalov \cite{GKT2020} have studied the Lax operator on the $\mathbb{T}$, while Wu \cite{Wu2016, Wu2017} has studied the direct scattering problem on $\mathbb{R}$.  Miller \& Wetzel \cite{MW2016, MW2016b} have done calculations for rational data and studied the small dispersion limit.  Soliton dynamics and blow-up have been considered by Gustafson, Takaoka, Tsai \cite{GTT}, Kenig \& Martel \cite{KM}, Martel \& Pilod \cite{MP2017}, and Zhang \cite{Z}.  New numerical simulations for solitons and blow-up have been produced by Ria\~no, Roudenko, Wang \& Yang \cite{RWY2020, RRY2021}.  Boundary value problems have been studied by Hayashi \& Kaikina \cite{HK2012} and control and stabilization by Laurent, Linares, \& Rosier \cite{LLR2015}.

In this paper, our interest is in soliton dynamics.   Amick \& Toland \cite{AT} and Frank \& Lenzmann \cite{FrLen} showed that there is a unique (up to translations) nontrivial $L^\infty$ solution to 
$$Q -HQ' - \frac12 Q^2 =0$$
given by
$$Q(y) = \frac{4}{1+y^2}$$
For any $c>0$, $a\in \mathbb{R}$, taking $Q_{a,c}(x) = cQ(c(x-a))$ we have
\begin{equation}
\label{E:scaled-sol}
cQ_{a,c} - HQ_{a,c}' - \frac12 Q^2_{a,c} =0
\end{equation}
Then
$$u(x,t) = Q_{ct,c}(x) = cQ(c(x-ct))$$
solves (BO) and we call it the \emph{single soliton} solution to distinguish it from the exact multi-soliton solutions \cite{Case} arising from the completely integrable structure.   The (BO) soliton is only decaying at infinity at power rate unlike for the famous Korteweg-de Vries (KdV) model, where the soliton enjoys exponential decay.

From the physical standpoint, it is of interest to consider the effects of perturbations of the equation on the dynamics of solitons.   For example, Matsuno \cite{Mat01, Mat02} derived a higher-order BO equation
$$\partial_t u + 4uu_x + Hu_{xx} = \epsilon f(u,u_x,u_{xx},u_{xxx})$$
where the right side is a specific nonlinear function, and carried out a heuristic multiscale analysis of the effect of this perturbation on the dynamics of multisolitons.   This equation considered in \cite{Mat01, Mat02} describes the unidirectional motion of interfacial waves in a two-layer fluid system, and provides motivation to consider the mathematical theory of Hamiltonian perturbations of (BO), for which we consider the following model case.  
$$
\text{(pBO)} \qquad
\partial_t u =  \partial_x (- H\partial_x u + V u - \frac12 u^2) 
$$
with \emph{slowly varying} potential
\begin{equation}
\label{E:potential}
V(x) = W(hx) \,, \qquad W\in C_c^\infty(\mathbb{R}) \text{ and } 0<h \ll 1
\end{equation}

The well-posedness of (pBO) in $H^1$ can be proved by adapting the gauge-transform method of Tao \cite{Tao}.   The Hamiltonian has been perturbed to
$$E(u) = E_0(u) + \frac12 \int Vu^2$$
(pBO) is of the form $\partial_t u = JE'(u)$, where $J = \partial_x$.    

Our main result (Theorem \ref{T:main} below) is a strengthening of Theorem 1.1 in Zhang \cite{Z}, on the dynamical behavior of near soliton solutions to (pBO).   For the statement, we will need the \emph{reference trajectory}, which is the solution $(\bar A(s), \bar C(s))$ to 
\begin{equation}
\label{E:ref-traj}
\left\{
\begin{aligned}
&\dot {\bar C} =  \bar CW'(\bar A) \\
&\dot {\bar A} =  \bar C- W( \bar A) 
\end{aligned}
\right.
\end{equation}
with initial condition $(\bar A(0),\bar C(0)) = (0,1)$, which is an $h$-independent system.   Using this reference trajectory, we can define $S_0>0$ to be the first time $s>0$ such that  $\bar C(s)=\frac12$ or $\bar C(s)=2$, or take $S_0=+\infty$ if $C(s)$ never reaches either $\frac12$ or $2$.    Thus, for all $0\leq s < S_0$, we have
$$\frac12 \leq \bar C(s) \leq 2$$
Let
\begin{equation}
\label{E:bar-convert}
\bar a(t)=h^{-1}\bar A(ht)\,, \qquad \bar c(t) = \bar C(ht)
\end{equation}
so that
$$\left\{\begin{aligned}
& \dot{\bar c} = h\bar c W'(h \bar a) \\
& \dot{\bar a} = \bar c - W(h \bar a)
\end{aligned} \right.
$$
with initial condition $(\bar a(0), \bar c(0))= (0,1)$.  Now let us introduce the \emph{exact trajectory}, which is the solution $(\hat A(s), \hat C(s))$ to
\begin{equation}
\label{E:exact-traj}
\left\{
\begin{aligned}
&\dot {\hat C} =  \hat CW'(\hat A) + \frac12 \hat C^{-1}h^2W'''(\hat A)\\
&\dot {\hat A} =  \hat C- W( \hat A)  +\frac12 \hat C^{-2} h^2 W''(\hat A)
\end{aligned}
\right.
\end{equation}
with initial condition $(\hat A(0),\hat C(0)) = (0,1)$, which is an $h$-dependent trajectory.  With a conversion analogous to \eqref{E:bar-convert}
\begin{equation}
\label{E:hat-convert}
\hat a(t)=h^{-1}\hat A(ht)\,, \qquad \hat c(t) = \hat C(ht)
\end{equation}
we have that $(\hat a, \hat c)$ solves
\begin{equation}
\label{E:exact-traj-t}
\left\{
\begin{aligned}
& \dot{\hat c} = h\hat c W'(h \hat a) + \frac12 \hat c^{-1}h^3W'''(h\hat a)\\
& \dot{\hat a} = \hat c - W(h \hat a) +\frac12 \hat c^{-2} h^2 W''(h\hat a)
\end{aligned} \right.
\end{equation}
with initial condition $(\hat a(0), \hat c(0))= (0,1)$.   

By elementary ODE perturbation (Lemma \ref{L:gronwall}), 
$$|\hat A - \bar A| \lesssim h^2e^{\mu s}\,, \qquad |\hat C - \bar C| \lesssim h^2 e^{\mu s}$$ 
for some $\mu>0$, which under the transformations \eqref{E:bar-convert}, \eqref{E:hat-convert} converts to 
\begin{equation}
\label{E:hatbar-compare}
|\hat a - \bar a| \lesssim he^{\mu ht}\,, \qquad |\hat c - \bar c| \lesssim h^2 e^{\mu ht}
\end{equation}

Our main theorem is the following.

\begin{theorem}[exact effective dynamics for (pBO)]
\label{T:main}
Given a potential $W\in C_c^\infty(\mathbb{R})$ (as in \eqref{E:potential}), there exists $\kappa \geq 1$, $\mu>0$, and $0<h_0\ll 1$ such that the following holds.  Let $0< h \leq h_0$ and suppose the initial data $u_0\in H_x^1$ satisfies
$$\| u_0(x) - Q_{0,1}(x) \|_{H_x^{1/2}} \leq h^{3/2}$$
Letting $(\hat a, \hat c)$ be the exact trajectory \eqref{E:exact-traj-t}, then $u$ solving (pBO) with initial condition $u_0$ satisfies
\begin{equation}
\label{E:control}
\| u(x,t) - Q_{\hat a(t),\hat c(t)}(x) \|_{H_x^{1/2}} \leq \kappa h^{3/2}e^{\mu ht}
\end{equation}
for $0\leq t \leq T_0=h^{-1}\min( \frac14\mu^{-1} \ln h^{-1}, S_0)$.  
\end{theorem}

In Zhang \cite{Z}, this result is obtained without specific equations for $\hat a$, $\hat c$, only the comparison estimate \eqref{E:hatbar-compare}.  For this reason, we refer to the result as providing \emph{exact dynamics} -- the precision of the parameter dynamics meets (in fact exceeds) the bound on the remainder \eqref{E:control}.  Notice that the $|\hat a-\bar a|$ estimate in \eqref{E:hatbar-compare} is not sufficiently strong to replace $(\hat a, \hat c)$ in \eqref{E:control} by $(\bar a, \bar c)$.   If this exchange were made, the upper bound in \eqref{E:control} would need to be replaced with $h e^{\mu ht}$.  Although in Theorem \ref{T:main}, the starting point is taken to be $(a(0),c(0))=(0,1)$, by scaling and translating the equation and potential, this result covers the case of general initial starting point $(a(0),c(0))$, with $a(0)\in \mathbb{R}$ and $c(0)>0$.  An overview of the literature on results on the dynamics of solitons in a slowly varying potential are given in the introduction of Zhang \cite{Z}.

The proof of Theorem \ref{T:main} relies upon an adaptation of a \emph{local virial} estimate in Kenig \& Martel \cite{KM}.  Let 
$$\mathcal{L} =-H\partial_y +1 -Q$$ be the linearized operator and we consider $v$ solving
\begin{equation} \label{E:D1C}
\partial_t v = \mathbb{P} v + \partial_y \mathcal{L} v + \partial_y  f 
\end{equation}
with 
\begin{equation}  \label{def:P}
\mathbb{P} v : = \frac{\la v , \mathcal{L} \partial_y^2 Q \ra }{\| \partial_y Q \|_{L^2}^2} \partial_y Q . 
\end{equation}
where $f=f(y,t)$ is a forcing function.  We will assume that $v$ satisfies the \emph{nonsymplectic} orthogonality conditions
\begin{equation}
\label{E:D1E}
\la v, Q \ra =0 \,, \qquad \la v, Q' \ra =0
\end{equation}

For any $\gamma>0$ and $y_0\in \mathbb{R}$, let
\begin{equation}
\label{E:g-def}
g_{\gamma,y_0}(y) = \gamma^{-1} \arctan( \gamma(y-y_0))
\end{equation}
so that
$$g_{\gamma,y_0}'(y) = \frac{1}{1+\gamma^2(y-y_0)^2} = \la \gamma(y-y_0) \ra^{-2}$$  
is a spatial localization factor with scale $\gamma>0$ and center $y_0$.

Define the operator 
$$\mathcal{D}_\gamma := 1+\gamma\partial_y $$ 
and let $\mathcal{D}_\gamma^{-1}$ be the Fourier multiplier operator with symbol $(1+i\gamma \xi)^{-1}$.    The operator $\mathcal{D}_\gamma^{-1}$ will be used in the analysis to conjugate our equation to a dual equation.  We remark that if $f$ is a real-valued function, then $\mathcal{D}_\gamma^{-1} f$ is also real-valued.  Furthermore, we remark that $\mathcal{D}_\gamma^{-1}\mathcal{L}$ is a pseudodifferential operator of order $0$.   

\begin{theorem}[local virial estimate for linearized Benjamin-Ono]
\label{T:local-virial}
There exists $0< \gamma_0 \ll 1$ such that for all $0<\gamma\leq \gamma_0$, for any time length $T>0$,  for any spatial center $y_0\in \mathbb{R}$, and for any solution $v$ to \eqref{E:D1C} satisfying \eqref{E:D1E}, we have 
\begin{equation}
\label{E:D1D}
\|  \la D_y \ra^{1/2} ((g_{\gamma,y_0}')^{1/2} v ) \|_{L_{[0,T]}^2L_y^2}^2 \lesssim_\gamma \|v\|_{L_{[0,T]}^\infty L_y^2}^2 + G_\gamma (f,v)
\end{equation}
where 
\begin{equation}
\label{E:G}
G_\gamma (f, v) = \int_0^T \int g_{\gamma,y_0} \, v \, \partial_y f \, dy dt + \int_0^T \int_y g_{\gamma,0} (\mathcal{D}_\gamma^{-1} \mathcal{L} v) ( \mathcal{D}_\gamma^{-1} \mathcal{L} \partial_y f) \, dy \, dt 
\end{equation}
Importantly,  the implicit constant in \eqref{E:D1D} is independent of $T$, $y_0$.
\end{theorem}

Notice that this is a local smoothing type estimate, where $\frac12$ derivative is gained after localizing in space and averaging in time.  Such an estimate can be proved for the linearization around $0$  using the Fourier representation of the propagator and Plancherel's theorem.

A weaker version of this estimate (estimate (3.7) in \cite{KM}) is proved in Kenig \& Martel \cite{KM} as part of their Theorem 3 or `linear Liouville' result (starting on p. 923).  We have isolated and refined this estimate, and to do so we use still essentially follow their method of passing to a dual problem and implementing a positive commutator argument.  But to obtain our version of the estimate, we use a slightly different transformation and corresponding dual problem, prove and employ additional commutator estimates, avoid using a uniform spatial decay hypothesis (as in (3.6) of \cite{KM}), and also invoke an extra spectral estimate.  

We will prove Theorem \ref{T:local-virial} in \S \ref{S:local-virial}, after giving the needed spectral estimates in \S\ref{S:spectral} and commutator estimates in \S \ref{S:commutator}.  In more detail, the paper begins as follows:  in \S\ref{S:notation}, we give an overview of notational conventions used in the paper (definitions of Fourier and Hilbert transforms), together with basic properties of the soliton profile $Q$ and the associated linearized operator $\mathcal{L}$.  In \S\ref{S:spectral}, spectral properties of $\mathcal{L}$ are stated and referenced and key coercivity (or positivity) properties of $\mathcal{L}^2$ and $\mathcal{L}$ are proved.  In \S\ref{S:commutator} commutator lemmas are stated and proved that will be employed in \S\ref{S:local-virial}, which features the proof of Theorem \ref{T:local-virial}. 

The proof of Theorem \ref{T:local-virial} proceeds as follows.  Setting $\psi = \mathcal{D}_\gamma^{-1} \mathcal{L} v$, for $\gamma>0$ chosen sufficiently small, the problem is reformulated in terms of $\psi$.  The equation satisfied by $\psi$ is \eqref{E:veqn}, roughly of the form
$$\partial_t \psi = \mathcal{L}\partial_y \psi + \mathcal{D}_\gamma^{-1}\mathcal{L}\partial_y f$$
and $\psi$ satisfies orthogonality conditions \eqref{E:v-orth}, approximately of the form
$$\la \psi, Q' \ra =0 \, \qquad \la \psi, (yQ)' \ra =0$$
The parameter $\gamma>0$ is taken sufficiently small so that the error terms in both of these approximations  become subordinate.  A new commutator estimate, Lemma \ref{L:com6}, proved via a weighted Schur test and a spectral estimate, Proposition \ref{P:L2-spec-est}, shows that once a local virial estimate is obtained for $\psi$, it can be recovered for $v$.  Thus, the task has been reduced to proving the local virial estimate for $\psi$.  To see what this entails, first consider the (nonlocal) virial identity obtained by computing $\partial_t \int y \psi^2 \, dy$, which upon substituting the equation for $\psi$ and reducing via integration by parts, yields a dominate term of the form $\la \tilde{\mathcal{L}} \psi, \psi \ra$, where $\tilde{\mathcal{L}}$ is given by \eqref{E:tildeLdef},
$$\tilde{\mathcal{L}} \defeq -2H\partial_y + 1 - yQ'-Q$$
In Proposition \ref{P:tildeLspec}, this operator is shown to be positive on the codimension two subspace given by the orthogonality conditions for $\psi$.  This is ultimately the purpose of passing from $v$ to the dual problem for $\psi$, as the operator that results from the computation of $\partial_t \int y v^2 \, dy$ is not known to be positive on the codimension two subspace described by the orthogonality conditions for $v$.  

The local virial estimate in Theorem \ref{T:local-virial} is applied to the (pBO) equation in \S\ref{S:pBO} in the following context.  Let 
\begin{equation}
\label{E:zeta-def01}
\zeta(x,t) = u(x,t) - Q_{\am(t),\cm(t)}(x)
\end{equation}
where the parameters $\am(t)$, $\cm(t)$ are selected to achieve orthogonality conditions \eqref{E:D1E} for $v(y,t) = \zeta(y+\am(t),t)$ (that is, $x=y+\am(t)$).    Theorem \ref{T:local-virial}, together with estimates for parameter trajectories and energy estimates, yields Proposition \ref{P:nonsymp-estimates}, which in particular provides the estimate
\begin{equation}
\label{E:v-est01}
\|v \|_{L_{[0,T]}^\infty H_y^{1/2}} + \sup_n \| v \|_{L_{t\in [0,T]}^2L_{y\in(n,n+1)}^2} \leq \kappa h^{3/2} e^{\mu hT}
\end{equation}
The parameter trajectory estimates and energy estimate appearing in \S\ref{S:pBO} are only a slight modification of those in Zhang \cite{Z}, although they have been reproduced here to make the paper self contained.  The main new ingredient beyond the material in Zhang \cite{Z} is the use of the local virial estimate, Theorem \ref{T:local-virial}.

The exact dynamics reported in Theorem \ref{T:main} are obtained in \S \ref{S:exact} as a consequence of Proposition \ref{P:nonsymp-estimates} using a different decomposition of $u(x,t)$.  Let 
\begin{equation}
\label{E:eta-def01}
\eta(x,t) = u(x,t) - Q_{a(t),c(t)}
\end{equation}
where the parameters $a(t)$, $c(t)$ are selected to achieve the \emph{symplectic} orthogonality conditions 
\begin{equation}
\label{E:D1F}
\la w, Q \ra =0 \,, \qquad \la w, yQ \ra=0
\end{equation}
for $w(y,t) = \eta(x+a(t),t)$ (so here $x=y+a(t)$).  In \S \ref{S:exact}, it is detailed how to convert the estimate \eqref{E:v-est01} to a similar estimate for $w$
\begin{equation} 
\label{E:w-est01}
\begin{aligned}
&\|w\|_{L_{[0,T]}^\infty H_y^{1/2}}  \leq \kappa h^{3/2} e^{\mu hT} \\
&\sup_n \| w\|_{L_{[0,T]}^2L_{y\in(n,n+1)}^2} \leq \kappa h^{3/2} (\ln h^{-1}) e^{\mu hT}
\end{aligned}
\end{equation}
The estimates \eqref{E:w-est01}, together with parameter trajectory estimates for $a(t)$, $c(t)$ analogous to those in \S\ref{S:pBO} for $\am(t)$, $\cm(t)$,  and similar to those in \cite{Z}, yield Theorem \ref{T:main}.  The reason that the advertised exact dynamics are now achievable, but were not in \cite{Z}, is that the local-virial estimate for $w$ (the second estimate of \eqref{E:w-est01}) is now available to control the terms in the ODE comparison estimate (Lemma \ref{L:gronwall}), which effectively achieves a gain of a power of $h$ in comparison to merely using the energy bound for $w$ (the first estimate of \eqref{E:w-est01}).

The method of deriving and applying a local virial estimate for the linearized equation in the setting of nonlinear dispersive PDE to achieve rigidity results on soliton dynamics was introduced about 20 years ago as a ``nonlinear Liouville theorem'' in the case of the $L^2$-critical gKdV by Martel \& Merle \cite{MM-liouville}.  The method of converting from $v$ to a dual problem for $\psi$ was introduced by Martel \cite{Martel-dual} in the gKdV context, where the transformation $\psi= \mathcal{L}v$ is used.  The addition of the regularization operator was used by Kenig \& Martel \cite{KM} in their treatment of asymptotic stability for the Benjamin-Ono equation. They used  $\psi= (1-\delta \Delta)^{-1} \mathcal{L} v$ while we use $\psi=\mathcal{D}_\gamma^{-1}\mathcal{L}v$, since the explicit kernel of the operator $\mathcal{D}_\gamma^{-1}$ facilitates the proof of some commutator estimates that we use to convert the estimate for $\psi$ back to an estimate for $v$.    This method of using a regularized transformation was also applied by Farah, Holmer, Roudenko, \& Yang \cite{FHRY2018} in the study of blow-up of the $L^2$ critical 2D Zakharov-Kuznetsov (ZK) equation, and a different method of handling regularity issues was recently developed in the context of asymptotic stability for solitary waves of the 3D ZK equation by the same authors in \cite{FHRY2020}.

We conclude this paper by showing in \S\ref{S:linear-Liouville} that the linear Liouville property (Theorem 3 in \cite[\S 3]{KM}) that appeared in Kenig \& Martel's proof of asymptotic stability for single soliton solutions to (BO) can be proved using the version of the local virial inequality in Theorem \ref{T:local-virial} instead of the one appearing in \cite{KM}.  

\begin{theorem}[linear Liouville property for linearized Benjamin-Ono]
\label{T:linear-Liouville}
Suppose that $v$ solves \eqref{E:D1C} with $f \equiv 0$ and $v$ satisfies the orthogonality conditions \eqref{E:D1E}.  Moreover, we assume that $\|v \|_{L_{t\in \mathbb{R}}^\infty L_y^2} < \infty$ and $v$ satisfies the following uniform-in-time spatial localization property:  there exists a constant $C>0$ such that for each $y_0\geq 1$ and each $t\in \mathbb{R}$,
\begin{equation}
\label{E:uniformspatial}
\int_{|y|\geq y_0} |v(y,t)|^2 \, dy \leq \frac{C}{y_0}
\end{equation}
Then $v\equiv 0$.
\end{theorem}

We will prove Theorem \ref{T:linear-Liouville} in \S \ref{S:linear-Liouville} using Theorem \ref{T:local-virial} and a monotonicity lemma from \cite{KM}.

\section{Notation and basic computations}
\label{S:notation}

We fix a convention for the Fourier transform (in dimension $1$) and its inverse:
$$\hat f(\xi) = \int e^{-ix\xi} f(x) \, dx \,, \qquad \check g(x) = \frac{1}{2\pi} \int e^{ix\xi} g(\xi)\, d\xi$$
and the Hilbert transform:
$$Hf(x) = -\frac{1}{\pi} \pv \int_{-\infty}^{+\infty} \frac{f(y)}{x-y} \, dy = -\frac{1}{\pi} \pv \frac{1}{x} * f$$
Hence
$$\widehat{Hf}(\xi) = i (\sgn \xi) \hat f(\xi)$$
The fractional derivative operator $D^s$ is defined as $\widehat{D^sf}(\xi) = |\xi|^s \hat f(\xi)$, and thus $-H\partial_x = D$.   

The soliton profile $Q(x)$ is defined explicitly by the formula
\begin{equation}
\label{E:D2}
Q(x) = \frac{4}{1+x^2}
\end{equation}
We have the partial fraction decomposition
$$\frac{1}{1+y^2} \frac{1}{x-y} = - \frac{1}{1+x^2}\frac{1}{y-x} + \frac{x}{1+x^2} \frac{1}{1+y^2} + \frac{1}{1+x^2} \frac{y}{1+y^2}$$
and hence (since first and third term integrate to zero)
\begin{equation}
\label{E:D4}HQ = -\frac{4}{\pi} \frac{x}{1+x^2} \int \frac{dy}{1+y^2} = \frac{-4x}{1+x^2}=-xQ
\end{equation}
From this, and the easily confirmed identity (direct computation)
\begin{equation}
\label{E:D4p}
xQ' = \frac12Q^2 -2Q
\end{equation}
we obtain that $Q$ solves the soliton profile equation
\begin{equation}
\label{E:D3}
Q -HQ' - \tfrac12 Q^2 = 0
\end{equation}
Amick \& Toland \cite{AT} showed that $Q(x)$ is the unique solution to \eqref{E:D3} such that $Q(x) \to 0$ as $|x|\to \infty$.   For soliton dynamics problems, we introduce the modulation parameters of translation $a$ and scale $c$ and define
$$Q_{a, c} = cQ(c(x-a))$$
so that $Q=Q_{0,1}$.  Note that from \eqref{E:D3}, $Q_{a, c}$ solves the equation
\begin{equation}
\label{E:D3p}
cQ_{a, c} -HQ'_{a, c} - \tfrac12 Q_{a ,c}^2 = 0
\end{equation}

The operator corresponding to linearization of \eqref{E:D3p} at $c=1$, $a=0$ is 
\begin{equation}
\label{E:Ldef}
\mathcal{L} \defeq I - H\partial_x - Q
\end{equation}
We also define a rescaled version of $\mathcal{L}$:
$$\mathcal{L}_c \defeq c - H\partial_x - c Q (cx) , $$ 
whose properties are basically the same as for $\mathcal{L}$. 

By differentiating \eqref{E:D3p} with respect to $a$, and evaluating at $c=1$, $a=0$, we obtain
\begin{equation}
\label{E:Da} 
\mathcal{L}Q' = 0
\end{equation}
By differentiating \eqref{E:D3p} with respect to $c$, and evaluating at $c=1$, $a=0$, we obtain
\begin{equation}
\label{E:Dc}
\mathcal{L} (xQ)' = -Q
\end{equation}
From \eqref{E:D3}, we can deduce
\begin{equation}
\label{E:app11}
\mathcal{L}Q = - \tfrac12 Q^2
\end{equation}
By \eqref{E:D4p}, it follows that \eqref{E:app11} converts to
\begin{equation}
\label{E:app11p}
\mathcal{L}Q = -(xQ)'-Q
\end{equation}

We can use \eqref{E:Dc} and \eqref{E:app11p} to locate two important eigenfunctions and eigenvalues of $\mathcal{L}$, although the complete spectral picture is provided by Proposition \ref{P:Lspec} below.   From \eqref{E:Dc} and \eqref{E:app11p}, for any constants $\alpha$ and $\beta$,
$$\mathcal{L}( \alpha Q + \beta (xQ)') = - (\alpha+\beta) Q - \alpha (xQ)'$$
To find eigenfunctions, we find $\alpha$ and $\beta$ such that
$$\frac{\alpha}{\alpha+\beta} = \frac{\beta}{\alpha} \implies \frac54 \alpha^2 = (\beta + \frac\alpha{2})^2 \implies \beta = \frac{\pm\sqrt{5} -1}{2}\alpha$$
Substituting, we obtain
$$\mathcal{L}e_\pm = \lambda_\pm e_\pm$$
where
\begin{equation}
\label{E:epm}
e_\pm = Q + \frac{ \mp \sqrt 5 - 1}{2} (xQ)' \,, \qquad \lambda_\pm = \frac{\pm \sqrt 5 -1}{2}
\end{equation}
Note that both $Q$ and $(xQ)'$ are even functions, so that $e_\pm$ are even as well.

\section{Spectral estimates}
\label{S:spectral}

For the key operator $\mathcal{L}$ defined in \eqref{E:Ldef}, there is a full description of the spectrum and spectral measure, stated as Proposition \ref{P:Lspec} below, taken from the appendix of \cite{Bennett}.  In this section, we state and prove an ``angle lemma'' (Lemma \ref{L:angle}) and give, as an application, spectral estimates (Proposition \ref{P:L2-spec-est} and Proposition \ref{P:tildeLspec}) needed for the proof of Theorem \ref{T:local-virial}.

\begin{proposition}[from Appendix of \cite{Bennett}]
\label{P:Lspec}
The operator $\mathcal{L}$ has exactly four eigenvalues
$$
\lambda_1= 1,  \quad \lambda_+ = \frac{-1+\sqrt{5}}{2}\approx 0.62, \quad \lambda_0 =0, \quad \lambda_- =\frac{-1-\sqrt{5}}{2}\approx -1.62
$$
and a continuous spectrum $[1, +\infty)$.
Moreover, the corresponding eigenspaces are one-dimensional, the (non-normalized) eigenfunction for $\lambda_0 =0$ is $ Q'$, and the (non-normalized) eigenfunctions for $\lambda_\pm$ are $e_\pm$, respectively, given by \eqref{E:epm}.  Note that $e_\pm$ are even functions, and $Q'$ is an odd function.
\end{proposition}

\begin{lemma}[angle lemma]
\label{L:angle}
Suppose that $\mathcal{L}$ is a self-adjoint operator with eigenvalue $\mu_1$ and corresponding eigenspace spanned by $e_1$ with $\|e_1\|_{L^2}=1$.  Let $P_1f = \la f,e_1\ra e_1$ be the corresponding orthogonal projection.   Assume that $(I-P_1)\mathcal{L}$ has spectrum bounded below by $\mu_\perp$, with $\mu_\perp>\mu_1$.  Suppose that $f$ is some other function such that $\|f\|_{L^2}=1$ and $0 \leq \beta \leq \pi$ is defined by $\cos \beta = \la f, e_1\ra$.   Then if $v$ is  satisfies $\la v, f \ra =0$, we have
$$\la \mathcal{L}v,v \ra \geq (\mu_\perp - (\mu_\perp - \mu_1)\sin^2\beta)\|v\|_{L^2}^2$$
\end{lemma}
\begin{proof}
It suffices to assume that $\|v\|_{L^2}=1$.  Decompose $v$ and $f$ into their orthogonal projection onto $e_1$ and its orthocomplement:
$$v =  (\cos \alpha) e_1 +  v_\perp \,, \qquad \|v_\perp\|_{L^2}=\sin \alpha$$
$$f = (\cos \beta) e_1 + f_\perp \,, \qquad \|f_\perp\|_{L^2} = \sin \beta$$
for $0\leq \alpha, \beta \leq \pi$.  
Then 
$$ 0 = \la v,f \ra = \cos\alpha \cos \beta + \la v_\perp, f_\perp \ra$$
from which it follows that
$$|\cos \alpha \cos \beta| = |\la v_\perp, f_\perp \ra| \leq \|v_\perp\|_{L^2}  \|f_\perp\|_{L^2} \leq \sin \alpha \sin \beta  $$
Taking the square yields
$$\cos^2 \alpha (1- \sin^2 \beta )  \leq  (1-  \cos^2 \alpha ) \sin^2 \beta  $$
and from this it follows that $|\cos \alpha| \leq \sin \beta$.  Now
\begin{align*}
\la \mathcal{L}v,v\ra &= \mu_1 \cos^2\alpha + \la \mathcal{L}v_\perp,v_\perp \ra \\
&\geq \mu_1 \cos^2\alpha + \mu_\perp\sin^2 \alpha \\
&= \mu_\perp - (\mu_\perp-\mu_1) \cos^2\alpha \\
&\geq  \mu_\perp - (\mu_\perp-\mu_1) \sin^2\beta
\end{align*}
\end{proof}

We will prove spectral estimates for $v$ satisfying the orthogonality conditions \eqref{E:D1E}.  For the proof of Theorem \ref{T:local-virial} given in \S \ref{S:local-virial} (in particular for the proof of Proposition \ref{P:local-virial2}, a component of the proof of Theorem \ref{T:local-virial}), we will take $\psi = \mathcal{D}_\gamma^{-1}\mathcal{L}v$.  Now, if $z= \mathcal{L}v$, then
$$\la z, Q' \ra = \la \mathcal{L}v, Q' \ra = \la v, \mathcal{L}Q' \ra = 0$$
since $\mathcal{L}Q'=0$ (by \eqref{E:Da}).  Moreover, the orthogonality condition $\la v, Q \ra=0$ (part of \eqref{E:D1E}) implies
$$\la z, (yQ)' \ra = \la \mathcal{L}v, (yQ)' \ra = \la v, \mathcal{L}(yQ)' \ra = - \la v, Q \ra =0$$
where we have used $\mathcal{L}(yQ)' = -Q$ (which is \eqref{E:Dc}).    Thus, when \eqref{E:D1E} is in place for $v$, and $z=\mathcal{L}v$, then we have
\begin{equation}
\label{E:z-orth}
\la z, Q' \ra =0 \,, \qquad \la z, (yQ)' \ra =0
\end{equation}
As mentioned, for the proof of Proposition \ref{P:local-virial2}, we will take $\psi= \mathcal{D}_\gamma^{-1} \mathcal{L}v$, for $\gamma>0$ small.  Since $\psi = \mathcal{D}_\gamma^{-1}z$, it follows that (using that $(\mathcal{D}_\gamma^*)^{-1} - I = \gamma (\mathcal{D}_\gamma^*)^{-1} \partial_x$) 
\begin{align*}
\la \psi, Q' \ra &= \la \mathcal{D}_\gamma^{-1} z, Q' \ra = \la z, (\mathcal{D}_\gamma^*)^{-1}Q' \ra \\
&= \la z, [(\mathcal{D}^*_\gamma)^{-1}-I] Q' \ra =  \gamma \la z, (\mathcal{D}_\gamma^*)^{-1}Q'' \ra = \gamma \la \psi, Q'' \ra
\end{align*}
and hence $\la \psi, Q' - \gamma Q'' \ra=0$.    Similarly, we have
\begin{align*}
\la \psi, (yQ)' \ra &= \la z, (D_\gamma^*)^{-1} (yQ)' \ra = \la z, [(D_\gamma^*)^{-1}-1] (yQ)' \ra \\
&= \gamma \la z, (D_\gamma^*)^{-1} (yQ)'' \ra = \gamma \la \psi, (yQ)'' \ra
\end{align*}
and hence $\la \psi, (yQ)' - \gamma (yQ)'' \ra=0$.  Collecting these results, we can assert that the orthogonality conditions for $\psi$ are
\begin{equation}
\label{E:v-orth}
\la \psi, Q' - \gamma Q''\ra =0 \,, \qquad \la \psi, (yQ)' - \gamma(yQ)'' \ra =0
\end{equation}
We will also deal with the perturbation $v_\gamma = \mathcal{D}_\gamma^{-1} \la \gamma y \ra^{-1} v$.  Similarly to the above calculations, we can deduce that if $v$ satisfies $\la v, Q'\ra =0$, then $v_\gamma$ satisfies the perturbed orthogonality 
\begin{equation}
\label{E:w-gamma-orth}
\la v_\gamma, Q' + \gamma q_\gamma\ra=0 \,, \quad \text{where} \quad q_\gamma = -Q'' + \mathcal{D}_\gamma^* \left[ \frac{\gamma y^2 Q'}{1+\la \gamma y \ra}  \right]
\end{equation}
We note that $q_\gamma$ has $L^2$ norm bounded uniformly in $\gamma$.  

\begin{proposition}
\label{P:L2-spec-est} \quad
\begin{enumerate}
\item 
For $v$ satisfying the orthogonality condition $\la v, Q' \ra =0$, we have
$$\la \mathcal{L}^2 v, v \ra \gtrsim \|v \|^2_{H^1}$$
\item For $v$ as above in (1), if we denote $v_\gamma = \mathcal{D}^{-1}_\gamma \langle \gamma y \rangle^{-1} v$, then $v_\gamma$ satisfies the orthogonality condition \eqref{E:w-gamma-orth}, and for $\gamma>0$ sufficiently small,
$$\la \mathcal{L}^2 v_\gamma, v_\gamma \ra \gtrsim \|v_\gamma \|^2_{H^1} $$
with constant independent of $\gamma$.
\end{enumerate}
\end{proposition}
\begin{proof}
For the proof of item (1), we note the spectrum of $\mathcal{L}^2$ is the square of the spectrum of $\mathcal{L}$, and thus it consists of two simple eigenvalues $0$ (with eigenfunction $Q'$) and $\lambda_+^2 \approx 0.38$ (with eigenfunction $e_+$) and essential spectrum in $[1,+\infty)$ (note that $\lambda_-^2>1$).   By the orthogonality condition $\la v, Q'\ra=0$, it is immediate that $\la \mathcal{L}^2w,w \ra \geq \lambda_+^2 \|w\|_{L^2}^2$.  

For the proof of item (2), we use that $v_\gamma$ satisfies orthogonality condition \eqref{E:w-gamma-orth}, and apply the angle lemma with $\mu_1=0$, $\mu_\perp = \lambda_+^2 \approx 0.38$, 
$$e_1 =   \frac{Q'}{\|Q'\|_{L^2}} \,, \qquad f = \frac{Q' + \gamma g_\gamma }{\|Q'+ \gamma g_\gamma\|_{L^2}}$$
and
$$\cos \beta = \la f,e_1 \ra = \frac{\la Q', Q'+\gamma g_\gamma \ra}{\|Q'\|_{L^2} \| Q' + \gamma g_\gamma\|_{L^2}} = 1 - O(\gamma)\neq 0$$ for $\gamma$ sufficiently small, and thus $\sin^2 \beta \neq 1$, so that Lemma \ref{L:angle} furnishes a positive lower bound $\la \mathcal{L}^2 v_\gamma, v_\gamma \ra \gtrsim \|v_\gamma \|_{L^2}^2$.   

In each case, the $H^1$ lower bound (as opposed to $L^2$) follows by standard elliptic regularity calculations. 
\end{proof}

The following lemma will be needed in the proof of Proposition \ref{P:tildeLspec} below.  Recall that if $z$ satisfies orthogonality conditions \eqref{E:z-orth} and $\psi= \mathcal{D}_\gamma^{-1} \mathcal{L}v$, then $\psi$ satisfies orthogonality conditions \eqref{E:v-orth}.

\begin{lemma}
\label{L:L-est}
\quad
\begin{enumerate}
\item For $z$ satisfying the orthogonality conditions \eqref{E:z-orth}, $\la \mathcal{L}z, z \ra \geq 0$. 
\item For $\psi$ satisfying the orthogonality conditions \eqref{E:v-orth}, $\la \mathcal{L}\psi,\psi \ra \gtrsim -\gamma \|\psi\|_{L^2}^2$.
\end{enumerate}
\end{lemma}
We note that the proof does not give strict positivity, only the claimed non-negativity.

\begin{proof}
We begin with item (1).
Decompose $z=z_e+z_o$ into even and odd components, respectively.  Since $\mathcal{L}$ preserves parity, 
$$\la \mathcal{L}z, z \ra = \la \mathcal{L}(z_e+z_o), z_e+z_o \ra = \la \mathcal{L}z_e,z_e \ra + \la \mathcal{L} z_o, z_o \ra$$
and it suffices to show that $\la \mathcal{L}z_e, z_e \ra \geq 0$ and $\la \mathcal{L}z_o,z_o \ra \geq 0$.  

First, consider $\mathcal{L}$ restricted to the odd subspace, which has eigenvalues $\lambda_0=0$ (corresponding to eigenfunction $Q'$) and $\lambda_1=1$ and continuous spectrum $[1,+\infty)$.   Since $\la z_o, Q' \ra =0$, it follows that $\la \mathcal{L}z_o,z_o \ra  \geq \|z_0\|_{L^2}^2 \geq 0$. 

Next, consider $\mathcal{L}$ restricted to the even subspace, which has eigenvalues $\lambda_- = - \frac{\sqrt{5}+1}{2}$ and $\lambda_+ = \frac{\sqrt{5}-1}{2}$, and continuous spectrum $[1,+\infty)$.  Apply Lemma \ref{L:angle} with 
$$\mu_1 = \lambda_- = -\frac{\sqrt{5}+1}{2} \,, \qquad \mu_\perp = \lambda_+ =\frac{\sqrt{5}-1}{2}$$
$$e_1 = \frac{ e_-}{ \|e_-\|_{L^2}} = \frac{ Q + \frac{\sqrt{5}-1}{2}(yQ)'}{ \| Q + \frac{\sqrt{5}-1}{2}(yQ)' \|_{L^2}} \,, \qquad f = \frac{ (yQ)' }{ \| (yQ)' \|_{L^2}} $$
and
$$\cos \beta = \la f, e_1 \ra = \frac{ \la (yQ)', Q + \frac{\sqrt 5 -1}{2} (yQ)' \ra}{ \|(yQ)' \|_{L^2} \| Q + \frac{\sqrt{5}-1}{2} (yQ)' \|_{L^2}}$$
From the explicit formula for $Q(y)$, 
$$\| Q \|_{L^2}^2 = 8\pi\,, \qquad \| (yQ)' \|_{L^2}^2 = 4\pi  \,, \qquad   \| Q + \frac{\sqrt{5}-1}{2} (yQ)' \|_{L^2}^2 = 2(5+\sqrt{5}) \pi $$ 
and hence
$$\la (yQ)', Q \ra = -\la yQ, Q' \ra = \frac12 \|Q\|_{L^2}^2 = 4\pi$$
Substituting above and simplifying, we obtain
$$\cos^2 \beta = \frac12 + \frac{\sqrt{5}}{10}$$
from which it follows that
$$\mu_\perp - (\mu_\perp-\mu_1) \sin^2\beta = 0$$
Hence Lemma \ref{L:angle} yields that $\la\mathcal{L}z_e,z_e \ra \geq 0$.

Item (2) in the lemma statement is addressed similarly with a decomposition $\psi=\psi_o+\psi_e$, although in applying Lemma \ref{L:angle} for $\psi_e$, $f$ is replaced by
$$ f = \frac{ (yQ)' - \gamma(yQ)'' }{ \| (yQ)' - \gamma(yQ)''\|_{L^2}}$$
The case of $z$ corresponds to $\gamma=0$, and in that case, we found $\mu_\perp - (\mu_\perp-\mu_1) \sin^2\beta = 0$.  Thus for $\gamma>0$, we find
$$\mu_\perp - (\mu_\perp-\mu_1) \sin^2\beta \gtrsim  -\gamma$$
In order to address $\psi_o$, we also need to apply Lemma \ref{L:angle}, although in this case we use $\mu_1 = 0$, $\mu_\perp =1$, 
$$e_1 = \frac{Q'}{\|Q'\|_{L^2}} \,, \qquad f = \frac{Q' - \gamma Q''}{\|Q' - \gamma Q''\|_{L^2}}$$
\end{proof}

Now we will apply Lemma \ref{L:L-est} to prove the following proposition.  Recall that if $z$ satisfies orthogonality conditions \eqref{E:z-orth} and $\psi= \mathcal{D}_\gamma^{-1} \mathcal{L}v$, then $\psi$ satisfies orthogonality conditions \eqref{E:v-orth}.

\begin{proposition}
\label{P:tildeLspec}
Let 
\begin{equation}
\label{E:tildeLdef}
\tilde{\mathcal{L}} \defeq -2H\partial_y + 1 - yQ'-Q
\end{equation}
\begin{enumerate}
\item For $z$ satisfying the orthogonality conditions \eqref{E:z-orth}, 
$$\la \tilde{\mathcal{L}} z, z \ra \gtrsim \|z\|^2_{H^{1/2}}$$
\item For $\psi$ satisfying the orthogonality conditions \eqref{E:v-orth}, for $\gamma>0$ sufficiently small,
$$\la \tilde{\mathcal{L}} \psi, \psi \ra \gtrsim \|\psi\|^2_{H^{1/2}}$$
with constant independent of $\gamma$.
\end{enumerate}
\end{proposition}
\begin{proof}
First, we prove item (1).  Note that for any $\delta>0$,
\begin{equation}
\label{E:tildeLspec2}
\tilde{\mathcal{L}} -(1-\delta)\mathcal{L} = -(1+\delta)H\partial_y + \delta - \delta Q -yQ'
\end{equation}
We claim that for $\delta>0$,
\begin{equation}
\label{E:tildeLspec1}
\la (\tilde{\mathcal{L}} -(1-\delta)\mathcal{L}) z, z \ra \geq (1-C_2\delta) \|D^{1/2}z \|_{L^2}^2 + \frac12 \delta \|z\|_{L^2}^2
\end{equation}
Indeed, since $-yQ' \geq 0$, we can discard this term, and we have from \eqref{E:tildeLspec2} 
\begin{equation}
\label{E:tildeLspec3}
\la (\tilde{\mathcal{L}} -(1-\delta)\mathcal{L}) z, z \ra \geq \|D^{1/2} z\|_{L^2}^2 + \delta \|z\|_{L^2}^2 - \delta \int Qz^2
\end{equation}
 By the Gagliardo-Nirenberg and Peter-Paul inequalities, there exists constant $C_1>0$, $C_2>0$ so that
$$\int Qz^2 \leq \|Q\|_{L^2} \|z\|_{L^4}^2 \leq C_1 \|z\|_{L^2} \| D^{1/2}z\|_{L^2} \leq \frac12 \|z\|_{L^2}^2 + C_2 \|D^{1/2} z\|_{L^2}^2$$
Applying this in \eqref{E:tildeLspec3}, we obtain \eqref{E:tildeLspec1}.  Taking $\delta>0$ sufficiently small so that $1-C_2\delta>0$, we obtain from \eqref{E:tildeLspec1} that
$$\la \tilde{\mathcal{L}}z,z\ra \geq (1-\delta)\la \mathcal{L}z,z \ra +C_3 \|z \|_{H^{1/2}}^2$$
Item (1) follows upon applying Lemma \ref{L:L-est}(1) ($\la\mathcal{L}z,z\ra \geq 0$).  

Item (2) is addressed similarly appealing to Lemma \ref{L:L-est}(2).
\end{proof}

\section{Commutator estimates}
\label{S:commutator}

We state and prove as necessary a few commutator estimates that will be needed in the computations in the proof of Theorem \ref{T:local-virial} given in \S \ref{S:local-virial}.

\begin{lemma}
\label{L:weightedD}
For all $0<\gamma \leq 1$, all $\alpha\in \mathbb{R}$, $\la y \ra^\alpha \mathcal{D}_\gamma^{-1} \la y \ra^{-\alpha}$ is $L^2\to L^2$ bounded with operator norm independent of $\gamma$.  
\end{lemma}
\begin{proof}
Let $k(y) = e^{-y} \mathbf{1}_{y>0}$.  Then $\hat k(\xi) = (1+i\xi)^{-1}$.  Then the kernel of the operator $\la y \ra^\alpha \mathcal{D}_\gamma^{-1} \la y \ra^{-\alpha}$ is 
$$K(y,y') = \la y \ra^\alpha \la y' \ra^{-\alpha}\gamma^{-1}k( \gamma^{-1}(y-y')) $$
By duality, it suffices to restrict to $\alpha \geq 0$.  We apply Schur's test, as follows.  Using that $\la y \ra^\alpha \lesssim \la y-y' \ra^\alpha + \la y' \ra^\alpha$,  
$$|K(y,y')| \lesssim (\la y-y' \ra^\alpha \la y' \ra^{-\alpha}+1) \gamma^{-1}k(\gamma^{-1}(y-y'))$$
Using that $\la y -y' \ra^\alpha \leq \la \gamma^{-1}(y-y') \ra^\alpha$ and $\la y' \ra^{-\alpha}\leq 1$ for the first term,
$$|K(y,y')| \lesssim \la \gamma^{-1}(y-y') \ra^\alpha \gamma^{-1} k(\gamma^{-1} (y-y') \ra$$
From this it follows that
$$\int_{y'} |K(y,y')| \, dy' \lesssim \int_z \la z\ra^\alpha |k(z)| \, dz <\infty$$
and similarly for $\int_y |K(y,y')| \, dy$.
\end{proof}

\begin{lemma}[fractional Leibniz rule]
\label{L:frac-Leib}
Suppose $0<\alpha<1$, $0\leq \alpha_1,\alpha_2 \leq \alpha$ with $\alpha_1+\alpha_2=\alpha$, and $1<p,p_1,p_2<\infty$ with $\frac{1}{p}=\frac{1}{p_1}+\frac{1}{p_2}$, then
$$\| D^\alpha(fh) - f D^\alpha h - h D^\alpha f \|_{L^p} \lesssim \|D^{\alpha_1} f \|_{L^{p_1}}\|D^{\alpha_2} h \|_{L^{p_2}}$$
\end{lemma}
\begin{proof}
See, for example, \cite[Theorem A.8]{KPV}.
\end{proof}

\begin{corollary}
\label{C:com3}
For each $0<\epsilon<\frac12$, 
\begin{equation}
\label{E:Leib5}
\| D^{1/2}(fh)  - f D^{1/2} h \|_{L^2} \lesssim_\epsilon \| f\|_{L^2}^{\epsilon} \|\partial_y f\|_{L^2}^{1-\epsilon} \| \la D \ra^{1/2} h\|_{L^2}
\end{equation}
Moreover,
\begin{equation}
\label{E:Leib1}
\| D^{1/2}(fh) \|_{L^2} \lesssim (\|f\|_{L^2}^{1/2}\|\partial_y f\|_{L^2}^{1/2} + \| f\|_{L^2}^{\epsilon} \|\partial_y f\|_{L^2}^{1-\epsilon})\| \la D \ra^{1/2} h\|_{L^2}
\end{equation}
The implicit constant diverges as $\epsilon \searrow 0$ or as $\epsilon \nearrow \frac12$.   
\end{corollary}
\begin{proof}
By applying Lemma \ref{L:frac-Leib} with $\alpha=\frac12$, $\alpha_1=\frac12$, $\alpha_2=0$, $p=2$, $p_1=\frac1\epsilon$, $p_2=\frac{2}{1-2\epsilon}$, and applying the H\"older inequality on the term $hD^{1/2}f$, we obtain 
\begin{equation}
\label{E:Leib2}
\| D^{1/2}(fh) - fD^{1/2} h \|_{L^2} \lesssim_\epsilon \|D^{1/2} f\|_{L^{1/\epsilon}} \| h\|_{L^{2/(1-2\epsilon)}}
\end{equation}
Since Lemma \ref{L:frac-Leib} is not available for $\epsilon=0$ (where $p_1=\infty$) or $\epsilon=\frac12$ (where $p_2=\infty$), the above estimate has a constant that diverges as $\epsilon \searrow 0$ or $\epsilon \nearrow \frac12$.  By Sobolev embedding
$$\| D^{1/2}(fh) - fD^{1/2} h \|_{L^2} \lesssim_\epsilon  \| D^{1-\epsilon} f \|_{L^2} \| D^\epsilon h \|_{L^2} $$ 
Since Sobolev embedding for the second term fails for $\epsilon = \frac12$, the above estimate has a constant that diverges as $\epsilon \nearrow \frac12$.  
Gagliardo-Nirenberg (Cauchy-Schwarz on the Fourier side) then yields \eqref{E:Leib5}, and \eqref{E:Leib1} follows from \eqref{E:Leib5} by the Gagliardo-Nirenberg estimate $\|f\|_{L^\infty} \lesssim \|f\|_{L^2}^{1/2}\|\partial_y f\|_{L^2}^{1/2}$.
\end{proof}

\begin{lemma}
\label{L:gamma-comm}
For $0<\gamma\leq 1$, the operator
\begin{equation}
\label{E:cc0}
 \la \gamma y \ra \Big( D \la \gamma y \ra^{-1} - \la \gamma y \ra^{-1} D \Big) 
 \end{equation}
 is $H^{1/4}\to L^2$ bounded with operator norm $\lesssim \gamma^{3/4}$.
Here $D$ is the Fourier multiplier with symbol $|\xi|$.  
\end{lemma}
\begin{proof}
First, we claim that it suffices to show that 
\begin{equation}
\label{E:cc1}
D (1+i\gamma y)^{-1} - (1+i\gamma y)^{-1} D
\end{equation}
and
\begin{equation}
\label{E:cc2}
\gamma y \Big( D (1+i\gamma y)^{-1} - (1+i\gamma y)^{-1} D\Big)
\end{equation}
are both $L^2\to L^2$ bounded with operator norm $\lesssim \gamma$.  To show this, first note that \eqref{E:cc1} and \eqref{E:cc2} combined give that 
\begin{equation}
\label{E:cc3}
\la \gamma y \ra \Big( D (1+i\gamma y)^{-1} - (1+i\gamma y)^{-1} D \Big)
\end{equation}
is $L^2\to L^2$ bounded with operator norm $\lesssim \gamma$, and it remains to show that the $(1+i\gamma y)^{-1}$ term can be replaced by $\la \gamma y \ra^{-1}$.  Since the operator of multiplication by $\frac{1+i\gamma y}{\la \gamma y\ra}$ is $L^2\to L^2$ unitary operator we can compose \eqref{E:cc3} on the right by $\frac{1+i\gamma y}{\la \gamma y\ra}$ to obtain that
\begin{equation}
\label{E:cc4}
\la \gamma y \ra \Big( D \la \gamma y\ra^{-1}  - (1+i\gamma y)^{-1} D \frac{1+i\gamma y}{\la \gamma y\ra} \Big)
\end{equation}
is $L^2\to L^2$ bounded with operator norm $\lesssim \gamma$.  Rewrite \eqref{E:cc4} as
\begin{equation}
\label{E:cc5}
\la \gamma y \ra \Big( D \la \gamma y\ra^{-1}  - \la \gamma y\ra^{-1} D \Big) + \frac{\la \gamma y \ra}{1+i\gamma y} \Big( \frac{1+i\gamma y}{\la \gamma y\ra} D - D \frac{1+i\gamma y}{\la \gamma y\ra} \Big)
\end{equation}
To establish \eqref{E:cc0}, it suffices to show that the second half of \eqref{E:cc5}, i.e. 
\begin{equation}
\label{E:cc6}
\frac{\la \gamma y \ra}{1+i\gamma y} \Big( \frac{1+i\gamma y}{\la \gamma y\ra} D - D \frac{1+i\gamma y}{\la \gamma y\ra} \Big)
\end{equation}
is $H^{1/4} \to L^2$ bounded with operator norm $\lesssim \gamma$.  Since the operator of multiplication by $\frac{\la \gamma y \ra}{1+i\gamma y}$ is $L^2\to L^2$ unitary, it suffices to show 
\begin{equation}
\label{E:cc7}
 \frac{1+i\gamma y}{\la \gamma y\ra} D - D \frac{1+i\gamma y}{\la \gamma y\ra}
\end{equation}
is $H^{1/4} \to L^2$ bounded with operator norm $\lesssim \gamma^{3/4}$.  This follows from the estimate of Calder\'on \cite{Calderon}
$$\| D(fg) - gDf \|_{L^2} \lesssim \|Dg\|_{L^4} \|f\|_{L^4}$$
by taking $g(y) = \frac{1+i\gamma y}{\la \gamma y\ra}$.  Then by the $L^4\to L^4$ boundedness of the Hilbert transform,
$$\| Dg\|_{L^4} \lesssim \| \partial_y g\|_{L^4} =  \gamma^{3/4}$$
This completes the proof of \eqref{E:cc0} assuming \eqref{E:cc1} and \eqref{E:cc2}.

We prove \eqref{E:cc1} and \eqref{E:cc2} by passing to the Fourier side, in which they become the assertions that the operators 
\begin{equation}
\label{E:cc8}
|\xi| T_\gamma - T_\gamma |\xi|
\end{equation}
\begin{equation}
\label{E:cc9}
\gamma \partial_\xi (|\xi| T_\gamma - T_\gamma |\xi|)
\end{equation}
are $L^2\to L^2$ bounded with operator norm $\lesssim \gamma$, where $T_\gamma$ is the operator of convolution with kernel $k_\gamma$, where
$$k_\gamma(\xi) = \gamma^{-1} k_1(\gamma^{-1} \xi) \,, \qquad k_1(\alpha) = e^{-\alpha} \mathbf{1}_{\alpha >0}$$
and $|\xi|$ is the operator of multiplication by $|\xi|$.  These correspond to the operators with distributional kernels
$$k_\gamma(\xi-\eta) (|\xi|-|\eta|)$$
$$\gamma \partial_\xi [ k_\gamma(\xi-\eta) (|\xi|-|\eta|)]$$
Schur's test can be applied to these explicit kernels.  To see this, let
$$K_\gamma(\xi,\eta) = \gamma^{-1} K_1(\gamma^{-1}\xi, \gamma^{-1}\eta) \,, \qquad K_1(\xi,\eta) = e^{-(\xi-\eta)} (|\xi|-|\eta|) \mathbf{1}_{\xi-\eta>0}$$
Let
$$L_\gamma(\xi,\eta) = \gamma^{-1} L_1(\gamma^{-1}\xi,\gamma^{-1}\eta) \,, \qquad L_1(\xi,\eta) = e^{-(\xi-\eta)} (-|\xi|+|\eta|+\sgn\xi) \mathbf{1}_{\xi-\eta>0}$$

Schur's test implies that the operators corresponding to $K_1$ and $L_1$ are $L^2\to L^2$ bounded, and thus the operators corresponding to kernels $K_\gamma$ and $L_\gamma$ are bounded with norms independent of $\gamma$.

Note that $\partial_\xi K_1  = L_1$ in the distributional sense (this uses, importantly, the fact that the factor $|\xi|-|\eta|$ in the definition of $K_1(\xi,\eta)$ vanishes at the line of discontinity of $\mathbf{1}_{\xi-\eta>0}$).  It follows that $\partial_\xi K_\gamma = \gamma^{-1} L_\gamma$.  

The kernel corresponding to the operator \eqref{E:cc8} can be expressed as
$$k_\gamma(\xi-\eta) (|\xi|-|\eta|) = \gamma K_\gamma(\xi,\eta)$$
and the kernel corresponding to the operator \eqref{E:cc9} can be expressed as
$$\gamma \partial_\xi [ k_\gamma(\xi-\eta) (|\xi|-|\eta|)]= \gamma L_\gamma(\xi,\eta)$$
Hence both operators are $L^2\to L^2$ bounded with operator norm $\sim \gamma$.  This completes the proof that the operators \eqref{E:cc8} and \eqref{E:cc9} are $L^2\to L^2$ bounded with operator norm $\sim  \gamma$, and thus that the same statement applies to the operators \eqref{E:cc1} and \eqref{E:cc2}, completing the proof.
\end{proof}

\begin{lemma}
\label{L:com5}
For $\chi \in C_c^\infty(\mathbb{R})$ and $0<h \leq 1$, we have
\begin{equation}
\label{E:co19}
\left| \int_y  \chi(hy)  \cdot  H \partial_y w \cdot  \partial_y w \, dy \right| \lesssim h^2 \|w\|_{L_y^2}^2 
\end{equation}
where the implicit constant depends on $\chi$ but is uniform in $h$.

For any $0< \gamma\leq 1$, all $y_0\in \mathbb{R}$,
\begin{equation}
\label{E:cc10}
\left| \int g_{\gamma,y_0} (Hw_y) w_y \right| \lesssim \gamma \int g_{\gamma,y_0}' \, w^2 \, dx
\end{equation}
with implicit constant independent of $\gamma$ and $y_0$, where $g_{\gamma,y_0}$ is defined in \eqref{E:g-def}.
\end{lemma}
\begin{proof}
See \cite[Lemma 4.2]{Z} and \cite[Lemma 3]{KM}.
\end{proof}

\begin{lemma}
\label{L:com6}
\begin{equation}
\label{E:com1}
\|  (\la \gamma y\ra^{-1}  \mathcal{D}_\gamma^{-1} \mathcal{L} - \mathcal{L} \mathcal{D}_\gamma^{-1} \la \gamma y \ra^{-1} ) w\|_{L_y^2} \lesssim \gamma \ln \gamma^{-1}  \| \la \gamma y \ra^{-1} w\|_{L_y^2}
\end{equation}
\end{lemma}

\begin{proof}
By splitting $\mathcal{L} = (-H\partial_y +1) - Q$ and taking $f = \la \gamma y  \ra^{-1}w$, it suffices to prove the three estimates

\begin{equation}
\label{E:com3}
\|  (\la \gamma y\ra^{-1}  \mathcal{D}_\gamma^{-1}  \la \gamma y \ra - \mathcal{D}_\gamma^{-1}  ) f\|_{L_y^2} \lesssim \gamma \| f\|_{L_y^2}
\end{equation}
\begin{equation}
\label{E:com4}
\|  (\la \gamma y\ra^{-1}  \mathcal{D}_\gamma^{-1} (-H\partial_y) \la \gamma y \ra - (-H\partial_y)\mathcal{D}_\gamma^{-1}  ) f\|_{L_y^2} \lesssim \gamma(\ln \gamma^{-1}) \| f\|_{L_y^2}
\end{equation}
\begin{equation}
\label{E:com5}
\|  (\la \gamma y\ra^{-1}  \mathcal{D}_\gamma^{-1} Q  \la \gamma y\ra -  Q \mathcal{D}_\gamma^{-1}  ) f\|_{L_y^2} \lesssim \gamma \| f\|_{L_y^2}
\end{equation}

First, we prove the estimate \eqref{E:com3}. Since $\frac{\la \gamma y' \ra}{\la \gamma y \ra} - 1 = \frac{\la \gamma y' \ra - \la \gamma y \ra}{\la \gamma y\ra} = \frac{\gamma (y'-y)}{\la \gamma y \ra} \frac{ \gamma y' + \gamma y}{\la \gamma y' \ra + \la \gamma y \ra}$, the kernel of the operator is 
\begin{align*}
K(y,y') &= \left( \frac{ \langle \gamma y' \rangle}{ \langle \gamma y \rangle} -1  \right) \gamma^{-1} e^{-\frac{|y-y'|}{\gamma}} \mathbf{1}_{y'<y}   \\ 
& = \frac{ \gamma^2 (y'+y) (y'-y)}{ \langle \gamma y \rangle  (\langle \gamma y' \rangle + \langle \gamma y \rangle)}   \gamma^{-1} e^{-\frac{|y-y'|}{\gamma}} \mathbf{1}_{y'<y}   \\
& = \frac{ \gamma^2 (y'+y) }{ \langle \gamma y \rangle  (\langle \gamma y' \rangle + \langle \gamma y \rangle)}    q(\gamma^{-1} |y-y'| ) 
\end{align*}
where $q(z) = z e^{-z} \mathbf{1}_{z>0}$.  Since $\int_{y'} q(\gamma^{-1} |y-y'|) dy' = \gamma$  and $\int_y q(\gamma^{-1} |y-y'|) dy = \gamma$, and the prefactor is uniformly bounded by $\gamma$,  the Schur test implies that this operator is $L^2_y \rightarrow L^2_y$ bounded with $O(\gamma^2)$ operator norm.

Next, we consider the estimate \eqref{E:com4}.  The kernel of the operator is
$$K(y,y')= \left( \frac{\la \gamma y' \ra}{\la \gamma y \ra} - 1\right) \gamma^{-2} k (\gamma^{-1}(y-y'))$$
where 
$$\hat k(\xi) = \frac{|\xi|}{1+i\xi}$$
Again since $\frac{\la \gamma y' \ra}{\la \gamma y \ra} - 1 = \frac{\la \gamma y' \ra - \la \gamma y \ra}{\la \gamma y\ra} = \frac{\gamma (y'-y)}{\la \gamma y \ra} \frac{ \gamma y' + \gamma y}{\la \gamma y' \ra + \la \gamma y \ra}$, we can rewrite 
$$K(y,y') =  \frac{1}{\la \gamma y \ra} \frac{ \gamma y' + \gamma y}{\la \gamma y' \ra + \la \gamma y \ra} \tilde k (\gamma^{-1}(y-y'))$$
where
$$\tilde k(z) = z k(z)$$
Now
$$\hat{\tilde k}(\xi) = i \partial_\xi \frac{|\xi|}{1+i\xi} = \frac{i \sgn \xi}{(1+i\xi)^2}$$
Since $\hat{\tilde k}(\xi)$ is in $L^1$, it follows that $\tilde k(z)$ is continuous and $|\tilde k(z)| \leq \|\hat{\tilde k} \|_{L^1}<\infty$ for all $z\in \mathbb{R}$.      Moreover, integration by parts in the inverse transform gives that $|k(z)| \lesssim |z|^{-1}$ for all $z\in \mathbb{R}$.  Combining, we obtain
$$|\tilde k(z)| \lesssim \la z \ra^{-1}$$
and this is the optimal decay estimate as $|z|\to \infty$.  Hence we have
$$|K(y,y')| \lesssim \frac{1}{\la \gamma y \ra \la \gamma^{-1}(y-y')\ra}$$

We apply the following ``weighted Schur test'' (see \cite[Chapter 4, Exercise 26, p.199]{SS}):  If $Tf(y) = \int K(y,y') f(y') \, dy'$ and $w$ is any measurable function such that $0<w(y')<\infty$ for all $y'$ and
$$M_{y'} = \sup_y w(y)^{-1} \int_{y'} |K(y,y')| w(y') \, dy' \,, \quad M_y = \sup_{y'} w(y')^{-1} \int_y |K(y,y')| w(y) \, dy$$
Then 
$$\|T\|_{L^2\to L^2} \leq \sqrt{M_yM_{y'}}$$
We apply this with $w(y') = \la \gamma y' \ra^{-1}$.  We have, for $0< \gamma \leq \frac12$,
\begin{equation}
\label{E:My-bound}
M_{y'}  \lesssim \sup_y  \int \frac{1}{\la \gamma^{-1}(y-y') \ra \la \gamma y' \ra} \, dy' \lesssim \gamma \ln \gamma^{-1}
\end{equation}
\begin{equation}
\label{E:Mx-bound}
M_y \lesssim \sup_{y'} \;  \la \gamma y' \ra \int \frac{1}{\la \gamma y\ra^2 \la \gamma^{-1}(y-y') \ra} \, dy \lesssim \gamma \ln \gamma^{-1}
\end{equation}
The estimate \eqref{E:com4} then follows, but we will give an outline of \eqref{E:My-bound} and \eqref{E:Mx-bound}.  

For \eqref{E:My-bound}, decompose the $y'$ integration into the regions $|y'|\sim |y|$, $|y'|\gg |y|$ and $|y'| \ll |y|$, and we label the corresponding pieces $M_{y',\sim}$, $M_{y',+}$, and $M_{y',-}$.

For $|y'|\sim |y|$, we have
$$M_{y',\sim} \lesssim  \sup_y \frac{1}{\la \gamma y \ra} \int_{|y'|\sim |y|}  \frac{1}{\la \gamma^{-1}(y-y') \ra} \, dy'$$
We can then change variable $z=y'-y$ (and still have $|z|\lesssim |y|$), and split into $|z|\leq \gamma$ and $|z|\geq \gamma$, to obtain 
\begin{align*}
M_{y',\sim} &\lesssim  \sup_y \frac{1}{\la \gamma y \ra} \left( \int_{|z|\leq \gamma} \, dz +  \int_{\gamma \leq |z| \lesssim |y|} \frac{1}{\la \gamma^{-1} z\ra } \, dz \right) \\
&\lesssim \sup_y \frac{1}{\la \gamma y \ra} \left( \gamma + \int_{\gamma \leq |z| \lesssim |y|} \frac{dz}{\gamma^{-1} z} \right) \\
&\lesssim \sup_{|y|\lesssim \gamma} \frac{\gamma}{\la \gamma y \ra}  + \sup_{\gamma \lesssim |y| \leq \gamma^{-1}} \frac{\gamma}{\la \gamma y \ra} \left(1 + \int_{\gamma \leq |z| \lesssim |y|} \frac{dz}{ z} \right) + \sup_{ |y| \gtrsim \gamma^{-1}} \frac{\gamma}{\la \gamma y \ra} \left(1 + \int_{\gamma \leq |z| \lesssim |y|} \frac{dz}{ z} \right)
\end{align*}
where in the last step, we used that $|y|\lesssim \gamma$ implies the integral on $\gamma \leq |z| \lesssim |y|$ is over the empty set.  The first two terms are bounded by $\gamma\ln \gamma^{-1}$, and in the third we use that $\la \gamma y \ra \sim \gamma y$,
$$M_{y',\sim} \lesssim \gamma\ln\gamma^{-1} + \sup_{|y| \gtrsim \gamma^{-1}} \frac{\ln |y|}{y} \lesssim \gamma \ln \gamma^{-1}$$

Now consider the case $|y'| \gg |y|$ in \eqref{E:My-bound}.    We have
$$M_{y',+} \lesssim \int \frac{dy'}{\la \gamma^{-1} y' \ra \la \gamma y' \ra}$$
Breaking the $y'$ integration into the regions $|y'|\leq \gamma$, $\gamma \leq |y'| \leq \gamma^{-1}$ and $|y'| \geq \gamma^{-1}$ and using the appropriate reductions for $\la \gamma y' \ra$ and $\la \gamma^{-1}y' \ra$ in each subregion,
$$M_{y',+} \lesssim \int_{|y'|\leq \gamma} \, dy' + \int_{\gamma \leq |y'| \leq \gamma^{-1}} \frac{dy'}{\gamma^{-1} y'} + \int_{|y'|\geq \gamma^{-1}} \frac{dy'}{y'^2} \lesssim \gamma$$

Finally consider the case $|y'| \ll |y|$ in \eqref{E:My-bound}.  Then
$$M_{y',-} \lesssim \sup_y  \frac{1}{\la \gamma^{-1} y\ra} \int_{|y'|\ll |y|} \frac{dy'}{\la \gamma y' \ra}$$
Breaking the supremum in $y$ into $|y|\leq \gamma$ and $|y|\geq \gamma$,  we find
\begin{align*}
M_{y',-} &\lesssim \sup_{|y|\leq \gamma}  \frac{1}{\la \gamma^{-1} y\ra} \int_{|y'|\ll |y|} \frac{dy'}{\la \gamma y' \ra }+ \sup_{|y|\geq \gamma}  \frac{1}{\la \gamma^{-1} y\ra} \int_{|y'|\ll |y|} \frac{dy'}{\la \gamma y' \ra} \\
&\lesssim \gamma + \sup_{|y|\geq \gamma} \frac{\gamma}{y} \int_{|y'| \leq |y|} \frac1{\la \gamma y' \ra} \, dy' \lesssim \gamma
\end{align*}
where the first term results from the fact that $\la \gamma^{-1}y\ra \sim 1$
 but the $y'$ integration is carried over the small set $|y'|\leq \gamma$, and in the second term we used that $\la \gamma^{-1} y \ra \sim \gamma^{-1} y$.   For this second term, we do not use the $\la \gamma y' \ra$ denominator and just bound the integral by $|y|$ obtaining the upper bound of $\gamma$.  This completes the proof of \eqref{E:My-bound}.
 
Now we prove \eqref{E:Mx-bound} by decomposing the $x$ integral into the regions $|y|\sim |y'|$, $|y|\ll |y'|$, and $|y| \gg |y'|$, and label the bounds on each piece by $M_{y,\sim}$, $M_{y,-}$, and $M_{y,+}$ respectively.  First we consider the case $|y|\sim |y'|$ in \eqref{E:Mx-bound}
$$M_{y,\sim} \leq \sup_{y'} \frac{1}{\la \gamma y' \ra} \int_{|y|\sim |y'|} \frac{dy}{\la \gamma^{-1}(y-y') \ra}$$
From here, it is completely analogous to the proof of the bound $M_{y',\sim}$ given above, so we conclude  $M_{y,\sim} \lesssim \gamma \ln \gamma^{-1}$.   Next, we consider the case $|y|\ll |y'|$ in \eqref{E:Mx-bound},
$$M_{y, -} \lesssim \sup_{y'} \frac{\la \gamma y' \ra}{\la \gamma^{-1} y' \ra} \int_{|y|\ll |y'|} \frac{dy}{\la \gamma y \ra^2 } $$
Splitting the supremum in $y'$ into the regions $|y'| \leq \gamma$, $\gamma \leq |y'| \leq \gamma^{-1}$, and $|y'| \geq \gamma^{-1}$, we obtain
$$M_{y,-} \lesssim  
\begin{aligned}[t]
&\sup_{|y'|\leq \gamma} \frac{\la \gamma y' \ra}{\la \gamma^{-1} y' \ra} \int_{|y|\ll |y'|} \frac{dy}{\la \gamma y \ra^2 }  +  \sup_{ \gamma \leq |y'| \leq \gamma^{-1} } \frac{\la \gamma y' \ra}{\la \gamma^{-1} y' \ra} \int_{|y|\ll |y'|} \frac{dy}{\la \gamma y \ra^2 } \\
&+  \sup_{ |y'| \gtrsim \gamma^{-1} } \frac{\la \gamma y' \ra}{\la \gamma^{-1} y' \ra} \int_{|y|\ll |y'|} \frac{dx}{\la \gamma y \ra^2 }
\end{aligned}$$
Making the appropriate reductions in each case
$$M_{y,-} \lesssim  \sup_{|y'|\leq \gamma} \int_{|y|\ll \gamma } dy  +  \sup_{ \gamma \leq |y'| \leq \gamma^{-1} } \frac{1}{ \gamma^{-1} y' } \int_{|y|\ll |y'|} dy +  \sup_{ |y'| \gtrsim \gamma^{-1} } \gamma^2 \int  \frac{dy}{\la \gamma y \ra^2 }$$
Each term is bounded by $\gamma$ (for the last, we use the substitution $z=\gamma y$ to evaluate the integral).  Finally, we consider the region $|y| \gg |y'|$ in \eqref{E:Mx-bound}.  We have
$$M_{y,+} \lesssim \sup_{y'} \la \gamma y' \ra \int_{|y| \gg |y'|} \int \frac{dy}{\la \gamma y \ra^2 \la \gamma^{-1} y \ra} \lesssim \int \frac{dy}{ \la \gamma y \ra \la \gamma^{-1} y\ra} $$
The integral on the right is analogous to that obtained in the estimate of $M_{y',+}$ in the estimate of \eqref{E:My-bound}, so a bound of $\gamma$ is obtained.  This completes the proof of \eqref{E:Mx-bound}.

Now that we have completed the proof of \eqref{E:My-bound}, \eqref{E:Mx-bound}, the proof of \eqref{E:com4} is complete.

Finally, we prove the estimate \eqref{E:com5}.  The kernel of the operator is 
$$K(y,y') = \mu_\gamma(y,y') \gamma^{-1} e^{-\frac{|y-y'|}{\gamma}} \mathbf{1}_{y'<y} \quad \text{where} \quad \mu_\gamma(y,y') = \frac{ Q(y') \langle \gamma y' \rangle}{ \langle \gamma y \rangle} -Q(y)$$
It suffices to show that 
\begin{equation}
\label{E:mu-bound}
|\mu_\gamma(y,y')| \lesssim |y-y'|
\end{equation}
Indeed, \eqref{E:mu-bound} implies that 
$$|K(y,y')| \lesssim  q(\gamma^{-1}|y-y'|)$$
where $q(z) = z e^{-z} \mathbf{1}_{z>0}$, so that by Schur's test, the operator in \eqref{E:com5} is $L^2\to L^2$ bounded with operator norm $\lesssim \gamma$.    To prove \eqref{E:mu-bound}, note that
\begin{equation}
\label{E:mu-bound2}
\mu_\gamma(y,y') = \frac{\la \gamma y' \ra}{\la \gamma y \ra}( Q(y') - Q(y)) + Q(y) \left( \frac{\la \gamma y' \ra}{\la \gamma y \ra} - 1 \right)
\end{equation}
For the second term in \eqref{E:mu-bound2},
$$Q(y) \left( \frac{\la \gamma y' \ra}{\la \gamma y \ra} - 1 \right) =Q(y) \frac{ (\la \gamma y' \ra - \la \gamma y \ra)(\la \gamma y' \ra + \la \gamma y \ra)}{\la \gamma y \ra( \la \gamma y' \ra + \la \gamma y \ra)} = Q(x)\frac{\gamma(y+y')}{\la \gamma y \ra( \la \gamma y\ra + \la \gamma y' \ra)} \gamma(y'-y)$$
from which it is clear that this quantity is bounded by $ |y-y'|$

For the first term in \eqref{E:mu-bound2}, applying the explicit formula  $Q(y)=4/(1+y^2)$,
$$ \frac{\la \gamma y' \ra}{\la \gamma y \ra}( Q(y') - Q(y))  = \frac{ 4 (y+y') \la \gamma y' \ra}{(1+y^2)(1+y'^2) \la \gamma y \ra} (y-y')$$ 
To see that this quantity is bounded by $|y-y'|$ (uniformly in $\gamma$), we investigate the prefactor
$$\nu_\gamma (y,y') = \frac{ 4 (y+y') \la \gamma y' \ra}{(1+y^2)(1+y'^2) \la \gamma y \ra}$$
and show that $|\nu_\gamma(y,y')| \lesssim 1$ independently of $\gamma>0$.  This is handled in three cases as follows:
$$|y|\sim |y'| \implies |\nu_\gamma(y,y')| \lesssim \frac{ |y| \la \gamma y \ra}{ \la y \ra^4 \la \gamma y \ra} \lesssim 1$$
$$|y| \ll |y'| \implies  |\nu_\gamma(y,y')| \lesssim \frac{ |y'| \la \gamma y' \ra}{\la y' \ra^2} \lesssim 1$$
and finally, when $|y| \gg |y'|$, we use that $\frac{\la \gamma y' \ra}{\la \gamma y\ra} \lesssim 1$ and thus
$$|y| \gg |y'| \implies |\nu_\gamma(y,y')| \lesssim \frac{ |y|}{\la y \ra^2} \lesssim 1$$
This completes the proof of \eqref{E:mu-bound} and thus the proof of \eqref{E:com5}.
\end{proof}

\begin{lemma}
\label{L:co3}
For $\chi \in C_c^\infty(\mathbb{R})$, then the commutator
$[ H \partial_y, \chi(hy) ]$ is $L_y^2\to L_y^2$ bounded with operator norm  $\lesssim h$,  with the implicit constant depending on $\chi$ but uniform in $h$.  
\end{lemma}
\begin{proof}
We compute
$$H  \partial_y \chi(hy) f(y) = \frac{1}{\pi} \pv \int \left( \frac{h \chi ' (hy') f(y')}{y - y'} + \frac{\chi (hy') f' (y')}{y-y'} \right) dy' $$
and
$$\chi(hy) H  \partial_y  f(y) =  \frac{1}{\pi} \pv \int \frac{\chi (hy)}{y - y'}  f'(y') dy'  $$
Subtracting,
\begin{align*}
\indentalign [ H \partial_y, \chi(hy) ] f(y) \\
& = \frac{1}{\pi} \pv \int  \frac{h \chi' (hy')}{y-y'}  f(y') dy' + \frac{1}{\pi} \pv \int \frac{\chi (hy') - \chi (hy)}{y- y'} f'(y') dy' \\
& =  \frac{1}{\pi} \pv \int \frac{h \chi' (hy')}{y-y'}  f(y') dy' - \frac{1}{\pi} \pv \int \partial_{y'} \Big( \frac{\chi (hy') - \chi (hy)}{y- y'} \Big) f(y') dy' \\
& =  - \frac{1}{\pi} \pv \int   \frac{\chi (hy') - \chi (hy)}{(y- y')^2 }   f(y') dy' \\
& = - \frac{1}{\pi} \pv \int   \frac{\chi (hy') - \chi (hy) - h\chi' (hy) (y-y')}{(y- y')^2 }   f(y') dy' \\
& \qquad  - \frac{1}{\pi} h \chi'(hy) \pv \int   \frac{f(y')}{y- y' } dy'   \\
&= Af(y) - h \chi'(hy) Hf(y)
\end{align*}
where the operator $A$ is defined by 
$$Af(y) = - \frac{1}{\pi} \pv \int   \frac{\chi (hy') - \chi (hy) - h\chi' (hy) (y-y')}{(y- y')^2 }   f(y') dy'  $$
The second term is $L^2\to L^2$ bounded with operator norm $h$ by the $L^2\to L^2$ boundedness of the Hilbert transform and thus it suffices to prove that the operator $A$ is $L^2\to L^2$ bounded with operator norm $h$.  We observe
$$ |\chi (hy') - \chi (hy) - h\chi' (hy) (y-y')| \lesssim h^2 |y-y'|^2  \ .$$
and note that the $\chi$ factors restrict both $|y'| \lesssim h^{-1}$ and $|y| \lesssim h^{-1}$ (with constant depending on the size of the $\chi$ support), and hence we can add the restriction $|y-y'|\lesssim h^{-1}$ to the integrand.  
\begin{equation}
| Af(y)  | \lesssim  h^2 \int_{|y-y'| \lesssim h^{-1}}    | f(y') | dy'  \ ,
\end{equation}
We conclude by applying Young's inequality (or the Schur test).
\end{proof}

\begin{lemma}
\label{L:co4}
For $0\leq \alpha\leq 2$,
$$\| \la y \ra^\alpha H \partial_y \la y \ra^{-\alpha} f \|_{L_y^\infty} \lesssim \| f\|_{L_y^\infty} + \| f'(y) \la y \ra^{-1} \|_{L_y^\infty} + \|f''\|_{L_y^\infty}$$
Consequently, $\mathcal{L}_c$ preserves decay up to quadratic order.
\end{lemma}
\begin{proof}
The operator has the representation
$$I = ( \la y \ra^\alpha H \partial_y \la y \ra^{-\alpha} f ) (y) =\lim_{\epsilon \searrow 0} \int_{|y'|>\epsilon} \frac{1}{y'} \partial_{y'} \left[ \frac{\la y \ra^\alpha}{\la y-y' \ra^{\alpha}} f(y-y') \right] \, dy'$$
Let $\chi(y') \in C_c^\infty(\mathbb{R})$ be an even nonnegative smooth compactly supported function with $\chi(y')=1$ on $|y'|\leq 1$.  Then we break $I=I_-+I_+$ into an inner piece $I_-$ and outer piece $I_+$ by inserting $\chi(y')$ and $1-\chi(y')$ respectively.   For the inner piece $I_-$, we distribute $\partial_{y'}$ to obtain
$$I_- = -\lim_{\epsilon \searrow 0} \int_{|y'|>\epsilon} \frac{ \la y \ra^\alpha \chi(y')}{y'} g(y-y') \, dy'$$
where
$$g(z) = \partial_z  [ \la z \ra^{-\alpha} f(z) ]$$
By the oddness of the inner kernel, we can reexpress as
$$I_- = \int_{y'=0}^\infty \frac{\la y \ra^\alpha \chi(y')}{y'} [ g(y+y')-g(y-y')] \, dy'$$
By the mean-value theorem, for each $y$ there exists $z_0=z_0(y')$ such that $-y' < z_0 < y'$ with
$$g(y+y') - g(y-y') = 2y' g'(y+z_0(y'))$$
Substituting,
$$I_- = 2\int_{y'=0}^\infty \la y \ra^\alpha \chi(y') g'(y+z_0(y'))\, dy'$$
Note that
$$|g'(z)| \lesssim \la z \ra^{-\alpha}[ \la z \ra^{-2} |f(z)| + \la z \ra^{-1} |f'(z)| + |f''(z)|]$$
Since $y'$ is confined to the compact support of $\chi$, it follows that $\la y + z_0(y')\ra^{-\alpha} \sim \la y \ra^{-\alpha}$, and hence 
$$|I_-| \lesssim \| \la z \ra^{-2} f\|_{L^\infty} + \| \la z \ra^{-1} f' \|_{L^\infty} + \| f'' \|_{L^\infty}$$

For the outer piece $I_+$, we have, by integration by parts
$$I_+ = \int_{y'} \zeta(y')\frac{\la y \ra^\alpha}{\la y-y' \ra^{\alpha}} f(y-y') \, dy'$$
with 
$$\zeta(y') = \left( \frac{ \chi(y')-1}{y'}  \right)' = \frac{\chi'(y')}{y'} + \frac{1-\chi(y')}{(y')^2}$$
which satisfies $|\zeta(y')| \lesssim \la y' \ra^{-2}$.   Thus
$$|I_+| \lesssim \int_{y'} K(y,y') |f(y-y')| \, dy'$$
where
$$K(y,y')= \frac{ \la y \ra^\alpha}{\la y' \ra^2 \la y-y'\ra^\alpha}$$
For $0\leq \alpha \leq 2$, we have $\int K(y,y') \, dy' \lesssim 1$, so 
$$|I_+| \lesssim \|f\|_{L^\infty}$$
\end{proof}


\begin{lemma}
\label{L:H-commutator}
For any functions $g$ and $F$, and any $k\geq 0$,
\begin{equation}
\label{E:H-commutator}
\| H(gF) - gHF \|_{L_y^2} \lesssim_k \|g \|_{H^{k+1}} \|F\|_{H^{-k}}
\end{equation}
\end{lemma}
\begin{proof}
First, we observe that it suffices to assume that $\hat F(\xi)$ is supported in $|\xi| \geq 4$.  Let $\chi(\xi)$ be a smooth function so that $\chi(\xi) = 1$ on $-1\leq |\xi|\leq 1$ and $\supp \chi \subset (-2,2)$.  Let $\widehat{P_{\lo}F}(\xi) = \chi(\xi/4) \hat F(\xi)$ and $P_\hi = I-P_\lo$.  Decompose
$$F= F_\lo+F_\hi \,, \qquad \text{where } F_\lo = P_\lo F \,, \quad F_\hi = P_\hi F$$
Then
$$H(gF) - gHF = H(gF_\lo) - gHF_\lo + [H(gF_\hi) - gHF_\hi]$$
We note that it is straightforward to obtain the bound \eqref{E:H-commutator} for the first two terms:
$$\| H(gF_\lo) \|_{L_y^2} \lesssim \| gF_\lo\|_{L_y^2} \lesssim \|g\|_{L_y^\infty} \|F_\lo\|_{L_y^2} \lesssim \|g\|_{H_y^1} \|F\|_{H_y^{-k}}$$
and very similarly for $gHF_\lo$.  Thus it suffices to prove the bound for $H(gF_\hi) - gHF_\hi$, i.e. it suffices to assume that  $\hat F(\xi)$ is supported in $|\xi| \geq 4$. 

Next, we observe that it suffices to assume that $\hat g(\xi)$ is supported in $|\xi| \geq 1$.  To see this, \emph{redefine} $\widehat{P_{\lo}g}(\xi) = \chi(\xi) \hat g(\xi)$ and $P_\hi = I-P_\lo$ (recall that in the argument above, $\chi(\xi)$ was replaced with $\chi(\xi/4)$).  Decompose
$$g= g_\lo+g_\hi \,, \qquad \text{where } g_\lo = P_\lo g \,, \quad g_\hi = P_\hi g$$
Then
$$H(gF) - gHF = [H(g_\lo F) - g_\lo HF] + [H(g_\hi F) - g_\hi HF]$$
where we can assume that $\hat F(\xi)$ is supported in $|\xi| \geq 4$, and we know that $\hat g_\lo(\xi)$ is supported in $|\xi|\leq 2$.  In the first term, decompose $F = F_-+F_+$ 
where $F_-$ is the projection of $F$ onto negative frequencies, and  $F_+$ is the projection of $F$ onto positive frequencies.  Then
$$H(gF) - gHF = [H(g_\lo F_-) - g_\lo HF_-] + [H(g_\lo F_+) - g_\lo HF_+] +  [H(g_\hi F) - g_\hi HF]$$
Noting that $HF_- = F_-$, and moreover due to the frequency supports, $H(g_\lo F_-)= g_\lo F_-$, the first term is zero.  Likewise, the second term is zero, leaving only to estimate $H(g_\hi F) - g_\hi HF$.  Thus we have shown that it suffices to assume that $\hat g(\xi)$ is supported in $|\xi| \geq 1$.

Now we complete the proof assuming that $\hat g(\xi)$ is supported $|\xi| \geq 1$ and $\hat F(\xi)$ is supported in $|\xi|\geq 1$.  
Apply a Littlewood-Paley decomposition
$$g = \sum_N P_N g \,, \qquad F = \sum_M P_M F$$
where the sums are taken over dyads $|N|\geq 1$ and $|M|\geq 1$, respectively.  Then
\begin{equation}
\label{E:freq01}
H(gF) - gHF = \sum_{M,N} [ H(P_Ng \, P_M F) - P_N g \, HP_MF ]
\end{equation}
Split the set of all $(M,N)$ into two subclasses.  The first subclass $S$ consists of those $(M,N)$ for which $|N| \geq |M|/4$, and the second subclass $S^c$ consists of those $(M,N)$ for which $|N| < |M|/4$.   Note that for any $(M,N) \in S^c$, the sign of $M$ is the same as the sign of $M+N$, and thus
$$H(P_Ng \, P_M F) - P_N g \, HP_MF =0$$
since either $H$ reduces to $+I$ in both terms, or $H$ reduces to $-I$ in both terms.
It follows that the sum in \eqref{E:freq01} is only over $(M,N) \in S$.  For $(M,N) \in S$ we can transfer any number of derivatives from $F$ to $g$:
\begin{align*}
\| H(P_Ng \, P_M F) \|_{L_y^2} &\lesssim \|P_Ng \, P_M F \|_{L_y^2}  \\
&\lesssim N^{k+\frac14} \|P_Ng \|_{L_y^\infty}   M^{-k-\frac14} \|P_MF \|_{L_y^2} \\
&\lesssim N^{k+\frac34} \|P_N g \|_{L_y^2} M^{-k-\frac14} \|P_MF \|_{L_y^2}  && \text{by Bernstein}\\
&\lesssim N^{-\frac14}M^{-\frac14} \|g\|_{H^{k+1}} \|F\|_{H^k}
\end{align*}
Similarly, for $(M,N) \in S$, we have
$$\| P_Ng \, HP_M F \|_{L_y^2} \lesssim N^{-\frac14}M^{-\frac14} \|g\|_{H^{k+1}} \|F\|_{H^k}$$
Thus, returning to \eqref{E:freq01}, we have
\begin{align*}
\|H(gF) - gHF \|_{L_y^2} &\lesssim \sum_{(M,N)\in S_1} \| H(P_Ng \, P_M F) - P_N g \, HP_MF\|_{L_y^2} \\
& \lesssim \left( \sum_{M,N} N^{-\frac14}M^{-\frac14} \right)\|g\|_{H^{k+1}} \|F\|_{H^k} \lesssim \|g\|_{H^{k+1}} \|F\|_{H^k} 
\end{align*}
as claimed.

\end{proof}

\section{The local virial inequality}
\label{S:local-virial}


In this section, we will carry out the proof of Theorem \ref{T:local-virial}. 

\begin{proof}
The proof combines two key steps covered in Proposition \ref{P:step1} and Proposition \ref{P:local-virial2}, which are each stated and proved after this proof (at the end of the section).    The proof uses commutator estimates Lemma \ref{L:com6} and the spectral estimate Proposition \ref{P:L2-spec-est}(2).

By Proposition \ref{P:step1}, we have available estimate \eqref{E:cc15}, and it suffices to prove the estimate for $y_0=0$, that is, to control the term $\gamma^{-1} \| (g_{\gamma,0}')^{1/2} v\|_{L_{[0,T]}^2 L_y^2}^2$ appearing on the right side in \eqref{E:cc15}. 
We follow the strategy of Kenig \& Martel \cite{KM} of passing from $v$ to $\psi$ solving an adjoint problem, although we will conjugate with a different operator. Let $v$ satisfy
\begin{equation}
\partial_t v = \mathbb{P} v + \partial_y \mathcal{L} v + \partial_y  f ,
\end{equation}
with 
\begin{equation}  
\mathbb{P} v = \frac{\la v , \mathcal{L} \partial_y^2 Q \ra }{\| \partial_y Q \|_{L^2}^2} \partial_y Q . 
\end{equation}
Let
$$\psi = \mathcal{D}_\gamma^{-1}\mathcal{L} v$$
Then
$$\partial_t \psi =  \mathcal{D}_\gamma^{-1} \mathcal{L}  \mathbb{P}  v +  \mathcal{D}_\gamma^{-1} \mathcal{L} \partial_y \mathcal{L} v + \mathcal{D}_\gamma^{-1} \mathcal{L} \partial_y f$$
Since
$$ \mathcal{D}_\gamma^{-1} \mathcal{L}  \mathbb{P}  v =  \mathcal{D}_\gamma^{-1} \mathcal{L}   \frac{\la v , \mathcal{L} \partial_y^2 Q \ra }{\| \partial_y Q \|_{L^2}^2} \partial_y Q =  \mathcal{D}_\gamma^{-1} \frac{\la v , \mathcal{L} \partial_y^2 Q \ra }{\| \partial_y Q \|_{L^2}^2}  (\mathcal{L}   \partial_y Q ) = 0 , $$
due to the fact that $\mathcal{L}   \partial_y Q =0$, we have 
$$\partial_t \psi =   \mathcal{D}_\gamma^{-1} \mathcal{L} \partial_y \mathcal{L} v + \mathcal{D}_\gamma^{-1} \mathcal{L} \partial_y f$$
Plugging in $\mathcal{L} v = \mathcal{D}_\gamma \psi$, we obtain
\begin{equation}
\label{E:1}
\partial_t \psi =  \mathcal{D}_\gamma^{-1} \mathcal{L} \mathcal{D}_\gamma \partial_y \psi + \mathcal{D}_\gamma^{-1} \mathcal{L} \partial_y f
\end{equation}
The chain rule easily gives
$$\mathcal{D}_\gamma \mathcal{L} = -\gamma Q' + \mathcal{L}\mathcal{D}_\gamma$$
Applying $\mathcal{D}_\gamma^{-1}$ to the left side, we obtain
$$\mathcal{D}_\gamma^{-1} \mathcal{L} \mathcal{D}_\gamma = \mathcal{L} + \gamma\mathcal{D}_\gamma^{-1} Q'$$
(in the last term, it is meant the composition of operators, where $Q'$ is a multiplication operator.)  Plugging into \eqref{E:1},
$$\partial_t \psi =   \mathcal{L}\partial_y \psi + \gamma \mathcal{D}_\gamma^{-1}( Q' \partial_y \psi) + \mathcal{D}_\gamma^{-1} \mathcal{L}\partial_y f$$
Using that $Q'\psi_y = \partial_y (Q'\psi) - Q''\psi$, we obtain
\begin{equation}
\label{E:veqn}
\partial_t \psi =   \mathcal{L}\partial_y \psi + \gamma\partial_y \mathcal{D}_\gamma^{-1} ( Q' \psi) - \gamma \mathcal{D}_\gamma^{-1} (Q'' \psi)   + \mathcal{D}_\gamma^{-1} \mathcal{L}\partial_y f
\end{equation}
As discussed in \S\ref{S:spectral}, $\psi$ satisfies the orthogonality conditions \eqref{E:v-orth}, inherited from the orthogonality conditions \eqref{E:D1E} or \eqref{E:D1F} imposed on $v$.

Now we can appeal to Proposition \ref{P:local-virial2} to obtain \eqref{E:loc-vir-v}. 
To complete the proof, we claim that for $\gamma$ sufficiently small, we have
\begin{equation}
\label{E:cc16}
 \|\langle \gamma y \rangle^{-1}v \|_{L^2_x} \lesssim \| \langle \gamma y \rangle^{-1} \mathcal{D}^{-1}_\gamma \mathcal{L} v \|_{L^2_y} 
\end{equation}
The implicit constant is independent of $\gamma$.  

We will prove \eqref{E:cc16}  as a consequence of  the commutator estimate  \eqref{E:com1} (Lemma \ref{L:com6}) as follows.  From \eqref{E:com1}, there exists $C>0$ so that 
\begin{equation}
\label{E:com2}
\| \mathcal{L} \mathcal{D}_\gamma^{-1} \la \gamma y \ra^{-1} v \|_{L_y^2} \leq \| \la \gamma y\ra^{-1} \mathcal{D}_\gamma^{-1} \mathcal{L} v \|_{L_y^2} + C\gamma\ln \gamma^{-1} \| \la \gamma y \ra^{-1} v \|_{L_y^2}
\end{equation}
By the spectral estimate Proposition \ref{P:L2-spec-est}(2), we can take $C>0$ larger if necessary so that
$$C^{-1} \| \la D \ra  \mathcal{D}_\gamma^{-1}  \la \gamma y \ra^{-1} v \|_{L_y^2} \leq \| \mathcal{L} \mathcal{D}_\gamma^{-1} \la \gamma y \ra^{-1} v \|_{L_y^2}$$
Combining this with the uniform in $0<\gamma \leq 1$ lower bound $1\leq \la D \ra \mathcal{D}_\gamma^{-1}$, we obtain
$$C^{-1} \|  \la \gamma y \ra^{-1} v \|_{L_y^2} \leq \| \mathcal{L} \mathcal{D}_\gamma^{-1} \la \gamma y \ra^{-1} v \|_{L_y^2}$$
Appending this inequality on left of \eqref{E:com2}, we obtain obtain for $\gamma$ sufficiently small,
$$(C^{-1}-C\gamma\ln \gamma^{-1}) \|  \la \gamma y \ra^{-1} v \|_{L_y^2} \leq  \| \la \gamma y\ra^{-1} \mathcal{D}_\gamma^{-1} \mathcal{L} v \|_{L_y^2}$$
which implies \eqref{E:cc16} for $\gamma$ sufficiently small.
\end{proof}

\begin{proposition}[reduction to $y_0=0$] 
\label{P:step1}
There exists $0< \gamma_0 \ll 1$ such that for all $0<\gamma\leq \gamma_0$, for any time length $T>0$,  for any spatial center $y_0\in \mathbb{R}$, and for any solution $v$ to \eqref{E:D1C}, we have 
\begin{equation}
\label{E:cc15}
\| \la D\ra^{1/2} ((g_{\gamma,y_0}')^{1/2} v) \|_{L_{[0,T]}^2 L_y^2}^2 \lesssim 
\begin{aligned}[t]
&\gamma^{-1} \| v\|_{L_{[0,T]}^\infty L_y^2}^2 + \gamma^{-1} \| (g_{\gamma,0}')^{1/2} v\|_{L_{[0,T]}^2 L_y^2}^2 \\
&+ \int_0^T  \int g_{\gamma,y_0} \, v \, \partial_y f \, dy  \, dt
\end{aligned}
\end{equation}
where  the implicit constant is independent of $T$, $y_0$ and $\gamma$.
\end{proposition}

\begin{proof}
The proof is a direct virial type (positive commutator) calculation that does not use a spectral estimate (this is employed in Proposition \ref{P:local-virial2} below). In place of the spectral estimate, the term C below is crudely estimated and becomes the right-side term $ \gamma^{-1} \| (g_{\gamma,0}')^{1/2} v\|_{L_{[0,T]}^2 L_y^2}^2$ in \eqref{E:cc15}.   As technical tools, we do use the commutator estimates in Lemma \ref{L:gamma-comm} and Lemma \ref{L:com5}.

For arbitrary $y_0\in \mathbb{R}$,
$$\frac12 \partial_t  \int g_{\gamma,y_0} v^2 \, dy = \int g_{\gamma,y_0} \, v \, [ \mathbb{P} v + \partial_y (\mathcal{L} v + f) ]  \, dy$$
Expanding $\mathcal{L} = -H\partial_y + 1 - Q$, and integrating by parts, we obtain
$$
\frac12 \partial_t  \int g_{\gamma,y_0} v^2 \, dy = 
\begin{aligned}[t]
&\int ( g_{\gamma,y_0} \, v \, \mathbb{P} v  + g_{\gamma,y_0} \, v_y \, H v_y + g_{\gamma,y_0}' \, v \, Hv_y - \frac12 g_{\gamma,y_0}' \, v^2 \\
&\qquad + \frac12 g_{\gamma,y_0}'Q \, v^2 - \frac12 g_{\gamma,y_0}Q' \, v^2 + g_{\gamma,y_0} \, v \, \partial_y f) \, dy
\end{aligned}
$$

We rearrange the terms as
\begin{equation}
\label{E:cc11}
\begin{aligned}
\indentalign -\int g_{\gamma,y_0}' \, v \, Hv_y \, dy + \frac12 \int g_{\gamma,y_0}' v^2 \, dy \\
&= - \frac12 \partial_t \int g_{\gamma,y_0} v^2 \, dy  + \int  g_{\gamma,y_0} \, v \, \mathbb{P} v \, dy + \int g_{\gamma,y_0} \, v_y \, Hv_y \, dy \\
&\qquad + \frac12 \int ( g_{\gamma,y_0}' Q - g_{\gamma,y_0}Q') \, v^2 \, dy + \int g_{\gamma,y_0} \, v \, \partial_y f \, dy
\end{aligned}
\end{equation}

Let us examine the first term on the left in \eqref{E:cc11}.  Taking $z = (g_{\gamma,y_0}')^{1/2} v$ and $f_0 = (g'_{\gamma,y_0})^{1/2}$.  Then
\begin{align*}
-\int g_{\gamma,y_0}' v\, Hv_y \, dy &= \int f_0^2 v \, Dv \, dy = \int D(f_0^2 v) \, v \, dy = \int f_0^{-1}D(f_0 z) \, z \, dy \\
&= \int Dz \,z \, dy + \int f_0^{-1}( D(f_0 z) - f_0 Dz) \, z \, dy \\
&= \int (D^{1/2}z)^2\, dy + \int f_0^{-1}( D(f_0 z) - f_0 Dz) \, z \, dy
\end{align*}
Substituting into \eqref{E:cc11},
\begin{equation}
\label{E:cc12}
 \|D^{1/2}z\|_{L^2}^2 + \frac12 \|z\|_{L^2}^2 = - \frac12 \partial_t \int g_{\gamma,y_0} v^2 \, dy +  \text{A} + \text{B} + \text{C} + \text{D}  + \int g_{\gamma,y_0} \, v \, \partial_y f \, dy
\end{equation}
where
\begin{align*}
&\text{A} = \int g_{\gamma,y_0} \, v_y \, Hv_y \, dy \\
&\text{B} = - \int f_0^{-1}(D(f_0 z)-f_0 Dz) \, z \, dy \\
&\text{C} = \frac12 \int ( g_{\gamma,y_0}' Q - g_{\gamma,y_0}Q') \, v^2 \, dy \\
&\text{D} =  \int  g_{\gamma,y_0} \, v \, \mathbb{P} v \, dy 
\end{align*}
By \eqref{E:cc10} (Lemma \ref{L:com5}), we obtain $|\text{A}| \lesssim \gamma \|z\|_{L^2}^2$.  
By \eqref{E:cc0} (Lemma \ref{L:gamma-comm}), we obtain
$$|\text{B}| \leq \| f_0^{-1}(D(f_0 z) - f_0 Dz) \|_{L^2} \|z\|_{L^2} \lesssim \gamma^{3/4} \|\la D\ra^{1/2} z\|_{L^2}^2$$
Thus the terms A and B can be absorbed back on the left in \eqref{E:cc12} provided $\gamma$ is taken sufficiently small.
Using the pointwise bound, 
$$g_{\gamma,y_0}' Q - g_{\gamma,y_0}Q' \lesssim \gamma^{-1} \la y \ra^{-2}  \lesssim \gamma^{-1} g_{\gamma,0}'$$
we can bound
$$|\text{C}| \lesssim \gamma^{-1} \| (g_{\gamma,0}')^{1/2} v \|_{L^2}^2$$
Moreover, recalling \eqref{def:P} as well as the definition of $g_{\gamma,y_0}$, we can bound $\text{D}$ by 
\begin{equation}
\begin{split}
|\text{D}|
& \lesssim \gamma^{-1}   \int  | v | \, | \mathbb{P} v | \, dy  \\
& \lesssim \gamma^{-1}   \int  | v | \, |  \la v , \mathcal{L} \partial_y^2 Q \ra   \partial_y Q| \, dy  \\
& \lesssim \gamma^{-1}   \int  | \la y \ra^{-1} v | \, |  \la v , \mathcal{L} \partial_y^2 Q \ra |  \la y \ra^{-1}  \, dy  \\
& \lesssim \gamma^{-1} \| \la y \ra^{-1} v \|_{L^2} \| |  \la v , \mathcal{L} \partial_y^2 Q \ra |  \la y \ra^{-1} \|_{L^2} \\
& \lesssim \gamma^{-1} \| (g_{\gamma,0}')^{1/2} v \|_{L^2}^2 \\
\end{split}
\end{equation}

We integrate \eqref{E:cc12} in time to complete the proof.
\end{proof}

\begin{proposition}[estimate for $\psi$ with $y_0=0$]
\label{P:local-virial2}
Suppose $\psi$ solves \eqref{E:veqn} and satisfies the orthogonality conditions \eqref{E:v-orth}.  Then there exists $\gamma_0>0$ such that for all $0<\gamma\leq \gamma_0$,
\begin{equation}
\label{E:loc-vir-v}
\|\langle D_y\rangle^{1/2} ((g'_{\gamma,0})^{1/2} \psi) \|^2_{L^2_{[0,T]} L^2_y} \lesssim \gamma^{-1} \|\psi \|^2_{L^{\infty}_{[0,T]} L^2_y} +  \int^T_0 \int_y g_{\gamma,0} \, \psi \, \mathcal{D}^{-1}_{\gamma}  \mathcal{L}\partial_y f \, dy dt 
\end{equation}
\end{proposition}

\begin{proof}
The proof will employ the spectral estimate for $\tilde{\mathcal{L}}$ (Proposition \ref{P:tildeLspec}(2)), and as technical tools, we will use commutator estimates from Lemma \ref{L:weightedD}, Corollary \ref{C:com3}, and Lemma \ref{L:gamma-comm}.  

For this proof, we will take $g_\gamma = g_{\gamma,0}$, i.e. we set $y_0=0$.  
Let $I(t) = \int g_{\gamma} \psi^2$.
Then, substituting \eqref{E:veqn},
\begin{equation}
\label{E:I-prime}
\begin{aligned}
\frac12 I'(t) &= \int g_{\gamma} \psi \partial_t \psi \\
&= \int \psi g_\gamma \mathcal{L}(\psi_y) + \int \psi g_\gamma \gamma \partial_y \mathcal{D}^{-1}_\gamma (Q' \psi) \\
& \qquad - \int \psi g_\gamma \gamma \mathcal{D}^{-1}_\gamma (Q'' \psi) + \int \psi g_\gamma \mathcal{D}^{-1}_\gamma  \mathcal{L}(f_x) \\
&= \text{A}+\text{B}+\text{C}+\text{D}
\end{aligned}
\end{equation}

The term $\text{A} = \int \psi g_\gamma  \mathcal{L}(\psi_y)$ can be controlled by 
$$
\text{A} \leq -\frac{1}{2} (\tilde{ \mathcal{L}} ( (g'_\gamma)^{1/2} \psi), (g'_\gamma)^{1/2} \psi) + \gamma^\theta \| (g'_\gamma)^{1/2} \psi \|_{H^{1/2}_y}
$$
where $\theta >0$, and $\tilde{\mathcal{L}}$ is defined in \eqref{E:tildeLdef} and satisfies
$$ (\tilde{ \mathcal{L}}z,z) := 2\int |D^{1/2} z| ^2 +\int z^2 -\int (yQ'+Q)z^2, \qquad \forall z $$ 
Indeed, let $z = (g'_\gamma)^{1/2} \psi$, we have
\begin{equation}
\begin{split}
\int \psi g_\gamma  \mathcal{L}(\psi_y) 
& = \int  \psi g_\gamma (-H \psi_{yy} + \psi_y -Q \psi_y) \\
& = - \int ( g'_\gamma \psi +g_\gamma \psi_y   ) (- H \psi_y +v\psi ) + \frac{1}{2} \int ( g'_\gamma Q + g_\gamma Q'  ) \psi^2  \\
& = - \int |D^{1/2} z|^2  + \frac{1}{2} \int (y Q' + Q ) z^2  - \frac{1}{2} \int z^2 \\
& - \int \psi ( D ( z (g'_\gamma)^{1/2}) - (D z) (g'_\gamma)^{1/2}  )   +  \int H \psi_y \psi_y g_\gamma      + \frac{1}{2} \int (g_\gamma - y g'_\gamma) Q' \psi^2 \\
& := -\frac{1}{2} (\tilde{ \mathcal{L}}  z, z)  + A_1 + A_2 +A_3 
\end{split}
\end{equation}

For $A_1$, we can use Lemma \ref{L:gamma-comm} and obtain
\begin{equation}
\begin{split}
|A_1| 
& =  | \int z (g'_\gamma)^{-1/2} ( D (z (g'_\gamma)^{1/2}) - (g'_\gamma)^{1/2} (Dz)  )    | \\
& \leq \|z\|_{L^2}  \|  (g'_\gamma)^{-1/2} ( D (z (g'_\gamma)^{1/2}) - (g'_\gamma)^{1/2} (Dz)  ) \|_{L^2}  \\
& \leq \|z\|_{L^2} \gamma^{3/4} \|z\|_{H^{1/4}} \\
& \leq \gamma^{3/4}  \|z\|_{L^2} \|z\|_{H^{1/2}} \\
\end{split}
\end{equation}

For $A_2$, we can follow the approach in the proof of \cite[Lemma 4]{KM} and estimate
\begin{equation}
|A_2| \leq C \gamma  \| z \|^2_{L^2} 
\end{equation}

For $A_3$, we compute (using $ | \arctan y - \frac{y}{1+y^2} | \leq C y^3 $ in the case $|\gamma y|  \leq 1$, and $\la y \ra^{-2} \leq \gamma^2$ in the case  $|\gamma y|  > 1$)
\begin{equation}
\begin{split}
2 |A_3| 
& = | \int_{|\gamma y|  \leq 1} (g_\gamma - y g'_\gamma) Q' (g'_\gamma)^{-1} z^2 + \int_{\gamma y| > 1 } (g_\gamma - y g'_\gamma) Q'  (g'_\gamma)^{-1}  z^2 | \\
& \leq ( \sup_{|\gamma y| \leq 1} |(g_\gamma - y g'_\gamma) Q' (g'_\gamma)^{-1}| + \sup_{|\gamma y|>1} |(g_\gamma - y g'_\gamma) Q' (g'_\gamma)^{-1}| ) \|z\|^2_{L^2} \\
& \leq ( \sup_{|\gamma y|  \leq 1} | \gamma^{-1} \gamma^3 y^3 \la y \ra^{-3} \la \gamma y \ra^2  | + \sup_{|\gamma y| > 1} |( \gamma^{-1} \arctan (\gamma y) \la \gamma y \ra^2 -y  ) \la y \ra^{-3}| ) \|z\|^2_{L^2} \\
& \leq \gamma^2 \|z\|^2_{L^2} 
\end{split}
\end{equation}

Combining the above we obtain
\begin{equation}
|\int \psi g_\gamma  \mathcal{L}(\psi_y)  | \leq -\frac{1}{2} (\tilde{ \mathcal{L}}  z, z)  +  \gamma^\theta \| z \|^2_{H^{1/2}_y}
\end{equation}

We now estimate the term $\text{B} = \int \psi g_\gamma \gamma \partial_y \mathcal{D}^{-1}_\gamma (Q' \psi)$ in \eqref{E:I-prime}.   By applying $\mathcal{D}_\gamma^{-1}$ to both sides of the identity $\mathcal{D}_\gamma f = f \mathcal{D}_\gamma + \gamma f'$, we obtain the commutator identity $f \mathcal{D}_\gamma^{-1} = \mathcal{D}_\gamma^{-1} f + \gamma \mathcal{D}_\gamma^{-1} f' \mathcal{D}_\gamma^{-1}$.  Applying $\partial_y$ to the left side, we obtain $\partial_y \mathcal{D}_\gamma^{-1} f = \partial_y f \mathcal{D}_\gamma^{-1} - \gamma \partial_y \mathcal{D}_\gamma^{-1} f' \mathcal{D}_\gamma^{-1}$.  In the first term, we use $\partial_y f = f\partial_y + f'$ and in the second term, we use $\gamma \partial_y \mathcal{D}_\gamma^{-1} = 1 - \mathcal{D}_\gamma^{-1}$.  Substituting yields $\partial_y \mathcal{D}_\gamma^{-1} f = f \partial_y \mathcal{D}_\gamma^{-1} + \mathcal{D}_\gamma^{-1} f' \mathcal{D}_\gamma^{-1}$.  Applying this with $f= Q'(g_\gamma')^{-1/2}$, 
\begin{align*}
\psi \, \gamma g_\gamma \, \partial_y \mathcal{D}^{-1}_\gamma Q' \psi 
&=  \psi  \, \gamma g_\gamma \, Q'(g_\gamma')^{-1/2} \, \partial_y \mathcal{D}_\gamma^{-1} \, (g_\gamma')^{1/2} \psi  \\
& \qquad +  \psi \, \gamma g_\gamma \, \mathcal{D}_\gamma^{-1} \, [Q'(g_\gamma')^{-1/2}]' \, \mathcal{D}_\gamma^{-1} (g_\gamma')^{1/2} \psi
\end{align*}
On the left, in both terms, we replace $\psi= \psi (g_\gamma')^{1/2} (g_\gamma')^{-1/2}$ to obtain 
\begin{equation}
\label{E:com20}
\begin{aligned}
\psi \, \gamma g_\gamma \, \partial_y \mathcal{D}^{-1}_\gamma Q' \psi 
&=  \psi (g_\gamma')^{1/2} \, \gamma g_\gamma \, Q'(g_\gamma')^{-1} \, \partial_y \mathcal{D}_\gamma^{-1} \, (g_\gamma')^{1/2} \psi  \\
& \qquad +  \psi (g_\gamma')^{1/2} \, (g_\gamma')^{-1/2}  \gamma g_\gamma \, \mathcal{D}_\gamma^{-1} \, [Q'(g_\gamma')^{-1/2}]' \, \mathcal{D}_\gamma^{-1} (g_\gamma')^{1/2} \psi
\end{aligned}
\end{equation}

After integration, we estimate the second term as follows
$$\underbrace{\psi (g_\gamma')^{1/2} }_{L^2} \, \underbrace{(g_\gamma')^{-1/2} \gamma g_\gamma \la y \ra^{-2}}_{L^\infty} \underbrace{\la y \ra^2 \mathcal{D}_\gamma^{-1} \la y \ra^{-2}}_{L^2\to L^2}  \underbrace{\la y \ra^2 [Q'(g_\gamma')^{-1/2}]'}_{L^\infty} \, \underbrace{\mathcal{D}_\gamma^{-1}}_{L^2\to L^2} \underbrace{(g_\gamma')^{1/2} v}_{L^2}$$
where, importantly, $\| (g_\gamma')^{-1/2} \gamma g_\gamma \la y \ra^{-2}\|_{L^\infty} \leq \gamma$, from the estimate $|\gamma g_\gamma(y)| \leq \min(\gamma |y|, \frac{\pi}{2})$.  The $L^2\to L^2$ boundedness of $\la y\ra^2 \mathcal{D}_\gamma^{-1} \la y \ra^{-2}$ (uniformly in $\gamma$) was established in Lemma \ref{L:weightedD}.  This produces the bound $\gamma \| \psi (g_\gamma')^{1/2} \|_{L^2}^2$.

Returning to \eqref{E:com20}, this leaves us to estimate
$$\int \psi (g_\gamma')^{1/2} \, \gamma g_\gamma \, Q'(g_\gamma')^{-1} \, \partial_y \mathcal{D}_\gamma^{-1} \, (g_\gamma')^{1/2} \psi  \, dx$$
Replacing $\partial_y = D^{1/2} H D^{1/2}$, we obtain
$$=\int \psi (g_\gamma')^{1/2} \, \gamma g_\gamma \, Q'(g_\gamma')^{-1} \, D^{1/2} H \mathcal{D}_\gamma^{-1} D^{1/2} \, (g_\gamma')^{1/2} \psi  \, dx$$
We view this integral as the inner product of  $\psi (g_\gamma')^{1/2} \, \gamma g_\gamma \, Q'(g_\gamma')^{-1}$ and $D^{1/2} H \mathcal{D}_\gamma^{-1} D^{1/2} \, (g_\gamma')^{1/2} \psi$, and use that $D^{1/2}$ is self adjoint to obtain
$$= \int [D^{1/2}  \psi (g_\gamma')^{1/2} \, \gamma g_\gamma \, Q'(g_\gamma')^{-1}] \cdot  [H \mathcal{D}_\gamma^{-1} D^{1/2} \, (g_\gamma')^{1/2} \psi ]  \, dx$$
By Cauchy-Schwarz,
$$\leq \| D^{1/2} [ \psi (g_\gamma')^{1/2} \, \gamma g_\gamma \, Q'(g_\gamma')^{-1}]\|_{L^2} \| H \mathcal{D}_\gamma^{-1} D^{1/2}[ (g_\gamma')^{1/2} \psi ]\|_{L^2}$$
For the first of these terms, we apply \eqref{E:Leib1} (Corollary \ref{C:com3}) with $f= \gamma g_\gamma \, Q'(g_\gamma')^{-1}$ and 
 $h=\psi (g_\gamma')^{1/2}$, and for the second of these terms, we just use that $H$ and $\mathcal{D}_\gamma^{-1}$ are $L^2\to L^2$ bounded with operator norm independent of $\gamma$, to obtain
 $$\lesssim \gamma^{1/3}(\| \psi (g_\gamma')^{1/2}\|_{L^2} + \| D^{1/2}[ \psi (g_\gamma')^{1/2}] \|_{L^2})\|D^{1/2}[ \psi (g_\gamma')^{1/2}] \|_{L^2} $$
Here, we have use that with $f= \gamma g_\gamma \, Q'(g_\gamma')^{-1}$, we have
 $$\|f\|_{L^2} \lesssim \gamma^{1/3}\,, \qquad   \|\partial_x f\|_{L^2}\lesssim \gamma$$
These estimates come from the bound $|\gamma g_\gamma(y)| \leq \min(\gamma |y|, \frac{\pi}{2})$, which implies $|f(y)|\lesssim \min(\gamma|y|, 1) \la y \ra^{-1}$ and $|f'(y)| \lesssim \gamma \la y \ra^{-1}$.  To see that $\|f\|_{L^2} \lesssim \gamma^{1/3}$, we divide the integration into $|y|<\gamma^{-2/3}$ and $|y|> \gamma^{-2/3}$.  For the region $|y|<\gamma^{-2/3}$, we use that $|f(y)|\leq \gamma^{1/3}\la y \ra^{-1}$ and for the region $|y|>\gamma^{-2/3}$, we use that $|f(y)| \leq \la y \ra^{-1}$.  
 
 In summary, we have obtained that
 $$|\text{B}| \lesssim \gamma^{1/3}  \| \la D_y\ra ^{1/2}[\psi (g_\gamma')^{1/2}] \|_{L^2}^2$$

Finally, we turn to term $C = - \int \psi \, \gamma g_\gamma \, \mathcal{D}_\gamma^{-1} Q'' \psi  \, dy$ in \eqref{E:I-prime}.  Rewrite the integrand as follows
$$ \psi \, \gamma g_\gamma \, \mathcal{D}_\gamma^{-1} \, Q'' \psi 
= \psi (g_\gamma')^{1/2} \, (g_\gamma')^{-1/2} \gamma g_\gamma \la y \ra^{-2} \, \la y \ra^2 \mathcal{D}_\gamma^{-1} \la y \ra^{-2} \, \la y \ra^2 Q'' (g_\gamma')^{-1/2} \, (g_\gamma')^{1/2} \psi $$
In the integral, we estimate follows
$$\underbrace{ \psi (g_\gamma')^{1/2}}_{L^2} \, \underbrace{(g_\gamma')^{-1/2} \gamma g_\gamma \la y \ra^{-2}}_{L^\infty} \, \underbrace{\la y \ra^2 \mathcal{D}_\gamma^{-1} \la y \ra^{-2}}_{L^2\to L^2} \, \underbrace{\la y \ra^2 Q'' (g_\gamma')^{-1/2}}_{L^\infty} \, \underbrace{(g_\gamma')^{1/2} \psi }_{L^2}$$
Since $\| (g_\gamma')^{-1/2} \gamma g_\gamma \la y \ra^{-2}\|_{L^\infty} \leq \gamma$, we obtain
$$|\text{C}| \lesssim \gamma \| (g_\gamma')^{1/2} \psi \|_{L^2}^2$$

Combining the above upper bounds for $\text{A}$, $\text{B}$, and $\text{C}$, we obtain from \eqref{E:I-prime} that there exists $C>0$ independent of $\gamma$ such that
$$I'(t) \leq -\frac12 \la \tilde{\mathcal{L}} (g_\gamma')^{1/2} \psi , (g_\gamma')^{1/2} \psi \ra + C\gamma^\theta \| \la D_y \ra^{1/2} (g_\gamma')^{1/2} \psi \|_{L_y^2}^2 + \int_y \, g_\gamma \, \psi \, \mathcal{D}_\gamma^{-1} \mathcal{L} \partial_y f \, dx$$
Rearranging terms and applying the spectral estimate Proposition \ref{P:tildeLspec}(2), and possibly making $C$ larger (but still independent of $\gamma>0$)
$$C^{-1} \| \la D_y \ra^{1/2} (g_\gamma')^{1/2} \psi \|_{L_y^2}^2 \leq -I'(t) + C \gamma^\theta \| \la D_y \ra^{1/2} (g_\gamma')^{1/2} \psi \|_{L_y^2}^2 + \int_y \, g_\gamma \, \psi \, \mathcal{D}_\gamma^{-1} \mathcal{L} \partial_y f \, dy$$
Taking $0<\gamma\leq \gamma_0$, where $\gamma_0$ is defined by $C\gamma_0^\theta \leq \frac12 C^{-1}$, integrating on $0\leq t \leq T$ and using that $|I(t)| \leq \gamma^{-1} \| \psi \|_{L_T^\infty L_y^2}^2$ for all $0\leq t\leq T$, we obtain \eqref{E:loc-vir-v}, completing the proof. 
\end{proof}

\section{Application of the local virial inequality to (pBO)} 
\label{S:pBO}

Now suppose that $u(x,t)$ satisfies (pBO).  Define the remainder $\zeta$ according to 
\begin{equation}
\label{E:decomp} 
u = Q_{\am, \cm} + \zeta
\end{equation}
imposing orthogonality conditions
\begin{equation}
\label{E:orth}
 \la \zeta, Q_{\am,\cm} \ra =0 \,, \qquad \la \zeta, \partial_x Q_{\am,\cm} \ra = 0
\end{equation}
An implicit function theorem argument shows that there exists a unique choice of $(\am,\cm)$ so that these orthogonality conditions hold.  This is the \emph{definition} of the parameters $(\am (t),\cm(t))$ and of the remainder $\zeta$.  The goal of this section is to prove the following:

\begin{proposition}[nonsymplectic decomposition estimates for (pBO)]
\label{P:nonsymp-estimates}
There exists $\kappa \geq 1$, $\mu>0$, and $0<h_0\ll 1$ such that the following holds.
Let $0< h \leq h_0$ and suppose the initial data $u_0\in H_x^1$ satisfies
$$\| u_0(x) - Q_{0,1}(x) \|_{H_x^{1/2}} \leq h^{3/2}$$
Suppose that $u$ satisfying (pBO) with initial condition $u(x,0)=u_0(x)$ is decomposed as \eqref{E:decomp} with remainder $\zeta$ satisfying orthogonality conditions \eqref{E:orth}.    For every $T>0$ such that $\frac12 \leq \cm(t) \leq 2$ for all $0\leq t \leq T$, 
we have that the recentered remainder $v(y,t) = \zeta(y+\am(t),t)$ satisfies
\begin{equation} \label{E:nonsymp-estimates-eq1}
\|v\|_{L_{[0,T]}^\infty H_y^{1/2}} + \sup_n \|v\|_{L_{[0,T]}^2L_{y\in(n,n+1)}^2} \leq \kappa h^{3/2}e^{\mu hT}
\end{equation}
and the parameters $\am(t)$, $\cm(t)$ satisfy the bounds \eqref{E:par2} below.
\end{proposition}

Starting with $\partial_t u = JE'(u)$, we substitute \eqref{E:decomp} to obtain
$$\partial_t (Q_{\am,\cm} + \zeta) = J E'( Q_{\am,\cm} + \zeta)$$
Using expansions
\begin{itemize}
\item $\partial_t Q_{\am,\cm} = \dot \am \partial_\am Q_{\am,\cm} + \dot \cm \partial_\cm Q_{\am,\cm}$
\item $E'(u) = -H\partial_x u - \frac12 u^2 + Vu$
\item $E''(u) = -H\partial_x  - u +V$
\end{itemize}
we obtain the \emph{equation for the remainder} $\zeta$
\begin{equation}
\label{E:eta}
\partial_t \zeta = - \dot \am \partial_\am Q_{\am,\cm} - \dot \cm \partial_\cm Q_{\am,\cm} + J E'(Q_{\am,\cm}) + JE''(Q_{\am,\cm}) \zeta  - \frac12 \partial_x (\zeta^2)
\end{equation}
The soliton part on the right side is simplified as
\begin{align*}
J E'(Q_{\am,\cm}) &= \partial_x (- H\partial_x Q_{\am,\cm} - \frac12Q_{\am,\cm}^2 + W(hx) Q_{\am,\cm}) \\
&= \partial_x ( -\cm Q_{\am,\cm} + W(hx) Q_{\am,\cm})
\end{align*}
We Taylor expand $W(hx)$ around $x=\am$ to obtain
\begin{align*}
W(hx) &= W(h\am) + hW'(h\am)(x-\am) + e_2(x,\am)
\end{align*}
Recall 
$$\partial_\am Q_{\am,\cm} = - \partial_x Q_{\am,\cm} \,, \qquad \partial_\cm Q_{\am,\cm} = \cm^{-1} \partial_x[ (x-\am) Q_{\am,\cm} ]$$
Substituting into \eqref{E:eta},
$$\partial_t \zeta = 
\begin{aligned}[t]
&(\dot \am - \cm + W(h\am) )\partial_x Q_{\am,\cm}   + (-\dot \cm \cm^{-1} + hW'(h\am) )\partial_x[(x-\am)Q_{\am,\cm}]   \\
& + \partial_x(e_2 Q_{\am,\cm}) + JE''(Q_{\am,\cm}) \zeta  - \frac12 \partial_x (\zeta^2)  
\end{aligned}
$$
We re-center the equation for $\zeta$ by letting 
$$v (y) = \zeta (y+\am) \quad \iff \quad \zeta (x) = v(x-\am)$$ 
Notice that
$$\partial_t \zeta =  - \dot \am  \partial_y v  + \partial_t v \,, \qquad E''(Q_{\am,\cm}) \zeta = (\mathcal{L}_\cm - \cm + W(hx)) v $$
The orthogonality conditions on $v$ read
\begin{equation}
\label{E:v-ortho}
\la v, Q_\cm \ra =0 \,, \qquad \la v, \partial_y Q_\cm \ra =0 .
\end{equation}
The equation for $v$ is
\begin{equation}
\label{E:v}
\begin{aligned}
\partial_t v = \;
&(\dot\am-\cm+W(h\am))\partial_yQ_\cm +(-\dot\cm \cm^{-1} + hW'(h\am))\partial_y(yQ_\cm) + \partial_y(e_2Q_\cm) \\
&+\partial_y \mathcal{L}_\cm v + \partial_y(\dot\am - \cm + W(hx)) v - \frac12 \partial_y v^2
\end{aligned}
\end{equation}

\begin{lemma}[nonsymplectic parameter control]
\label{L:nonsymp-ODE-control}
For all $t$, if $\frac12 \leq \cm \leq 2$ and $\|v \|_{L_y^2} \ll 1$, then 
\begin{equation}
\label{E:par2}
\begin{aligned}
& |\dot\am - \cm + W(h\am) - \tfrac12 h^2W''(h\am) \cm^{-1} - \frac{1}{4\pi} \cm^{-3} \la v, \mathcal{L}_c \partial_y^2Q_\cm\ra | \lesssim h^4 + \sup_{n\in \mathbb{Z}} \|v\|_{L_{n<y < n+1}^2}^2 \\
&|\dot\cm  - hW'(h\am)\cm - \tfrac12 h^3 W''(h\am)\cm^{-1}| \lesssim h^4 + h^2(\ln h^{-1})\sup_{n\in \mathbb{Z}} \|v\|_{L_{n<y < n+1}^2} + \|v\la y \ra^{-1} \|_{L_y^2}^2
\end{aligned}
\end{equation}
Moreover, for any time interval $I$,
\begin{equation}
\label{E:par2t}
\begin{aligned}
& \int_I |\dot\am - \cm + W(h\am) - \tfrac12 h^2W''(h\am) \cm^{-1} - \frac{1}{4\pi} \cm^{-3} \la v, \mathcal{L}_c \partial_y^2Q_\cm\ra |\, dt \lesssim h^4|I| + \sup_{n\in \mathbb{Z}} \|v\|_{L_I^2L_{n<y < n+1}^2}^2 \\
&\int_I |\dot\cm  - hW'(h\am)\cm - \tfrac12 h^3 W''(h\am)\cm^{-1}|\, dt \lesssim h^4|I| + h^2(\ln h^{-1})|I|^{1/2} \sup_{n\in \mathbb{Z}} \|v\|_{L_I^2 L_{n<y < n+1}^2} + \|v\la y \ra^{-1} \|_{L_I^2 L_y^2}^2
\end{aligned}
\end{equation}
In particular, we have the following weaker formulation, needed in subsequent lemmas.  Let $E_\am$ and $E_\cm$ denote the following trajectory equation remainders:
\begin{equation}
\begin{aligned}
E_\am &= \dot\am - \cm + W(h\am)  - \frac{1}{4\pi} \cm^{-3} \la v, \mathcal{L}_c \partial_y^2Q_\cm\ra \\
E_\cm &=  \dot\cm  - hW'(h\am)\cm
\end{aligned}
\end{equation}
Then the following estimates for $E_\am$ and $E_\cm$ hold:
\begin{equation}
\label{E:est-52}
|E_\am| \lesssim h^2 + \|v\la y \ra^{-1}\|_{L_y^2}^2 \,, \qquad |E_\cm| \lesssim h^3 + \|v\la y \ra^{-1} \|_{L_y^2}^2
\end{equation}
\end{lemma}

\begin{proof}
Taking $\partial_t$ of the orthogonality condition $\la v, Q_\cm \ra =0 $, we obtain
$$ 0 = \la \partial_t v, Q_\cm \ra + \dot\cm \cm^{-1} \la v, \partial_y(yQ_\cm) \ra$$
where we have used that $\partial_t Q_\cm = \dot \cm \partial_\cm Q_\cm = \dot \cm \cm^{-1} \partial_y(yQ_\cm)$.  Substituting \eqref{E:v}, we obtain
\begin{equation}
\label{E:par1}
\begin{aligned}
0 = \;
& (\dot\am-\cm+W(h\am)) \la \partial_yQ_\cm, Q_\cm \ra && \leftarrow \text{I} \\
& +(-\dot\cm \cm^{-1} + hW'(h\am)) \la \partial_y(yQ_\cm) , Q_\cm \ra && \leftarrow \text{II}\\
& + \la \partial_y(e_2Q_\cm) , Q_\cm \ra && \leftarrow \text{III} \\
& + \la \partial_y \mathcal{L}_\cm v , Q_\cm \ra && \leftarrow \text{IV} \\
& + \la \partial_y(\dot\am - \cm + V) v , Q_\cm \ra && \leftarrow \text{V}\\
&- \tfrac12 \la \partial_y v^2, Q_\cm \ra && \leftarrow \text{VI} \\
& + \dot\cm\cm^{-1}\la v, \partial_y(yQ_\cm) \ra && \leftarrow \text{VII}
\end{aligned}
\end{equation}

Since $\la \partial_yQ_\cm, Q_\cm\ra =0$, we conclude that $\text{I}=0$.  Using that $\la \partial_y(yQ_\cm) , Q_\cm \ra = - \la yQ_\cm, \partial_y Q_\cm \ra = - \frac12 \int y \partial_y Q_\cm^2 = \frac12 \int Q_\cm^2 = 4\pi \cm$, we obtain that
$$\text{II} = 4\pi (-\dot\cm  + hW'(h\am)\cm)$$
Via integration by parts, $\text{III}$ simplifies to
\begin{equation}
\label{E:term3a}
\text{III} = \frac12 \int (\partial_y e_2) Q_\cm^2
\end{equation}
Since
$$e_2(y,\am) = W(h(y+\am)) - W(h\am) - hW'(h\am)y$$
we have
\begin{equation}
\label{E:term3b}
\begin{aligned}
\partial_y e_2(y,\am) &= h W'(h(y+\am)) - hW'(h\am) \\
&=h^2W''(h(y_*+\am))y 
\end{aligned}
\end{equation}
for some $y_*$ between $0$ and $y$ by the mean-value theorem.   We could also carry the expansion out to fifth order
\begin{equation}
\label{E:term3c}
\begin{aligned}
\partial_y e_2(y,\am)
&=h^2W''(h\am)y + \frac12h^3W'''(h\am)y^2 \\
& \qquad + \frac16 h^4 W^{(4)}(h\am) y^3 + \frac1{24} h^5 W^{(5)}(h(y_*+\am)) y^4
\end{aligned}
\end{equation}
for some $y_*$ between $0$ and $y$, by the Lagrange form of the remainder in Taylor's theorem. 

Divide the integration in \eqref{E:term3a} into the two regions $|y|<h^{-1}$ and $|y|>h^{-1}$, producing the two terms $\text{III}_{\text{in}}$ and $\text{III}_{\text{out}}$.  Plugging \eqref{E:term3c} into \eqref{E:term3a} to compute $\text{III}_{\text{in}}$, we obtain
\begin{align*}
\text{III}_{\text{in}} &= \frac12h^2W''(h\am) \int_{|y|< h^{-1}} y Q_\cm^2 \, dy + \frac14h^3 W'''(h\am)\int_{|y|<h^{-1}}  y^2 Q_\cm^2 \, dy \\
&\qquad + \frac{1}{12} h^4 W^{(4)}(h\am) \int_{|y|<h^{-1}}  y^3 Q_\cm^2 \, dy + \frac{1}{48} h^5 \int_{|y|<h^{-1}} W^{(5)}(h(y_*+\am)) y^4 Q_\cm^2 \, dy
\end{align*}
The first and third integrals are zero (they are integrals of odd functions) and the fifth integral returns $O(h^{-1})$ since the integrand is uniformly bounded.  Thus
$$
\text{III}_{\text{in}} =  \frac14h^3 W'''(h\am)\int_{|y|<h^{-1}}  y^2 Q_\cm^2 \, dy +O(h^4)$$
But
\begin{align*}
\int_{|y|<h^{-1}} y^2Q_\cm^2 \, dy &= \int y^2 Q_\cm^2 \,dy - \int_{|y|>h^{-1}} y^2 Q_\cm^2 \, dy \\
& = c^{-1} \int y^2 Q(y)^2 \, dy + O(h^{-1}) = 8\pi c^{-1} + O(h^{-1})
\end{align*}
Thus 
$$
\text{III}_{\text{in}} =  2\pi h^3 W'''(h\am)c^{-1} +O(h^4)$$

On the other hand, plugging \eqref{E:term3b} into \eqref{E:term3a} to compute $\text{III}_{\text{out}}$, we obtain
$$\text{III}_{\text{out}} = \frac12 h^2 \int_{|y|>h^{-1}} W''(h(y_*+\am)) y Q_\cm^2 \, dy = O(h^4)$$
where we used that $W''$ is bounded.  Consequently
$$\text{III} = \text{III}_{\text{in}} + \text{III}_{\text{out}} = 2\pi h^3 W'''(h\am)c^{-1} +O(h^4)$$

Since $\mathcal{L}_\cm \partial_y Q_\cm =0$, we conclude that $\text{IV}=0$.  

Using the expansion $W(h(y+\am)) = W(h\am) + hW'(h\am)y + e_2(y,a)$, we have
\begin{align*}
\text{V} &= \la \partial_y(\dot \am - \cm +W(h(y+\am)))v, Q_\cm \ra \\
&= (\dot \am - \cm + W(h\am)) \la \partial_y v, Q_\cm \ra + hW'(h\am) \la  v, Q_\cm \ra  \\
& \qquad + hW'(h\am) \la y \partial_y v, Q_\cm \ra + \la \partial_y(e_2v), Q_\cm \ra
\end{align*}
By the orthogonality conditions \eqref{E:v-ortho}, the first two terms drop away, leaving
$$\text{V} = hW'(h\am) \la y \partial_y v, Q_\cm \ra + \la \partial_y(e_2v), Q_\cm \ra $$
Combining with Term VII,
$$\text{V} + \text{VII} = (\dot\cm\cm^{-1} - hW'(h\am)) \la v, \partial_y(yQ_\cm) \ra + \la \partial_y(e_2v), Q_\cm \ra $$
Once again, by Taylor's theorem with the Lagrange form of the remainder
$$e_2(y,\am) = \frac12 h^2 W''(h(y_*+\am))y^2$$
Let $R>0$ such that $\supp W \subset [-R,R]$.  Then $-\am - Rh^{-1} \leq y \leq -\am + Rh^{-1}$.  This gives
\begin{align*}
 \la \partial_y(e_2v), Q_\cm \ra &= - \int_{-\am-Rh^{-1}}^{\am+Rh^{-1}} v e_2\partial_y Q_\cm \, dy = - \frac12 h^2 \int_{-\am-Rh^{-1}}^{\am+Rh^{-1}} v W''(h(y_*+\am)) y^2 \partial_y Q_\cm \, dy \\
 &= - \frac12 h^2 \sum_{n\in \mathbb{Z}} \int_{-\am-Rh^{-1}}^{\am+Rh^{-1}} \mathbf{1}_{[n,n+1]} \, v \, W''(h(y_*+\am)) \,  \, y^2 \partial_y Q_\cm \, dy
\end{align*}
Thus
\begin{equation}
\label{E:e2est01}
\begin{aligned}
|\la \partial_y(e_2v), Q_\cm \ra| &\lesssim h^2 \sum_{n\in \mathbb{Z}} \left( \int_n^{n+1} v^2 \, dy \right)^{1/2} \left( \int_n^{n+1} \mathbf{1}_{[-\am-Rh^{-1}, \am+Rh^{-1}]}(y) \, y^2 \partial_y Q_\cm(y) \, dy \right)^{1/2} \\
&\lesssim h^2 \sum_{n\in \mathbb{Z}} \left( \int_n^{n+1} v^2 \, dy \right)^{1/2}  \frac{1}{\la n \ra} \mathbf{1}_{[-\am - Rh^{-1}-1, -\am + Rh^{-1}+1]}(n) \\
&\lesssim h^2 (\ln h^{-1}) \sup_n \|v\|_{L_{n\leq y\leq n+1}^2}
\end{aligned}
\end{equation}
Moreover, integrating over a time interval $I$,
\begin{align*}
\int_I |\la \partial_y(e_2v), Q_\cm \ra|  \, dt &\lesssim h^2 \sum_{n\in \mathbb{Z}} \int_I \left( \int_n^{n+1} v^2 \, dy \right)^{1/2} \left( \int_n^{n+1} \mathbf{1}_{[-\am-Rh^{-1}, \am+Rh^{-1}]}(y) \, y^2 \partial_y Q_\cm(y) \, dy \right)^{1/2} \,d t
\end{align*}
Applying Cauchy-Schwarz in $t$,
\begin{equation}
\label{E:e2est02}
\begin{aligned}
&\lesssim h^2 \sum_{n\in \mathbb{Z}} |I|^{1/2} \left( \int_I \int_n^{n+1} v^2 \, dy \right)^{1/2}  \frac{1}{\la n \ra} \mathbf{1}_{[-\am - Rh^{-1}-1, -\am + Rh^{-1}+1]}(n) \\
&\lesssim h^2|I|^{1/2} (\ln h^{-1}) \sup_n \|v\|_{L_I^2L_{n\leq y\leq n+1}^2}
\end{aligned}
\end{equation}

And finally 
$$\text{VI} = -\tfrac12 \la \partial_y v^2, Q_\cm\ra = \tfrac12 \la v^2, \partial_y Q_\cm \ra \lesssim \|\la y \ra^{-1} v\|_{L_y^2}^2$$
Collecting the estimates and identities above, we obtain that \eqref{E:par1} yields 
$$|\dot\cm  - hW'(h\am)\cm  - \frac12 h^3 W''(h\am)\cm^{-1}| (4\pi - \cm^{-1}\la v, \partial_y(yQ_\cm)\ra) $$
$$\lesssim h^4 +  h^2 (\ln h^{-1}) \sup_n \|v\|_{L_{n\leq y\leq n+1}^2}+ \|\la y \ra^{-1} v\|_{L_y^2}^2$$
from which the second inequality in \eqref{E:par2} follows.  The second inequality in \eqref{E:par2t} follows in the same way but using \eqref{E:e2est02} in place of \eqref{E:e2est01}.

Now, by similar methods, we prove the first inequality in \eqref{E:par2}.
Taking $\partial_t$ of the orthogonality condition $0=\la v, \partial_y Q_\cm \ra$, we obtain
$$0 = \la \partial_t v, \partial_y Q_\cm \ra + \la v, \partial_t \partial_yQ_\cm \ra$$
For the first term, we substitute \eqref{E:v}, and for second term, we use that $\partial_t\partial_yQ_\cm = \dot\cm \cm^{-1} \partial_y^2(y Q_\cm)$, to obtain
\begin{align*}
0 &= ( \dot\am - \cm + W(h\am)) \la \partial_yQ_\cm, \partial_yQ_\cm\ra && \leftarrow \text{I}\\
&\quad +(-\dot \cm \cm^{-1} +hW'(h\am)) \la \partial_y(yQ_\cm),\partial_yQ_\cm \ra  && \leftarrow \text{II}\\
&\quad + \la \partial_y(e_2Q_\cm), \partial_yQ_\cm \ra && \leftarrow \text{III}\\
&\quad + \la \partial_y\mathcal{L}_\cm v, \partial_yQ_\cm \ra && \leftarrow \text{IV}\\
&\quad + \la \partial_y(\dot \am - \cm + W(h(y+\am)))v, \partial_y Q_\cm \ra && \leftarrow \text{V}\\
&\quad - \tfrac12 \la \partial_y v^2, \partial_y Q_\cm \ra && \leftarrow \text{VI}\\
&\quad + \dot \cm \cm^{-1} \la v, \partial_y^2(yQ_\cm) \ra && \leftarrow \text{VII}
\end{align*}

Given that $\|\partial_yQ_\cm\|_{L^2}^2 = 4\pi \cm^3$, we have
$$\text{I} = 4\pi \cm^3( \dot\am - \cm + W(h\am))$$
Also, given that $\partial_y(yQ_\cm)$ is even and $\partial_yQ_\cm$ is odd, we have $\la \partial_y(yQ_\cm),\partial_yQ_\cm \ra =0$ and thus $\text{II}=0$.  

To address term III, we carry out the Taylor expansion 
\begin{align*}
e_2 &= W(h(y+\am)) - W(h\am) - hW'(h\am) y \\
&= \frac12 h^2W''(h\am)y^2 + \frac16 h^3W'''(h\am) y^3 + \frac1{24}h^4W''''(h(y_*+\am))y^4
\end{align*}
 for some $y_*$ between $0$ and $y$, by the Lagrange form of the remainder.  Substituting
\begin{align*}
\text{III} &= -\la e_2, Q_\cm \partial_y^2Q_\cm \ra \\
&= - \frac12 h^2 W''(h\am) \int y^2 Q_\cm \partial_y^2Q_\cm \, dy - \frac16h^3 W'''(h\am) \int y^3 Q_\cm \partial_y^2Q_\cm \, dy \\
& \qquad - \frac1{24}h^4W''''(h\am) \int y^4 Q_\cm \partial_y^2Q_\cm 
\end{align*}
Since $\int z^2 Q(z) Q''(z) \, dz = 4\pi$,
$$\text{III} = - 2\pi h^2W''(h\am)\cm + O(h^4)$$

Now, unlike the previous calculation, the contribution from Term IV does not drop out:
$$\text{IV} = - \la v, \mathcal{L}_\cm \partial_y^2Q_\cm \ra$$

In Term V, we expand 
$$W(h(y+\am))= W(h\am) + hW'(h\am)y + e_2(y,\am)$$
to yield
$$\text{V} = (\dot \am - \cm + W(h\am))\la \partial_y v, \partial_yQ_\cm\ra - hW'(h\am) \la v, y \partial_y^2 Q_\cm \ra - \la e_2v, \partial_y^2 Q_\cm \ra$$
In the second (middle) of these term, we use the operator commutator identity $y\partial_y^2 = \partial_y^2 y - 2\partial_y$ and the orthogonality condition $\la v, \partial_yQ_\cm \ra =0$ to obtain
$$\text{V} = (\dot \am - \cm + W(h\am))\la \partial_y v, \partial_yQ_\cm\ra - hW'(h\am) \la v, \partial_y^2( y Q_\cm) \ra - \la e_2v, \partial_y^2 Q_\cm \ra$$
This allows a combination with Term VII
$$\text{V}+\text{VII} = (\dot \am - \cm + W(h\am))\la \partial_y v, \partial_yQ_\cm\ra +(\dot\cm \cm^{-1} - hW'(h\am)) \la v, \partial_y^2( y Q_\cm) \ra - \la e_2v, \partial_y^2 Q_\cm \ra$$
By Taylor's theorem with the Lagrange form of the remainder
$$e_2(y,\am) = \frac12 h^2 W''(h(y_*+\am))y^2$$
Let $R>0$ such that $\supp W \subset [-R,R]$.  Then $-\am - Rh^{-1} \leq y \leq -\am + Rh^{-1}$.  This gives
\begin{align*}
 \la e_2v, \partial_y^2 Q_\cm \ra &=  \int_{-\am-Rh^{-1}}^{\am+Rh^{-1}} v e_2\partial_y^2 Q_\cm \, dy =  \frac12 h^2 \int_{-\am-Rh^{-1}}^{\am+Rh^{-1}} v W''(h(y_*+\am)) y^2 \partial_y^2 Q_\cm \, dy \\
  &= \frac12 h^2 \sum_{n\in \mathbb{Z}} \int_{-\am-Rh^{-1}}^{\am+Rh^{-1}} \mathbf{1}_{[n,n+1]} \, v \, W''(h(y_*+\am)) \,  \, y^2 \partial_y^2 Q_\cm \, dy
\end{align*}
Thus
\begin{equation}
\label{E:e2est03}
\begin{aligned}
|\la \partial_y(e_2v), Q_\cm \ra| &\lesssim h^2 \sum_{n\in \mathbb{Z}} \left( \int_n^{n+1} v^2 \, dy \right)^{1/2} \left( \int_n^{n+1} \mathbf{1}_{[-\am-Rh^{-1}, \am+Rh^{-1}]}(y) \, y^2 \partial_y^2 Q_\cm(y) \, dy \right)^{1/2} \\
&\lesssim h^2 \sum_{n\in \mathbb{Z}} \left( \int_n^{n+1} v^2 \, dy \right)^{1/2}  \frac{1}{\la n \ra^2} \mathbf{1}_{[-\am - Rh^{-1}-1, -\am + Rh^{-1}+1]}(n) \\
&\lesssim h^2  \sup_n \|v\|_{L_{n\leq y\leq n+1}^2}
\end{aligned}
\end{equation}
Also,
\begin{align*}
\int_I |\la \partial_y(e_2v), Q_\cm \ra|\, dt &\lesssim h^2 \sum_{n\in \mathbb{Z}} \int_I \left( \int_n^{n+1} v^2 \, dy \right)^{1/2} \left( \int_n^{n+1} \mathbf{1}_{[-\am-Rh^{-1}, \am+Rh^{-1}]}(y) \, y^2 \partial_y^2 Q_\cm(y) \, dy \right)^{1/2} \\
&\lesssim h^2 \sum_{n\in \mathbb{Z}} \int_I \left( \int_n^{n+1} v^2 \, dy \right)^{1/2}\, dt  \frac{1}{\la n \ra^2} \mathbf{1}_{[-\am - Rh^{-1}-1, -\am + Rh^{-1}+1]}(n)
\end{align*}
By Cauchy-Schwarz in $t$,
\begin{equation}
\label{E:e2est04}
\int_I |\la \partial_y(e_2v), Q_\cm \ra|\, dt
\lesssim h^2 |I|^{1/2} \sup_n \|v\|_{L_I^2 L_{n\leq y\leq n+1}^2}
\end{equation}

Finally, we have
$$|\text{VI}| = \frac12|\la v^2, \partial_y^2Q_c\ra| \lesssim \| v\la y \ra^{-1} \|_{L_y^2}^2$$
Combining the estimates above,
\begin{align*}
& |(\dot\am - \cm + W(h\am))\Big(1 - \frac{1}{4\pi \cm^3} \la v, \partial_y^2Q_\cm\ra\Big) - \frac12 h^2W''(h\am)\cm^{-2} - \frac{1}{4\pi\cm^3}\la v, \mathcal{L}_\cm \partial_y^2Q_\cm\ra| \\
& \qquad \lesssim h^4 + \sup_n \|v\|_{L_{y\in(n,n+1)}^2}^2
\end{align*}
This implies
\begin{align*}
& |(\dot\am - \cm + W(h\am)) - \frac12 h^2W''(h\am)\cm^{-2} - \frac{1}{4\pi\cm^3}\la v, \mathcal{L}_\cm \partial_y^2Q_\cm\ra|\Big(1 - \frac{1}{4\pi \cm^3}\la v, \partial_y^2Q_\cm\ra\Big) \\
& \qquad \lesssim h^4 + \sup_n \|v\|_{L_{y\in(n,n+1)}^2}^2
\end{align*}
which implies the first inequality in \eqref{E:par2}.  Similarly, the first inequality in \eqref{E:par2t} follows by using \eqref{E:e2est04} in place of \eqref{E:e2est03}.
\end{proof}

Now we apply the result of Lemma \ref{L:nonsymp-ODE-control} to reformulate the equation for $v$. Plugging \eqref{E:est-52} into \eqref{E:v}, the equation for $v$ is now
\begin{equation}
\label{E:v2}
\begin{aligned}
\partial_t v = \;
&\frac{1}{4\pi} \cm^{-3} \la v, \mathcal{L}_c \partial_y^2Q_\cm\ra \partial_yQ_\cm + E_\am \partial_yQ_\cm +E_\cm \partial_y(yQ_\cm) + \partial_y(e_2Q_\cm) \\
&+\partial_y \mathcal{L}_\cm v + \partial_y(\dot\am - \cm + W(hx)) v - \frac12 \partial_y v^2
\end{aligned}
\end{equation}
This takes the form
\begin{equation} \label{E:v-eqn}
\partial_t v = \mathbb{P} v + \partial_y \mathcal{L}_\cm v + \partial_y f ,
\end{equation}
where $\mathbb{P}$ is the rank one operator
$$\mathbb{P}v = \frac{1}{4\pi} \cm^{-3} \la v, \mathcal{L}_\cm \partial_y^2 Q_\cm \ra \partial_y Q_\cm$$
and
\begin{equation}
\label{E:f-def}
 f(y,\am,\cm) = E_\am \, Q_\cm + E_\cm \, y Q_\cm + e_2 Q_\cm + (\dot\am - \cm + W(h(y+\am))v - \frac12 v^2
 \end{equation}

\begin{lemma}[energy estimate]
\label{L:en}
Consider a time interval $I= [T_*,T^*]$ of length 
$$|I|=T^*-T_* \lesssim h^{-1}$$  
on which $\|v\|_{L_I^\infty H_y^{1/2}} \leq h^{4/3}$ and $\frac12 \leq \cm(t) \leq 2$ holds for all $t\in I$.   Then
$$\| v\|_{L_I^\infty H_y^{1/2}}^2 \lesssim \|v(T_*) \|_{H_y^{1/2}}^2 + h^2|I|^{1/2} \| \la y \ra^{-1} v \|_{L_I^2 L_y^2}  + h^4 |I| $$
\end{lemma}
\begin{proof}
\newcommand{\rum}{\mathfrak{r}}
Let $\rum(t) = \int_{T_*}^t |\dot \cm(s)| \, ds$.  Since $|\dot \cm(t)| \lesssim h$ and $T^*-T_* \lesssim h^{-1}$, it follows that $\rum(t) = O(1)$ on $T_*\leq t \leq T^*$.  For a sufficiently large constant $\kappa$ (to be selected below), we have
\begin{align*}
\indentalign e^{\kappa \rum} \, \partial_t \,  e^{-\kappa \rum} \left( \frac12 \la \mathcal{L}_\cm v, v \ra - \frac16 \int v^3 \right) \\
&= - \kappa |\dot \cm| \left( \frac12 \la \mathcal{L}_\cm v, v \ra - \frac16 \int v^3 \right) + \tfrac12 \dot \cm \la v,v\ra + \la \mathcal{L}_\cm v, \partial_t v \ra  - \tfrac12 \la v^2, \partial_t v \ra
\end{align*}
By the spectral bounds, there exists a constant $\kappa>0$ sufficiently large so that the first term dominates the second, giving  
$$e^{\kappa \rum} \, \partial_t \,  e^{-\kappa \rum} \left( \frac12 \la \mathcal{L}_\cm v, v \ra - \frac16 \int v^3 \right) \leq \la \mathcal{L}_\cm v, \partial_t v \ra  - \tfrac12 \la v^2, \partial_t v \ra 
$$
By substituting \eqref{E:v-eqn},
\begin{align*}
\indentalign e^{\kappa \rum} \, \partial_t \,  e^{-\kappa \rum} \left( \frac12 \la \mathcal{L}_\cm v, v \ra - \frac16 \int v^3 \right) \\
&= \la \mathcal{L}_\cm v, \mathbb{P} v \ra +  \la \mathcal{L}_\cm v, \partial_y \mathcal{L}_\cm v \ra +  \la \mathcal{L}_\cm v, \partial_y f \ra \\
& \qquad - \tfrac12 \la v^2, \mathbb{P} v \ra - \tfrac12 \la v^2, \partial_y \mathcal{L}_\cm v \ra - \tfrac12 \la v^2, \partial_y f \ra \\
&= \text{A} + \text{B} + \text{C} - \text{D} - \text{E} - \text{F}
\end{align*}
Term A drops away since $\mathcal{L}_\cm \mathbb{P} v = 0$ and Term B drops away by skew-symmetry.  It is fairly straightforward to obtain suitable bounds on $|\text{D}|$ and $|\text{F}|$, specifically
$$|\text{D}| \lesssim \|\la y \ra^{-1} v \|_{L_y^2}^3$$
$$|\text{F}| \lesssim h^2 \| \la y \ra^{-1} v \|_{L_y^2}^2 + h \|v\|_{H_y^{1/2}}^3$$
which more than suffice.  The main task is to prove the following estimate for $|\text{C} -\text{E}|$:
$$ \left|  \la \partial_yf, \mathcal{L}_\cm v \ra - \frac12 \la v^2, \partial_y \mathcal{L}_\cm v \ra \right| \lesssim h^2 \| \la y \ra^{-1} v\|_{L_y^2} + h\|v\|_{H_y^{1/2}}^2 +h^4$$
Substituting \eqref{E:f-def},
\begin{align*}
\la \partial_yf, \mathcal{L}_\cm v \ra - \tfrac12 \la v^2, \partial_y \mathcal{L}_\cm v \ra &= 
E_\am \la \partial_y Q_\cm, \mathcal{L}_\cm v \ra && \leftarrow \text{I} \\
& \qquad + E_\cm \la \partial_y(yQ_\cm) , \mathcal{L}_\cm v \ra && \leftarrow \text{II} \\
& \qquad + \la \partial_y(e_2 Q_\cm), \mathcal{L}_\cm v \ra && \leftarrow \text{III} \\
& \qquad + h \la W'(h(y+\am)) v, \mathcal{L}_\cm v \ra && \leftarrow \text{IV} \\
& \qquad + (\dot \am - \cm + W(h\am)) \la  \partial_y v, \mathcal{L}_\cm v \ra && \leftarrow \text{V} \\
& \qquad + \la (W(h(y+\am))-W(h\am)) \partial_y v, \mathcal{L}_\cm v \ra && \leftarrow \text{VI} \end{align*}
Each of these six terms is estimated separately, as follows.  

In Term I, we break up the terms of $\mathcal{L}_\cm = \cm - H\partial_y - Q_\cm$, and for the middle term, integrate by parts: $\la \partial_y Q_\cm, H\partial_y v \ra = \la H\partial_y^2 Q_\cm,  v \ra$, and note that $|H\partial_y^2 Q_\cm(y) | \lesssim \la y \ra^{-3}$.    Then each of these terms is estimated via Cauchy Schwarz:
$$|\la  \partial_yQ_\cm, \mathcal{L}_\cm v \ra | \lesssim \| \la y \ra^{-1} v \|_{L_y^2}$$
Combining with \eqref{E:est-52} completes the estimate for Term I.  Term II is similar -- $yQ_\cm$ has weaker decay, but still sufficient to obtain the same bound as for Term I.    In particular, $|H\partial_y^2 [yQ_\cm(y)] | \lesssim \la y \ra^{-2}$.  

For Term III, we refer to the estimate of Term III in \cite[Lemma 8.1]{Z}, where the  estimate  $\|  \mathcal{L}_\cm \partial_y (e_2Q_\cm) \|_{L_y^2} \lesssim h^{5/2}$ is proved.  Cauchy-Schwarz then yields
$$| \text{III} | \lesssim h^{5/2} \|v\|_{L_y^2} \lesssim h^4 + h\|v\|_{L_y^2}^2$$

For Term IV, we estimate the contribution of each term of $\mathcal{L}_\cm = \cm - H\partial_y - Q_\cm$ separately.  The  nontrivial term is
$$h \la W'(h(y+\am)v, H\partial_y v \ra = h \la D_y^{1/2}[W'(h(y+\am))v],D_y^{1/2} v \ra$$
After Cauchy-Schwarz, appeal to the fractional Leibniz estimate \eqref{E:Leib1}, noting that $\|W'(h(y+\am)) \|_{L_y^2}\sim h^{-1/2}$ while $\|\partial_y[W'(h(y+\am))]\|_{L_y^2} \sim h^{1/2}$.  This yields
$$h | \la W'(h(y+\am)v, H\partial_y v \ra | \lesssim h \|v \|_{H_y^{1/2}}^2$$
and thus the same estimate for Term IV.
For Term V, we use that $\la \partial_y v, \mathcal{L}_\cm v \ra = \frac12 \la \partial_y Q_\cm, v^2 \ra$ and thus
$$| \la \partial_y v, \mathcal{L}_\cm v \ra| \lesssim \| \la y \ra^{-1} v\|_{L_y^2}^2$$
Also the coefficient $\dot \am - \cm +W(h\am) = E_{\am} + \frac{1}{4\pi} \cm^{-3} \la v, \mathcal{L}_\cm \partial_y^2 Q_\cm \ra$ and thus  by \eqref{E:est-52}, $|\dot \am - \cm + W(h\am)|\lesssim h$.  Combining gives
$$|\text{V}| \lesssim h \| v \|_{H_y^{1/2}}^2 $$
For Term VI,  we substitute $\mathcal{L}_\cm = \cm - H\partial_y - Q_\cm$ and integrate by parts to obtain
\begin{align*}
\text{VI} &= - \frac12 h \la W'(h(y+\am)), v^2 \ra - \la [W(h(y+\am))- W(h\am)]\partial_y v, H\partial_y v \ra \\
&\qquad + \frac12 \la \partial_y ( [W(h(y+\am))-W(h\am)]Q_\cm(y)), v^2 \ra
\end{align*}
The first and third of these terms is easily estimated with Cauchy-Schwarz, and for the middle term we use Lemma \ref{L:com5} to obtain
$$|\text{VI}| \lesssim h \|v\|_{L_y^2}^2$$

\end{proof}

Recall from the local virial estimate (Theorem \ref{T:local-virial}) the form of the remainder $G$ in \eqref{E:G}, 
$$
G_\gamma (f, v) = \int_0^T \int g_{\gamma,y_0} \, v \, \partial_y f \, dy dt +  \int_0^T \int_y g_{\gamma,0} (\mathcal{D}_\gamma^{-1} \mathcal{L}_\cm v) ( \mathcal{D}_\gamma^{-1} \mathcal{L}_\cm \partial_y f) \, dy \, dt 
$$
where $f$ is given in \eqref{E:f-def}.

\begin{lemma}[estimate on $G$ remainder in local virial estimate]
\label{L:G-est}
$$|G_\gamma(f,  v)|  \lesssim_\gamma  h^2T^{1/2} \|\la y \ra^{-1} v\|_{L_T^2L_y^2} + hT  \|v\|_{L^\infty_{t \in [0, T]} L_y^2}^2 + T\|v\|_{L^\infty_{t \in [0, T]} H_y^{1/2}}^3$$
\end{lemma}
\begin{proof}
There are several terms to estimate, but one of primary interest is:
$$ I = \int_y g_\gamma [\mathcal{D}_\gamma^{-1}\mathcal{L}_\cm v][ \mathcal{D}_\gamma^{-1}\mathcal{L}_\cm \partial_y (v^2) ] \, dy $$
We will now show
\begin{equation}
\label{E:est-50}
| I | \lesssim_\gamma \|v \|_{H_y^{1/2}}^3
\end{equation}
In the composition
$$\mathcal{D}_\gamma^{-1} \mathcal{L}_\cm = \mathcal{D}_\gamma^{-1}(I-H\partial_y - Q)$$
the term $\mathcal{D}_\gamma^{-1} H \partial_y$ is somewhat delicate.  Since  $\partial_y \mathcal{D}_\gamma^{-1}  = \gamma^{-1}(I - \mathcal{D}_\gamma^{-1})$, it follows that
$$\mathcal{D}_\gamma^{-1} \mathcal{L}_\cm = -\gamma^{-1}H + \mathcal{D}_\gamma^{-1}A \,, \qquad \text{where }A = I+ \gamma^{-1}H - Q$$
Substituting,
\begin{equation}
\label{E:I-decomp}
I = \gamma^{-2} I_0+\gamma^{-1} I_1+\gamma^{-1} I_2+I_3
\end{equation}
where
$$I_0 =  \int_y g_\gamma \,  H v \;  H \partial_y (v^2)  \, dy\,, \qquad I_1 = -\int_y g_\gamma \, Hv \; \mathcal{D}_\gamma^{-1} A\partial_y (v^2) \, dy $$
$$I_2 = -\int_y g_\gamma \, \mathcal{D}_\gamma^{-1} Av \; H \partial_y (v^2) \, dy \,, \qquad I_3 = \int_y g_\gamma \, \mathcal{D}_\gamma^{-1} Av \; \mathcal{D}_\gamma^{-1} A\partial_y (v^2) \, dy$$
First we address term $I_0$.  Note that
\begin{align*}
 I_0 &=  \int \partial_y H(gHv) \, v^2 \, dy \\
&=  \int H(g' Hv) \, v^2 \, dy +  \int H(g \, H\partial_y v) \, v^2 \, dy \\
&=   \int H(g' Hv) \, v^2 \, dy +  \int [H(g \, H\partial_y v) - g \, H^2\partial_y v] \, v^2 \, dy +  \int g \, \partial_y v \, v^2 \, dy \\
&= \text{I} + \text{II} + \text{III}
\end{align*}
In term III, we use integration by parts
$$ \text{III} = - \frac13  \int g' v^3 \, dy \lesssim \|v\|_{L^3}^3 \lesssim \|v\|_{H_y^{1/2}}^3$$
For term I, we use the $L^3\to L^3$ boundedness of $H$ to deduce
$$
|\text{I}|  \lesssim \|H(g'Hv) \|_{L^3} \|v\|_{L^3}^2 \lesssim \|g' Hv \|_{L^3} \|v\|_{L^3}^2 \lesssim \|Hv \|_{L^3} \| v\|_{L^3}^2  \lesssim \|v\|_{L^3}^3 \lesssim \|v\|_{H_y^{1/2}}^3$$
To address term II, we apply Lemma \ref{L:H-commutator}, as follows:
$$|\text{II}| \lesssim \| H(g \, H\partial_y v) - g \, H^2\partial_y v \|_{L_y^2}  \| v \|_{L_y^4}^2 \lesssim \|g\|_{H^2} \| \partial_y v\|_{H^{-1}} \|v\|_{L_4}^2 \lesssim \|v\|_{H_y^{1/2}}^3$$
This completes term $I_0$.  Returning to \eqref{E:I-decomp}, we need to address terms $I_1$, $I_2$ and $I_3$.  For terms $I_1$ and $I_3$, we will use
$$A\partial_y = \partial_y A + Q'$$
together with the fact that $\mathcal{D}_\gamma^{-1} \partial_y$ is $L^2\to L^2$ bounded with operator norm $\lesssim \gamma^{-1}$.  These observations, together with H\"older and Sobolev yeild the needed bounds for $I_1$ and $I_3$ .  After integrating by parts, term $I_2$ is 
\begin{align*}
I_2 &= \int_y \partial_y [ g_\gamma \, \mathcal{D}_\gamma^{-1} Av] \; H  (v^2) \, dy \\
&=  \int_y [ g_\gamma' \, \mathcal{D}_\gamma^{-1} Av] \; H  (v^2) \, dy +  \int_y [ g_\gamma \, \partial_y \mathcal{D}_\gamma^{-1} Av] \; H  (v^2) \, dy
\end{align*}
The estimate for $I_2$ is now completed with H\"older, Sobolev, the fact that $\mathcal{D}_\gamma^{-1} \partial_y$ is $L^2\to L^2$ bounded with operator norm $\lesssim \gamma^{-1}$, and the $L^2\to L^2$ boundedness of $H$.  This completes the proof of \eqref{E:est-50}.
\end{proof}

Now we can insert the bound from Lemma \ref{L:G-est} into Theorem \ref{T:local-virial} to obtain the following
\begin{corollary}[local virial estimate]
\label{C:lv-pBO}
Consider a time interval $I= [T_*,T^*]$ of length 
$$|I|=T^*-T_* \lesssim h^{-1}$$  
on which $\|v\|_{L_I^\infty H_y^{1/2}} \leq h^{4/3}$ and $\frac12 \leq \cm(t) \leq 2$ holds for all $t\in I$.  Then
$$ \sup_n \|  v \|_{L_I^2 L_{ y \in [n, n+1] }^2  }^2 \lesssim  h^3 + \|v\|_{L_I^\infty L_y^2}^2 $$
\end{corollary}

\begin{proof}
Plugging the bound in the statement of Lemma \ref{L:G-est} into \eqref{E:D1D} gives 
\begin{equation}
\begin{split}
\|  \la D_y \ra^{1/2} ((g_{\gamma,y_0}')^{1/2} v ) \|_{L_{I}^2L_y^2}^2 
& \lesssim_\gamma \|v\|_{L_{I}^\infty L_y^2}^2 + h^2 |I|^{1/2} \|\la y \ra^{-1} v\|_{L_I^2 L_y^2} + h |I| \|v\|_{L^\infty_{ I} L_y^2}^2 + |I| \|v\|_{L^\infty_{ I} H_y^{1/2}}^3  \\
\end{split}
\end{equation}
The left hand side satisfies 
\begin{equation*}
\sup_{y_0 \in \mathbb{R}} \|  \la D_y \ra^{1/2} ((g_{\gamma,y_0}')^{1/2} v ) \|_{L_{[0,T]}^2L_y^2}^2  \gtrsim \sup_n \|  v \|_{L_I^2 L_{ y \in [n, n+1] }^2  }^2 ,
\end{equation*}
and the right hand side is controlled as 
\begin{equation}
\begin{split}
& \quad \|v\|_{L_{I}^\infty L_y^2}^2 + h^2 |I|^{1/2} \|\la y \ra^{-1} v\|_{L_I^2 L_y^2} + h |I| \|v\|_{L^\infty_{ I} L_y^2}^2 + |I| \|v\|_{L^\infty_{ I} H_y^{1/2}}^3  \\
& \lesssim \|v\|_{L_{I}^\infty L_y^2}^2 + h^{3/2} \sup_n \|  v \|_{L_I^2 L_{ y \in [n, n+1] }^2  }^2 +   h^{3}  \\
\end{split}
\end{equation}
since $|I|=T^*-T_* \lesssim h^{-1}$ and $\|v\|_{L_I^\infty H_y^{1/2}} \leq h^{4/3}$. Combining these together we obtain 
$$ \sup_n \|  v \|_{L_I^2 L_{ y \in [n, n+1] }^2  }^2 \lesssim  h^3 + \|v\|_{L_I^\infty L_y^2}^2 . $$

\end{proof}

The proof of Proposition \ref{P:nonsymp-estimates} can now be completed by combining Lemma \ref{L:nonsymp-ODE-control} on the $\am(t)$, $\cm(t)$ parameter trajectories, Lemma \ref{L:en} on the energy estimate, and Corollary \ref{C:lv-pBO} on the local virial estimate.

\begin{proof}[Proof of Prop. \ref{P:nonsymp-estimates}]
It suffices to show the bound \eqref{E:nonsymp-estimates-eq1}. Plugging the local virial estimate in Corollary \ref{C:lv-pBO} into the estimate in Lemma \ref{L:en} gives
\begin{equation}
\begin{split}
\| v\|_{L_I^\infty H_y^{1/2}}^2  
& \lesssim \|v(T_*) \|_{H_y^{1/2}}^2 + h^2|I|^{1/2} (h^3 + \|v\|_{L_I^\infty L_y^2}^2)^{1/2}  \\
& \lesssim \|v(T_*) \|_{H_y^{1/2}}^2 + h^{3/2} (h^3 + \|v\|_{L_I^\infty L_y^2}^2)^{1/2}  \\ 
& \lesssim \|v(T_*) \|_{H_y^{1/2}}^2 + h^3+ h^{3/2}   \|v\|_{L_I^\infty L_y^2} .  \\ 
\end{split}
\end{equation}
This yields
$$\|v\|_{L_{[0,T]}^\infty H_y^{1/2}}  \lesssim \|v(T_*) \|_{H_y^{1/2}} +  h^{3/2}  . $$
Plugging this into the local virial estimate in Corollary \ref{C:lv-pBO}, we obtain
\begin{equation}
\begin{split}
 \sup_n \|v\|_{L_{I}^2L_{y\in(n,n+1)}^2}^2 
& \lesssim   h^3 + \|v\|_{L_I^\infty L_y^2}^2  \lesssim  \|v(T_*) \|_{H_y^{1/2}}^2 + h^3 . \\
\end{split}
\end{equation}
Combining the results above, we have, for a time interval $I= [T_*,T^*]$ of length $|I|=T^*-T_* \lesssim h^{-1}$,  
\begin{equation} \label{E:nonsymp-estimates-eq2}
\|v\|_{L_{I}^\infty H_y^{1/2}} + \sup_n \|v\|_{L_{I}^2L_{y\in(n,n+1)}^2} \leq  C h^{3/2} + C \|v(T_*) \|_{H_y^{1/2}} 
\end{equation}
for some universal constant $C> 1$ (which only depends on the initial data).

Now, we consider the time interval $[0, T]$, and split it into sub-intervals of length $\delta h^{-1}$ (here $\delta >0$ is a small constant): $I_1 = [0, T_1]$, $I_2 = [T_1, T_2]$, $\cdots$, $I_J = [T_{J-1}, T]$, with $J = [ Th/\delta ]$ ($[*]$ means the ceiling function), and $|T_j - T_{j-1}| = \delta h^{-1}$ for all $j =1, 2, \cdots, J-1$. We iterate the estimate \eqref{E:nonsymp-estimates-eq2} on $I_1$, $I_2$, $\cdots$, $I_J$ and obtain
\begin{equation*}
\begin{split}
\|v\|_{L_{[0,T]}^\infty H_y^{1/2}} + \sup_n \|v\|_{L_{[0,T]}^2L_{y\in(n,n+1)}^2} 
& \leq  (C^J+ C^{J-1} + \cdots + C) h^{3/2} + C^J \|v(T_*) \|_{H_y^{1/2}}  \\
& \leq \frac{C(C^J-1)}{C-1} (h^{3/2} + \|v(T_*) \|_{H_y^{1/2}}  )  \\
& = \frac{C(C^{[Th/\delta]}-1)}{C-1} (h^{3/2} + \|v(T_*) \|_{H_y^{1/2}}  )  \\
\end{split}
\end{equation*}
Taking $\kappa = 10$ and $\mu = \frac{\ln C}{\delta}$ completes the proof for \eqref{E:nonsymp-estimates-eq1}.
\end{proof}

\bigskip

\section{Exact dynamics for (pBO)}
\label{S:exact}

In this section, we prove Theorem \ref{T:main}.  To start, we will describe how to convert from the nonsymplectic orthogonality condition \eqref{E:D1E} to the symplectic orthogonality condition \eqref{E:D1F}. 

We introduce the following codimension $2$ (closed) subspaces of $H_x^{1/2}$:  for given $(\am,\cm)$
$$X_{\am,\cm} = \{ \; \zeta\in H_x^{1/2} \; | \; \la \zeta, Q_{\am,\cm} \ra =0 \,, \; \la \zeta, Q'_{\am,\cm} \ra =0 \; \}$$
Also, for given $( a,  c)$, we define
$$Y_{ a, c} = \{ \; \eta\in H_x^{1/2} \; | \; \la \eta, Q_{ a, c} \ra =0 \,, \; \la \eta, (x- a)Q_{ a, c} \ra =0 \; \}$$
Within $H_x^{1/2}$, for a fixed small $\epsilon>0$, we consider the tubular neighborhood of the two-dimensional soliton manifold
$$M = \{ \; u\in H_x^{1/2} \; | \; \text{ there exists $a\in \mathbb{R}$, $\frac12<c<2$ such that }\|u - Q_{a,c} \|_{H_x^{1/2}} < \epsilon \; \}$$
By an argument appealing to the implicit function theorem (the $\epsilon>0$ is chosen so that this argument is valid), there is a well-defined map
$$\Lambda: M \to  \mathbb{R}^2\times H^{1/2}$$
that sends
$$u \mapsto (\am,\cm, \zeta)$$
where $\zeta \in X_{\am,\cm}$ and $ \zeta = u - Q_{\am,\cm}$.  Similarly there is a well-defined map
$$\Gamma:  M \to  \mathbb{R}^2\times H^{1/2}$$
that sends
$$u \mapsto ( a, c,  \eta)$$
where $ \eta \in Y_{a,c}$ and $  \eta = u - Q_{ a, c}$.

Here, we investigate a feature of the composition
$$\Gamma \circ \Lambda^{-1}:  \Lambda(M) \to \Gamma(M)$$
that sends
$$(\am,\cm,\zeta) \mapsto ( a,  c, \eta)$$
It follows from the construction of $\Lambda$ (via the implicit function theorem) that $|a-\am| \lesssim \epsilon$ and $|c-\cm|\lesssim \epsilon$.

Let $\tilde X_{\am,\cm}$ be the $\epsilon$-ball in $X_{\am,\cm}$ around the origin.  If $\|\zeta\|_{H^{1/2}} <\epsilon$, then $u = \zeta+Q_{\am,\cm} \in M$, so that $(\am,\cm,\zeta) \in \Lambda(M)$.  Thus, for fixed $\am$, $\cm$,  one has the restricted map
$$ \tilde X_{\am,\cm} \to \mathbb{R}^2 \times H_x^{1/2}$$
given by
$$\zeta \mapsto ( a,  c, \eta)$$
We will use the notation $(a(\zeta), c(\zeta))$ to emphasize the dependence of $a$, $c$ upon $\zeta$ through this mapping.  After composing this mapping with the projection onto the third component, we obtain the mapping, for fixed $\am$, $\cm$:
$$\Omega_{\am, \cm}: \tilde X_{\am,\cm} \to H^{1/2}_x$$
that sends
$$ \zeta \mapsto  \eta$$

\begin{lemma}
\label{L:coord-conv}
For fixed $\am$, $\cm$, under the mapping $\eta = \Omega_{\am,\cm}(\zeta)$ defined above,
\begin{equation}
\label{E:coord-conv01}
\begin{aligned}
\eta(x) &= \zeta(x) + \int_{s=0}^1 Q_{a(s\zeta),c(s\zeta)}'(x) \frac{\partial a}{\partial\zeta} \Big|_{s\zeta}(\zeta) \, ds \\
& \qquad - \int_{s=0}^1 [(\bullet-a)Q_{a,c}]\Big|_{a(s\zeta),c(s\zeta)}'(x)\frac{\partial c}{\partial\zeta} \Big|_{s\zeta}(\zeta)  \, ds
\end{aligned}
\end{equation}
where $\frac{\partial  a}{\partial \zeta}\Big|_{s\zeta}(\zeta)$ and $ \frac{\partial  c}{\partial \zeta}\Big|_{s\zeta} (\zeta)$ are given by
$$
\begin{bmatrix} \frac{\partial  a}{\partial \zeta}\Big|_{s\zeta}(\zeta) \\ \frac{\partial  c}{\partial \zeta}\Big|_{s\zeta} (\zeta) \end{bmatrix}
=\begin{bmatrix} 
a_{11} & a_{12} \\ 
a_{21} & a_{22}
\end{bmatrix}^{-1}
\begin{bmatrix}
\la \zeta , Q_{ a(s\zeta),  c(s\zeta)} \ra \\ 
\la \zeta , (x- a(s\zeta))Q_{ a(s\zeta),  c(s\zeta)} \ra
\end{bmatrix}
$$
and, with $\eta_s = \Omega_{\am,\cm}(s\zeta)$,
$$
\begin{bmatrix}
a_{11} & a_{12} \\ a_{21} & a_{22} 
\end{bmatrix}^{-1}
= 2\|Q\|_{L_x^2}^{-2} \begin{bmatrix} 0 & 1 \\ 1 & 0 \end{bmatrix} + O( \|\eta_s\|_{L_x^2})$$
\end{lemma}

So as not to get lost in the complexity of the formula, note the following approximation, which basically suffices for our purposes:  $a(s\zeta) \approx a(0) \approx \am$ and $c(s\zeta) \approx c(0) \approx \cm$ (all accurate within $O(\epsilon)$) and therefore
\begin{equation}
\label{E:coord-conv02}
\eta(x) \approx \zeta(x) + 2 \|Q\|_{L^2}^{-2} Q_{a,c}'(x) \la \zeta, (\bullet -a)Q_{a,c}\ra + 2\|Q\|_{L^2}^{-2} [(x-a)Q_{a,c}(x)]' \la \zeta, Q_{a,c} \ra
\end{equation}

\begin{proof}
The derivative of the map $\Omega_{\am,\cm}:X_{\am,\cm} \to H_x^{1/2}$ is of the form
$$D\Omega_{\am,\cm}: X_{\am,\cm} \to \mathcal{L}(X_{\am,\cm}; H_x^{1/2})$$
Using that $\Omega_{\am,\cm}(0)=0$, we obtain
\begin{align}
\notag \eta &= \Omega_{\am,\cm}(\zeta) - \Omega_{\am,\cm}(0) \\
\notag &= \int_0^1 \frac{d}{ds} \Omega_{\am,\cm}(s\zeta) \, ds \\
\label{E:est-57} &= \int_0^1 \underbrace{D\Omega_{\am,\cm}(s\zeta)}_{\in \mathcal{L}(X_{\am,\cm};H_x^{1/2})}(\zeta) \, ds 
\end{align}

We will compute  bounds on $D\Omega_{\am,\cm}(\zeta_0)(\delta \zeta)$ and apply them to \eqref{E:est-57}.  A workable expression can be obtained for the derivative $D\Omega_{\am,\cm}$ by taking an implicit derivative of the defining equations.  Indeed, note that 
$$\eta = \Omega_{\am,\cm}(\zeta) = \zeta +Q_{\am,\cm} - Q_{ a(\zeta),  c(\zeta)}$$
so that at a reference point $\zeta_0\in X_{\am,\cm}$,
\begin{equation}
\label{E:est-61}
D\Omega_{\am,\cm}(\zeta_0) = I - D[ Q_{ a(\bullet),  c(\bullet)} ] (\zeta_0)
\end{equation}
where $I: X_{\am,\cm} \to H_x^{1/2}$ is the identity map and $D[ Q_{ a(\bullet),  c(\bullet)} ] (\zeta_0)$ refers to the derivative at $\zeta_0$ of the composite map 
$$\zeta \mapsto (  a(\zeta),  c(\zeta)) \mapsto Q_{ a(\zeta),  c(\zeta)}$$
This composition is a map
$$X_{\am,\cm} \to \mathbb{R}^2 \to H^{1/2}$$
and we take the derivative of this composite map by the chain rule:  
\begin{equation}
\label{E:est-59}
D[ Q_{ a(\bullet),  c(\bullet)} ] (\zeta_0) = DQ_{ a,  c}(  a(\zeta_0),  c(\zeta_0)) \circ D( a,  c)(\zeta_0)
\end{equation}
Here
\begin{equation}
\label{E:est-58}
D( a,  c)(\zeta_0) \in \mathcal{L}( X_{\am,\cm}; \mathbb{R}^2) \,, \qquad DQ_{ a,  c}(  a(\zeta_0),  c(\zeta_0)) \in \mathcal{L}( \mathbb{R}^2; H^{1/2})
\end{equation}
The right map in \eqref{E:est-58} is simply represented as a $1\times 2$ matrix (row vector) of functions
$$
DQ_{ a,  c}(  a(\zeta_0),  c(\zeta_0)) =  \begin{bmatrix} -Q'_{ a,  c} & [(x- a)Q_{ a,  c}]' \end{bmatrix} \Big|_{(a(\zeta_0),c(\zeta_0))}
$$
that acts on a $2\times 1$ matrix of real number increments:
$$\begin{bmatrix} \delta  a \\ \delta  c \end{bmatrix}$$
to yield an element of $H^{1/2}$ by the usual multiplication.  Thus \eqref{E:est-59} becomes, when evaluated at an ``increment function'' $\delta \zeta$, is the function
\begin{equation}
\label{E:est-60}
\begin{aligned}
\Big(D[ Q_{ a(\bullet),  c(\bullet)} ] (\zeta_0)(\delta \zeta)\Big)(x) &= -Q'_{a(\zeta_0),c(\zeta_0)} (x) \, \frac{\partial a}{\partial \zeta}\Big|_{\zeta_0}(\delta \zeta) \\
& \qquad + [(x-a)Q_{a,c}(x)]'\Big|_{(a(\zeta_0),c(\zeta_0))} \, \frac{\partial c}{\partial \zeta} \Big|_{\zeta_0}(\delta \zeta)
\end{aligned}
\end{equation}
where $D( a,  c)(\zeta_0) \in \mathcal{L}( X_{\am,\cm}; \mathbb{R}^2)$ in \eqref{E:est-58} is represented as the 2-vector with real number entries
$$D( a,  c)(\zeta_0) = \begin{bmatrix} \frac{\partial a}{\partial \zeta} \Big|_{\zeta_0}(\delta \zeta) \\ \frac{\partial c}{\partial \zeta} \Big|_{\zeta_0}(\delta \zeta) \end{bmatrix}$$
This must be understood by returning to the defining condition for $ \eta$ and applying implicit differentiation. 
Starting with 
$$0 = \left\la \zeta+ Q_{\am,\cm} - Q_{ a(\zeta),  c(\zeta)}, Q_{ a(\zeta),  c(\zeta)} \right\ra $$
take the derivative with respect to $\zeta$ at $\zeta_0$ in the direction $\delta \zeta$ to obtain
\begin{equation}
\label{E:1stdeta}
\begin{aligned}
0 &= \left\la \delta \zeta , Q_{ a(\zeta_0),  c(\zeta_0)} \right\ra \\
& - \left\la \frac{\partial Q_{ a,  c}}{\partial  a} \Big|_{( a(\zeta_0),  c(\zeta_0))} 
\frac{\partial  a}{\partial \zeta}(\zeta_0)(\delta \zeta), Q_{ a(\zeta_0),  c(\zeta_0)} \right\ra \\
&- \left\la \frac{\partial Q_{ a,  c}}{\partial  c} \Big|_{( a(\zeta_0),  c(\zeta_0))} 
\frac{\partial  c}{\partial \zeta}(\zeta_0)(\delta \zeta), Q_{ a(\zeta_0),  c(\zeta_0)} \right\ra \\
& -\left\la \eta_0, \frac{ \partial Q_{ a,  c}}{\partial  a}\Big|_{( a(\zeta_0),  c(\zeta_0))} \frac{\partial  a}{\partial \zeta}(\zeta_0) (\delta \zeta) \right\ra \\
& -\left\la \eta_0  , \frac{ \partial Q_{ a,  c}}{\partial  c}\Big|_{( a(\zeta_0),  c(\zeta_0))} \frac{\partial  c}{\partial \zeta}(\zeta_0) (\delta \zeta) \right\ra
\end{aligned}
\end{equation}
Since $\frac{\partial  a}{\partial \zeta}(\zeta_0)(\delta \zeta)$ and $\frac{\partial  c}{\partial \zeta}(\zeta_0)(\delta \zeta)$ are just real numbers, they pull out of the inner products.

Similarly, starting with 
$$0 = \left\la \zeta+ Q_{\am,\cm} - Q_{ a(\zeta),  c(\zeta)}, (x- a(\zeta)) Q_{ a(\zeta),  c(\zeta)} \right\ra $$
take the derivative with respect to $\zeta$ at $\zeta_0$ in the direction $\delta \zeta$, we can obtain another equation
\begin{equation}
\label{E:2nddeta}
\begin{aligned}
0 &= \left\la \delta \zeta , (x- a(\zeta_0)) Q_{ a(\zeta_0),  c(\zeta_0)} \right\ra \\
&- \left\la \frac{\partial Q_{ a,  c}}{\partial  a} \Big|_{( a(\zeta_0),  c(\zeta_0))} 
\frac{\partial  a}{\partial \zeta}(\zeta_0)(\delta \zeta), (x- a(\zeta)) Q_{ a(\zeta_0),  c(\zeta_0)} \right\ra \\
&- \left\la \frac{\partial Q_{ a,  c}}{\partial  c} \Big|_{( a(\zeta_0),  c(\zeta_0))} 
\frac{\partial  c}{\partial \zeta}(\zeta_0)(\delta \zeta), (x- a(\zeta)) Q_{ a(\zeta_0),  c(\zeta_0)} \right\ra \\
& -\left\la \eta_0, (x- a(\zeta)) \frac{ \partial Q_{ a,  c}}{\partial  a}\Big|_{( a(\zeta_0),  c(\zeta_0))} \frac{\partial  a}{\partial \zeta}(\zeta_0) (\delta \zeta) \right\ra \\
& -\left\la \eta_0, (x- a(\zeta)) \frac{ \partial Q_{ a,  c}}{\partial  c}\Big|_{( a(\zeta_0),  c(\zeta_0))} \frac{\partial  c}{\partial \zeta}(\zeta_0) (\delta \zeta) \right\ra \\
& +\left\la \eta_0, \frac{\partial  a}{\partial \zeta}(\zeta_0) (\delta \zeta) Q_{ a(\zeta_0),  c(\zeta_0)} \right\ra  \\
\end{aligned}
\end{equation}
Note that by moving the terms that involve $\frac{\partial  a}{\partial \zeta}(\zeta_0)(\delta \zeta)$ or $\frac{\partial  c}{\partial \zeta}(\zeta_0)(\delta \zeta)$ in \eqref{E:1stdeta} and \eqref{E:2nddeta} to the left hand side, \eqref{E:1stdeta} and \eqref{E:2nddeta} can be combined into a vector equation
\begin{equation}
\label{E:daceta}
\begin{bmatrix}
\la \delta \zeta , Q_{ a(\zeta_0),  c(\zeta_0)} \ra \\ 
\la \delta \zeta , (x- a(\zeta_0))Q_{ a(\zeta_0),  c(\zeta_0)} \ra
\end{bmatrix}
=
\begin{bmatrix} 
a_{11} & a_{12} \\ 
a_{21} & a_{22}
\end{bmatrix}
\begin{bmatrix} \frac{\partial  a}{\partial \zeta}(\zeta_0)(\delta \zeta) \\ \frac{\partial  c}{\partial \zeta}(\zeta_0)(\delta \zeta) \end{bmatrix}
\end{equation}
where the coefficient matrix has the following components:  $a_{11} = a_{11}^0+b_{11}$, where
$$
a_{11}^0 =  \left\la \frac{\partial Q_{ a,  c}}{\partial  a} \Big|_{( a(\zeta_0),  c(\zeta_0))} , Q_{ a(\zeta_0),  c(\zeta_0)} \right\ra  = \frac{\partial}{\partial a}\Big|_{a(\zeta_0)} \|Q_{a,c(\zeta_0)} \|_{L_x^2}^2 =0
$$
and
$$b_{11}=\left\la \frac{\partial Q_{ a,  c}}{\partial  a} \Big|_{( a(\zeta_0),  c(\zeta_0))} , \eta_0 \right\ra \quad \implies \quad |b_{11}| \leq \|\eta_0\|_{L_x^2}$$  Next, $a_{12}=a_{12}^0+b_{12}$, where
\begin{align*}
a_{12}^0 &=  \left\la \frac{\partial Q_{ a,  c}}{\partial  c} \Big|_{( a(\zeta_0),  c(\zeta_0))} , Q_{ a(\zeta_0),  c(\zeta_0)} \right\ra \\
&= \frac12 \frac{\partial}{\partial c}\Big|_{c(\zeta)} \| Q_{a,c} \|_{L_x^2}^2 = \frac12 \frac{\partial}{\partial c}\Big|_{c(\zeta)} \left( c \|Q\|_{L_x^2}^2 \right) =\frac12 \|Q\|_{L_x^2}^2 
\end{align*}
and
$$b_{12} =  \left\la \frac{\partial Q_{ a,  c}}{\partial  c} \Big|_{( a(\zeta_0),  c(\zeta_0))} , \eta_0\right\ra
 \quad \implies \quad |b_{12}| \leq \|\eta_0\|_{L_x^2}$$
Next, $a_{21} = a_{21}^0 + b_{21}$, where 
\begin{align*}
a_{21}^0 &= \left\la (x- a(\zeta)) \frac{\partial Q_{ a,  c}}{\partial  a} \Big|_{( a(\zeta_0),  c(\zeta_0))} ,  Q_{ a(\zeta_0),  c(\zeta_0)} \right\ra \\
&= \frac12\int (x-a) \frac{\partial}{\partial a} [ Q_{a,c}(x)^2 ]\, dx  \Big|_{(a(\zeta_0),c(\zeta_0)}
= -\frac12\int (x-a) \frac{\partial}{\partial x} [ Q_{a,c}(x)^2 ]\, dx \Big|_{(a(\zeta_0),c(\zeta_0)} \\
&= \frac12 \| Q_{a(\zeta_0),c(\zeta_0)} \|_{L_x^2}^2
\end{align*}
by integration by parts,
and
$$b_{21} = \left\la \frac{\partial}{\partial a}\Big|_{(a(\zeta_0),c(\zeta_0))} [(x-a(\zeta)) Q_{a,c}], \eta_0 \right\ra \qquad \implies \quad |b_{12}|\lesssim \|\eta_0\|_{L_x^2}$$
Finally $a_{22} = a_{22}^0+b_{22}$, where
\begin{align*}
a_{22}^0 &= 
 \left\la (x- a(\zeta))\frac{\partial Q_{ a,  c}}{\partial  c} \Big|_{( a(\zeta_0),  c(\zeta_0))} ,  Q_{ a(\zeta_0),  c(\zeta_0)}\right\ra \\
 & = \frac12 \frac{\partial}{\partial c} \Big|_{(a(\zeta_0),c(\zeta_0))} \la (x-a) Q_{a,c}, Q_{a,c} \ra = 0
 \end{align*}
 and
 $$
b_{22} = 
 \left\la (x- a(\zeta))\frac{\partial Q_{ a,  c}}{\partial  c} \Big|_{( a(\zeta_0),  c(\zeta_0))} , \eta_0 \right\ra \qquad \implies \quad |b_{22}| \lesssim \|\eta_0\|_{L_x^2}$$
 Thus
 $$
\begin{bmatrix}
a_{11} & a_{12} \\ a_{21} & a_{22} 
\end{bmatrix}
= \frac12 \|Q\|_{L_x^2}^2 \begin{bmatrix} 0 & 1 \\ 1 & 0 \end{bmatrix} + O( \|\eta_0\|_{L_x^2})$$
from which it follows that
$$
\begin{bmatrix}
a_{11} & a_{12} \\ a_{21} & a_{22} 
\end{bmatrix}^{-1}
= 2\|Q\|_{L_x^2}^{-2} \begin{bmatrix} 0 & 1 \\ 1 & 0 \end{bmatrix} + O( \|\eta_0\|_{L_x^2})$$

Solving \eqref{E:daceta} by inverting this $2\times 2$ matrix 
\begin{equation}
\label{E:est-62}
\begin{bmatrix} \frac{\partial  a}{\partial \zeta}(\zeta_0)(\delta \zeta) \\ \frac{\partial  c}{\partial \zeta}(\zeta_0) (\delta \zeta) \end{bmatrix}
=\begin{bmatrix} 
a_{11} & a_{12} \\ 
a_{21} & a_{22}
\end{bmatrix}^{-1}
\begin{bmatrix}
\la \delta \zeta , Q_{ a(\zeta_0),  c(\zeta_0)} \ra \\ 
\la \delta \zeta , (x- a(\zeta_0))Q_{ a(\zeta_0),  c(\zeta_0)} \ra
\end{bmatrix}
\end{equation}
which  gives the needed components of \eqref{E:est-60}.  Combining \eqref{E:est-57}, \eqref{E:est-61}, \eqref{E:est-60}, and \eqref{E:est-62}, we obtain
$$\eta(x) = \zeta(x) + \int_{s=0}^1 Q_{a(s\zeta),c(s\zeta)}'(x) \frac{\partial a}{\partial\zeta} \Big|_{s\zeta}(\zeta) \, ds - \int_{s=0}^1 [(x-a)Q_{a,c}]\Big|_{a(s\zeta),c(s\zeta)}'(x)\frac{\partial c}{\partial\zeta} \Big|_{s\zeta}(\zeta)  \, ds$$
\end{proof}

\begin{corollary}
\label{C:convert}
For each $\am$, $\cm$, and corresponding $a$, $c$,
\begin{equation}
\label{E:est-55}
\| \eta\|_{H_x^{1/2}} \lesssim \| \zeta\|_{H_x^{1/2}}
\end{equation}
Taking $\am(t)$, $\cm(t)$ and correspondingly $a(t)$, $c(t)$, along the flow\footnote{For this, we need only assume that $a(t)\sim t$ and $\frac12<c(t)<2$},
\begin{equation}
\label{E:coeff-conv03}
\sup_n \| \eta \|_{L_I^2 L_{x\in (n,n+1)}^2} \lesssim  (\ln h^{-1})\sup_n \| \zeta \|_{L_I^2L_{x\in (n,n+1)}^2}  + h^{1/2}\| \zeta\|_{L_I^2L_x^2}
\end{equation}
\end{corollary}
\begin{proof}
\eqref{E:est-55} follows directly from \eqref{E:coord-conv01} and the two equations after \eqref{E:coord-conv01}.  To prove \eqref{E:coeff-conv03}, we will use that for $a,b >0$ and $\alpha,\beta \in \mathbb{R}$,
$$\sup_t \la t - \alpha \ra^{-a} \la t - \beta \ra^{-b}  \lesssim \la \alpha - \beta \ra^{-\min(a,b)}$$
We know that $a(t) \sim t$.  Starting with \eqref{E:coord-conv01} (see the approximation \eqref{E:coord-conv02} to help with conceptualization), apply the $L_I^2L_{x\in (n,n+1)}^2$ norm for fixed $n$, and estimate as:
\begin{align*}
\| \eta \|_{L_I^2 L_{x\in (n,n+1)}^2} &\lesssim \| \zeta \|_{L_I^2 L_{x\in (n,n+1)}^2} + \int_x \| \zeta(x,t) \la n - a(t) \ra^{-2} \la x- a(t) \ra^{-1} \|_{L_{t\in I}^2} \, dx \\
&\lesssim \| \zeta \|_{L_I^2 L_{x\in (n,n+1)}^2} + \int_x \| \zeta(x,t) \|_{L_{t\in I}^2} \sup_{t\in I} \la n - a(t) \ra^{-2} \la x- a(t) \ra^{-1}  \, dx \\
&\lesssim \| \zeta \|_{L_I^2 L_{x\in (n,n+1)}^2} + \int_x \| \zeta(x,t) \|_{L_{t\in I}^2} \la n -x \ra^{-1}  \, dx
\end{align*}
Split the $x$-integral into $|x-n|< h^{-1}$ and $|x-n|>h^{-1}$.  The region $|x-n|<h^{-1}$ is divided into unit-sized $x$-pieces producing the factor $\sum_{|m|<h^{-1}} \la m\ra^{-1} \lesssim \ln h^{-1}$.  In the region $|x-n|>h^{-1}$, we apply Cauchy-Schwarz and use $\| \la x\ra^{-1} \|_{L_{|x|>h^{-1}}^2} \leq h^{1/2}$.  Together, this yields
$$\| \eta \|_{L_I^2 L_{x\in (n,n+1)}^2} \lesssim \| \zeta \|_{L_I^2 L_{x\in (n,n+1)}^2} + (\ln h^{-1})\sup_m \| \zeta \|_{L_I^2L_{x\in (m,m+1)}^2}  + h^{1/2}\| \zeta\|_{L_I^2L_x^2}
$$
From this, \eqref{E:coeff-conv03} follows.
\end{proof}

 Define the remainder $\eta$ according to 
\begin{equation}
\label{E:decomp2} 
u = Q_{a, c} + \eta
\end{equation}
imposing orthogonality conditions
\begin{equation}
\label{E:orth2}
 \la \eta, Q_{a,c} \ra =0 \,, \qquad \la \eta, (x-a) Q_{a,c} \ra = 0
\end{equation}
An implicit function theorem argument shows that there exists a unique choice of $(a,c)$ so that these orthogonality conditions hold.  This is the \emph{definition} of the parameters $(a(t),c(t))$ and of the remainder $\eta$. 

Starting with $\partial_t u = JE'(u)$, we substitute \eqref{E:decomp2} to obtain
$$\partial_t (Q_{a,c} + \eta) = J E'( Q_{a,c} + \eta)$$
Analogously to the derivation of \eqref{E:eta}, we find
\begin{equation}
\label{E:eta2}
\partial_t \eta = - \dot a \partial_a Q_{a,c} - \dot c \partial_c Q_{a,c} + J E'(Q_{a,c}) + JE''(Q_{a,c}) \eta  - \frac12 \partial_x (\eta^2)
\end{equation}
We re-center the equation for $\eta$ by letting 
$$w (y) = \eta (y+a) \quad \iff \quad \eta (x) = w(x-a)$$ 
The orthogonality conditions on $w$ read
\begin{equation}
\label{E:w-ortho}
\la w, Q_c \ra =0 \,, \qquad \la w, yQ_c \ra =0 .
\end{equation}
The equation for $w$ is
\begin{equation}
\label{E:w}
\begin{aligned}
\partial_t w = \;
&(\dot a-c+W(ha))\partial_yQ_c +(-\dot c c^{-1} + hW'(ha))\partial_y(yQ_c) + \partial_y(e_2Q_c) \\
&+\partial_y \mathcal{L}_c v + \partial_y(\dot a - c + W(hx)) w - \frac12 \partial_y w^2
\end{aligned}
\end{equation}
Here, \eqref{E:w} is analogous to \eqref{E:v}.

\begin{lemma}[symplectic parameter control]
\label{L:symp-ODE-control}
For all $t$, if $\frac12 \leq  c  \leq 2$ and $\|w \|_{L_y^2} \ll 1$, then 
\begin{equation}
\label{E:par2b}
\begin{aligned}
& |\dot a  -  c  + W(h a ) - \tfrac12 h^2W''(h a )  c ^{-1}  | \lesssim h^3 + \|w\la y \ra^{-1} \|_{L_y^2}^2 \\
&|\dot c   - hW'(h a ) c  - \tfrac12 h^3 W''(h a ) c ^{-1}| \lesssim h^4 + h^2(\ln h^{-1})\sup_{n\in \mathbb{Z}} \|w\|_{L_{n<y < n+1}^2} + \|w\la y \ra^{-1} \|_{L_y^2}^2
\end{aligned}
\end{equation}
Also, for a time interval $I$,
\begin{equation}
\label{E:par2bt}
\begin{aligned}
& \int_I |\dot a  -  c  + W(h a ) - \tfrac12 h^2W''(h a )  c ^{-1}  |\, dt \lesssim h^3|I| + \|w\la y \ra^{-1} \|_{L_I^2L_y^2}^2 \\
&\int_I |\dot c   - hW'(h a ) c  - \tfrac12 h^3 W''(h a ) c ^{-1}|\, dt \\
& \qquad \qquad \lesssim h^4|I| + h^2(\ln h^{-1})|I|^{1/2}\sup_{n\in \mathbb{Z}} \|w\|_{L_I^2L_{n<y < n+1}^2} + \|w\la y \ra^{-1} \|_{L_I^2L_y^2}^2
\end{aligned}
\end{equation}
\end{lemma}

\begin{proof}
Taking $\partial_t$ of the orthogonality condition $\la w, Q_ c  \ra =0 $, then exactly as in the proof of Lemma \ref{L:nonsymp-ODE-control}, we obtain 
$$|\dot c   - hW'(h a ) c   - \frac12 h^3 W''(h a ) c ^{-1}| (4\pi -  c ^{-1}\la w, \partial_y(yQ_ c )\ra) $$
$$\lesssim h^4 +  h^2 (\ln h^{-1}) \sup_n \|w\|_{L_{n\leq y\leq n+1}^2}+ \|\la y \ra^{-1} w\|_{L_y^2}^2$$
from which the second inequality in \eqref{E:par2b} follows.  The second inequality in \eqref{E:par2bt} also follows as in the proof of Lemma \ref{L:nonsymp-ODE-control}.

Now we prove the first inequality in \eqref{E:par2b}.
Taking $\partial_t$ of the orthogonality condition $0=\la w, y Q_ c  \ra$, we obtain

$$0 = \la \partial_t w, y Q_ c  \ra + \la w, y\partial_t Q_ c  \ra$$
For the first term, we substitute \eqref{E:w}, and for second term, we use that $\partial_tQ_ c  = \dot c   c ^{-1} \partial_y (y Q_ c) $, to obtain
\begin{align*}
0 &= ( \dot a  -  c  + W(h a )) \la \partial_yQ_ c , yQ_ c \ra && \leftarrow \text{I}\\
&\quad +(-\dot  c   c ^{-1} +hW'(h a )) \la \partial_y(yQ_ c ),yQ_ c  \ra  && \leftarrow \text{II}\\
&\quad + \la \partial_y(e_2Q_ c ), yQ_ c  \ra && \leftarrow \text{III}\\
&\quad + \la \partial_y\mathcal{L}_ c  w, yQ_ c  \ra && \leftarrow \text{IV}\\
&\quad + \la \partial_y(\dot  a  -  c  + W(h(y+ a )))w, y Q_ c  \ra && \leftarrow \text{V}\\
&\quad - \tfrac12 \la \partial_y w^2, y Q_ c  \ra && \leftarrow \text{VI}\\
&\quad + \dot  c   c ^{-1} \la w, y\partial_y (yQ_ c ) \ra && \leftarrow \text{VII}
\end{align*}

Given that $\la \partial_y Q_c, y Q_c \ra = - \frac12\|Q_c\|_{L_y^2}^2$ and 
$\|Q_ c \|_{L^2}^2 = 8\pi  c$, we have
$$\text{I} = -4\pi  c( \dot a  -  c  + W(h a ))$$
Also, given that $\partial_y(yQ_ c )$ is even and $yQ_ c $ is odd, we have $\la \partial_y(yQ_ c ),yQ_ c  \ra =0$ and thus $\text{II}=0$.  

To address term III, we carry out the Taylor expansion 
\begin{align*}
e_2 &= W(h(y+ a )) - W(h a ) - hW'(h a ) y \\
&= \frac12 h^2W''(h a )y^2 + \frac16 h^3W'''(h a ) y^3 + \frac1{24}h^4W''''(h(y_*+ a ))y^4
\end{align*}
 for some $y_*$ between $0$ and $y$, by the Lagrange form of the remainder.  Substituting
\begin{align*}
\text{III} &= -\la e_2, Q_ c  \partial_y (yQ_ c)  \ra \\
&= - \frac12 h^2 W''(h a ) \int y^2 Q_ c  \partial_y (yQ_ c)  \, dy - \frac16h^3 W'''(h a ) \int y^3 Q_ c  \partial_y (yQ_ c)  \, dy \\
& \qquad - \frac1{24}h^4W''''(h a ) \int y^4 Q_ c  \partial_y (yQ_ c)  
\end{align*}
Since $\int y^2 Q_ c  \partial_y (yQ_ c)  \, dy= -\frac12 \int y^2Q_c^2 \, dy = -\frac12 c^{-1} \int z^2Q^2 \, dz = -4\pi c^{-1}$,
$$\text{III} =  2\pi h^2W''(h a ) c^{-1}  + O(h^4)$$
To address Term IV, we use \eqref{E:Dc} and orthogonality condition $\la w, Q_c \ra =0$:
$$\text{IV} = - \la w, \mathcal{L}_ c  \partial_y(yQ_c)  \ra = - c \la w, \mathcal{L}_c \partial_c Q_c \ra =  c\la w, Q_c \ra = 0$$
In Term V, we expand 
$$W(h(y+ a ))= W(h a ) + hW'(h a )y + e_2(y, a )$$
to yield
$$\text{V} = (\dot  a  -  c  + W(h a ))\la \partial_y w, yQ_ c \ra - hW'(h a ) \la w, y \partial_y(y Q_ c)  \ra - \la e_2w, \partial_y(y Q_ c)  \ra$$
Note that the middle term combines with Term VII
$$\text{V}+\text{VII} = (\dot  a  -  c  + W(h a ))\la \partial_y w,  yQ_ c  \ra +(\dot c   c ^{-1} - hW'(h a )) \la w, y \partial_y(y Q_ c)  \ra - \la e_2w, \partial_y(y Q_ c) \ra$$
By Taylor's theorem with the Lagrange form of the remainder
$$e_2(y,a) = \frac12 h^2 W''(h(y_*+a))y^2$$
Let $R>0$ such that $\supp W \subset [-R,R]$.  Then $-a - Rh^{-1} \leq y \leq -a + Rh^{-1}$.  This gives
$$\la e_2 w, \partial_y(yQ_c) \ra = \frac12 h^2 \int_{-a-Rh^{-1}}^{-a+Rh^{-1}} W''(h(y_*+a) w y^2\partial_y(yQ_c) \, dy$$
Since $\| y^2\partial_y(yQ_c)\|_{L^\infty_y} \lesssim 1$ and $\|W''\|_{L^\infty_y}\lesssim 1$ Cauchy-Schwarz gives
$$|\la e_2 w, \partial_y(yQ_c) \ra| \lesssim h^2 \|w\|_{L^2} (Rh^{-1})^{1/2} \lesssim h^{3/2} \|w\|_{L^2_y}$$
Next,
$$\text{VI} = \frac12 \la w^2, \partial_y(yQ_c) \ra \lesssim \|w\la y \ra^{-1} \|_{L_y^2}^2$$
Combining the estimates for Terms I-VII, we obtain
\begin{align*}
\indentalign |(\dot a -c + W(ha)(1+ \frac{1}{4\pi c} \la w, \partial_y(yQ_c)\ra) - \frac12 h^2 W''(ha) c^{-2}| \\
&\lesssim h^3 + \|w\|_{L_y^2}^2 + |\dot c c^{-1} - hW'(ha)| \|w\|_{L_y^2}
\end{align*}
By the second inequality in \eqref{E:par2b}, $|\dot c c^{-1} - hW'(ha)|\lesssim h^2$, and therefore, the corresponding term in the inequality above can be bounded by the other terms.  From this it follows that
$$ |(\dot a -c + W(ha) - \frac12 h^2 W''(ha) c^{-2}| \lesssim h^3 + \|w\|_{L_y^2}^2$$
which is the first inequality in \eqref{E:par2b}.  The first inequality in \eqref{E:par2bt} follows from the first inequality in \eqref{E:par2b} after integrating in $t$.
\end{proof}

\begin{proposition}[symplectic decomposition estimates for (pBO)]
\label{P:symp-estimates}
There exists $\kappa \geq 1$, $\mu>0$, and $0<h_0\ll 1$ such that the following holds.
Let $0< h \leq h_0$ and suppose the initial data $u_0\in H_x^1$ satisfies
$$\| u_0(x) - Q_{0,1}(x) \|_{H_x^{1/2}} \leq h^{3/2}$$
Suppose that $u$ satisfying (pBO) with initial condition $u(x,0)=u_0(x)$ is decomposed as \eqref{E:decomp2} with remainder $\eta$ satisfying orthogonality conditions \eqref{E:orth2}.    For every $T>0$ such that $\frac12 \leq c(t) \leq 2$ for all $0\leq t \leq T$, 
we have that the recentered remainder $w(y,t) = \eta(y+a(t),t)$ satisfies
\begin{equation} 
\label{E:symp-estimates-eq1}
\begin{aligned}
&\|w\|_{L_{[0,T]}^\infty H_y^{1/2}}  \leq \kappa h^{3/2} e^{\mu hT} \\
&\sup_n \| w\|_{L_{[0,T]}^2L_{y\in(n,n+1)}^2} \leq \kappa h^{3/2} (\ln h^{-1}) e^{\mu hT}
\end{aligned}
\end{equation}
and the parameters $a(t)$, $c(t)$ satisfy the following bounds \eqref{E:par2b}.  \end{proposition}
\begin{proof}
From \eqref{E:est-55} and \eqref{E:coeff-conv03} in Corollary \ref{C:convert}, combined with \eqref{E:nonsymp-estimates-eq1} in Proposition \ref{P:nonsymp-estimates}, we immediately obtain \eqref{E:symp-estimates-eq1}.  The ODE bounds \eqref{E:symp-estimates-eq1} hold by Lemma \ref{L:symp-ODE-control}.
\end{proof}

Theorem \ref{T:main} can now be proved as a consequence of Proposition \ref{P:symp-estimates}. 

\begin{proof}[Proof that Proposition \ref{P:symp-estimates} implies Theorem \ref{T:main}]
By Proposition \ref{P:symp-estimates}, we have the estimate \eqref{E:symp-estimates-eq1} for $w$. 
The parameters $(a(t),c(t))$ in Proposition \ref{P:symp-estimates} satisfy the bounds in Lemma \ref{L:symp-ODE-control}.  Define $(A(s),C(s))$ by $a(t)=h^{-1}A(ht)$ and $c(t) = C(ht)$.  Then by \eqref{E:par2bt} and \eqref{E:symp-estimates-eq1}, $(A(s),C(s))$ satisfy
\begin{equation}
\label{E:ODE11}
\begin{aligned}
&\int_0^s |\dot { A} -  C+ W( A) +\frac12 C^{-2} h^2 W''(A)| \, d\tau \leq \kappa^2 h^3 (\log h^{-1}) e^{2\mu s}\\
&\int_0^s |\dot { C}  - CW'( A) - \frac12 C^{-2} h^2 W'''(A) |\, d\tau \leq \kappa^2 h^3 (\log h^{-1}) e^{2\mu s}
\end{aligned}
\end{equation}
on $0\leq s \leq  \min(\frac14\mu^{-1} \ln h^{-1},S_0)$.     Now apply Lemma \ref{L:gronwall} on ODE perturbation to compare the $(A,C)$ parameter dynamics with the so-called exact trajectory $(\hat A, \hat C)$ defined in \eqref{E:exact-traj}.   Specifically, we obtain that
$| A - \hat A| \lesssim h^3 e^{2\mu s}$ and $|C - \hat C| \lesssim h^3e^{2\mu s}$, and thus $|a-\hat a| \lesssim  h^2e^{2\mu ht}$ and $|c-\hat c| \lesssim h^3 e^{2\mu ht}$.   These bounds imply
$$\| Q_{\hat a, \hat c} - Q_{a, c} \|_{H_x^{1/2}} \lesssim h^2 e^{2\mu ht}$$
Therefore
$$\|u - Q_{\hat a, \hat c} \|_{H_x^{1/2}} \leq \| u - Q_{a,c} \|_{H_x^{1/2}} + \|Q_{a,c} - Q_{\hat a, \hat c} \|_{H_x^{1/2}} =  \|w \|_{H_x^{1/2}} + \|Q_{a,c} - Q_{\hat a, \hat c} \|_{H_x^{1/2}}$$
By \eqref{E:symp-estimates-eq1},
$$\|u - Q_{\hat a, \hat c} \|_{H_x^{1/2}} \lesssim h^{3/2} e^{\mu ht}$$
Thus Theorem \ref{T:main} follows.
\end{proof}

\begin{lemma}[Gronwall]
\label{L:gronwall}
Suppose $X, \bar X :\mathbb{R} \to \mathbb{R}^d$ solve
\begin{align*}
&\dot X(s) = f(X(s)) + h^2 g(X,s) \\
&\dot{\bar X}(s) = f(\bar X(s))
\end{align*}
with the same initial condition $X(0)=\bar X(0)$,
where $f: \mathbb{R}^d \to \mathbb{R}^d$, $ g: \mathbb{R}^{d+1} \to \mathbb{R}^d$.  Suppose that the $d\times d$ matrix $f'(X)$ is uniformly bounded:  for all $X\in \mathbb{R}^d$, 
$$\| f'(X)\|_{\ell^2} \leq \kappa$$
where $\ell^2$ is the square sum norm on the $d^2$ entries of the matrix.  Then
$$|X(s)-\bar X(s)|^2 \leq  h^4 \int_0^s e^{-(2\kappa+1)(s-s')}|g(X(s'),s')|^2 \, ds'$$
\end{lemma}
\begin{proof}
Let $V(s) = X(s)-\bar X(s)$.  Then ($| \bullet |$ is the usual square sum norm on $\mathbb{R}^d$)
\begin{equation}
\label{E:ODE1}
\frac{d}{ds} |V|^2 = 2V\dot V = 2V\cdot (f(X)-f(\bar X)) + 2h^2 V\cdot g(X,s)
\end{equation}
We have
$$f(X) - f(\bar X) = \int_{\sigma=0}^1 \frac{d}{d\sigma}[ f(\bar X+ \sigma V)] \, d\sigma = \left( \int_{\sigma=0}^1 f'(\bar X+ \sigma V) \, d\sigma \right)V$$
Then by Cauchy-Schwarz,
$$| f(X) - f(\bar X)| \leq \kappa |V|$$
Substituting this into \eqref{E:ODE1}, and using that $2h^2 V \cdot g(X,s) \leq |V|^2 + h^4|g(X,s)|^2$, we obtain
$$\frac{d}{ds} |V|^2 \leq (2\kappa+1) |V|^2 + h^4|g(X,s)|^2$$
The standard integrating factor method completes the proof. 
\end{proof}
In our application,
$$X = \begin{bmatrix} A \\ C \end{bmatrix} \,, \qquad f(X) = \begin{bmatrix}
C-W(A) \\ CW'(A) \end{bmatrix}$$
Then
$$f'(X) = \begin{bmatrix} -W'(A)  & 1 \\ CW''(A) & W'(A) \end{bmatrix}$$
Since $\frac12\leq C \leq 2$, this is uniformly bounded.

\section{Linear Liouville theorem for (BO) asymptotic stability}
\label{S:linear-Liouville}

In this section, we will prove Theorem \ref{T:linear-Liouville}. By Theorem \ref{T:local-virial},
\begin{equation}
\label{E:D1Dmod}
\sup_{y_0\in \mathbb{R}} \|  \la D_y \ra^{1/2} ((g_{\gamma,y_0}')^{1/2} v ) \|_{L_{[0,T]}^2L_y^2}^2 \lesssim_\gamma  \|v\|_{L_{[0,T]}^\infty L_y^2}^2 
\end{equation}
uniformly in $T>0$.  This implies the conveniently stated estimate
\begin{equation}
\label{E:linlio1}
\sup_{|I|=1} \| v \|^2_{L^2_{t\in \mathbb{R}} L^2_{y \in I }} \lesssim  \|v \|^2_{L^{\infty}_{t\in \mathbb{R}} L^2_y }
\end{equation}
where the supremum is taken over all unit-length intervals $I\subset \mathbb{R}$.  From \eqref{E:linlio1}, we will obtain
\begin{lemma} \label{L:decay}
We have 
\begin{equation}
\label{E:temperedbound}
\int_{t\in \mathbb{R}} \frac{1}{\langle t-t_0 \rangle^{4/5}} \int_{y\in\mathbb{R}} v^2 (t,y) dy dt < \infty
\end{equation}
uniformly in $t_0\in \mathbb{R}$.
\end{lemma}

\begin{proof}
By translation in time, it suffices to assume that $t_0=0$.
Split the integral into
$$
\int_t \frac{1}{\langle t \rangle^{4/5}} \int_{|y| < \langle t \rangle^{3/5}} v^2 (t,y) dy dt + \int_t \frac{1}{\langle t \rangle^{4/5}} \int_{|y| > \langle t \rangle^{3/5}} v^2 (t,y) dy dt := \text{I}+\text{II}
$$
and compute
\begin{align*}
\text{I} & = \int_t \frac{1}{ \langle t \rangle^{4/5}} \sum_n \int_{y \in [n, n+1], |y| \leq \langle t \rangle^{3/5}} v^2 (t,y) dy dt   \\
& =  \sum_n \int_t \frac{1}{ \langle t \rangle^{4/5}} \int_{y \in [n, n+1], |y| \leq \langle t \rangle^{3/5}} v^2 (t,y) dy dt   
\end{align*}
The condition on the inner integral implies that $\la n \ra \lesssim \la t \ra^{3/5}$, from which it follows that $\la t \ra^{-4/5} \leq \la n \ra^{-4/3}$.  Therefore, we can continue the estimate as
$$
\text{I} \lesssim \sum_n \frac{1}{\la n\ra^{4/3}} \int_t  \int_{y \in [n, n+1], |y| \leq \langle t \rangle^{3/5}} v^2 (t,y) dy dt \leq \sum_n \frac{1}{\la n\ra^{4/3}}  \sup_I \|v \|^2_{L^2_{[0,T]} L^2_{y\in I} }  < \infty
$$
by \eqref{E:linlio1}.  Moreover, by the uniform spatial decay hypothesis \eqref{E:uniformspatial}, we have
$$
\int_{|y| > \langle t \rangle^{3/5}} v^2 (t,y) dy \lesssim \frac{1}{\langle t \rangle^{3/5}}
$$
from which we obtain
$$
\text{II} \lesssim \int_t \frac{1}{\langle t \rangle^{7/5}} dt < \infty
$$
Since $\text{I}<\infty$ and $\text{II}<\infty$, \eqref{E:temperedbound} holds.
\end{proof}

By Proposition 2 on p. 920 of Kenig \& Martel \cite{KM}, there exists $A\gg 1$ such that with
\begin{equation}
\label{E:monphi}
\phi (y)= \frac{\pi}{2} + \arctan (\frac{y}{A})
\end{equation}
the following holds:  for any $\lambda \in (0,1)$, $t\leq t_0$, and $y_0 >1$, we have the monotonicity estimate
\begin{equation}
\label{E:mon}
\begin{aligned}
\indentalign \int v^2 (y,t_0) ( \phi (y- y_0) -  \phi (-y_0)  ) dy \\
& \leq \int v^2 (y,t) ( \phi (y- y_0 -\lambda (t_0 - t)) -  \phi (-y_0 - \lambda (t_0 - t))  ) dy \\
& \qquad + C \int^{t_0}_t \frac{\| v(t') \|^2_{L^2_y}  }{ (y_0 + \lambda (t_0 - t'))^2 } dt'  = p_1(t)+p_2(t)+p_3(t)
\end{aligned}
\end{equation}
We have decomposed the right-side as
$$p_1(t) = \int_{y> \frac12(y_0+\lambda(t_0-t))} v^2 (y,t) ( \phi (y- y_0 -\lambda (t_0 - t)) -  \phi (-y_0 - \lambda (t_0 - t))  ) dy$$
$$p_2(t) = \int_{y< \frac12(y_0+\lambda(t_0-t))} v^2 (y,t) ( \phi (y- y_0 -\lambda (t_0 - t)) -  \phi (-y_0 - \lambda (t_0 - t))  ) dy$$
$$p_3(t) = C\int^{t_0}_t \frac{\| v(t') \|^2_{L^2_y}  }{ (y_0 + \lambda (t_0 - t'))^2 } dt'$$
Note that
$$p_3(t) \lesssim \int_t^{t_0} \la t'-t_0\ra^{-4/5} \|v(t')\|_{L_y^2}^2 \, dt' \; \sup_{t'\leq t_0} \left[ \frac{\la t'-t_0\ra^{4/5}}{(y_0+\lambda(t_0-t'))^2} \right]$$
Thus by \eqref{E:temperedbound}, $p_3(t) \lesssim y_0^{-6/5}$ uniformly in $t<t_0$.  Next, we will show that $\lim_{t\to -\infty} p_1(t) = 0$ and  $\lim_{t\to -\infty} p_2(t) = 0$.  Indeed, by \eqref{E:uniformspatial},
$$|p_1(t)| \lesssim \frac{1}{y_0 + \lambda(t_0-t)} \quad \implies \quad \lim_{t\to -\infty} p_1(t) = 0$$ Also,
$$|p_2(t)| \lesssim \int_{y} v^2 (y,t)  dy \; \sup_{y<\frac12(y_0+\lambda(t_0-t))} ( \phi (y- y_0 -\lambda (t_0 - t)) -  \phi (-y_0 - \lambda (t_0 - t)))$$
and from the formula \eqref{E:monphi} for $\phi(y)$,
$$\sup_{y<\frac12(y_0+\lambda(t_0-t))} ( \phi (y- y_0 -\lambda (t_0 - t)) -  \phi (-y_0 - \lambda (t_0 - t))) \leq 2 \phi(-\tfrac12(y_0+\lambda(t_0-t)))$$
from which it follows that $\lim_{t\to-\infty} p_2(t) =0$.  

From these estimates on $p_1(t)$, $p_2(t)$ and $p_3(t)$, we see that by taking $t\to -\infty$ in \eqref{E:mon}, we obtain that for all $t_0\in \mathbb{R}$
\begin{equation}
\label{E:mon1}
\int v^2 ( y,t_0) ( \phi (y- y_0) -  \phi (-y_0)  ) dy \lesssim y_0^{-6/5}
\end{equation}
The whole argument leading to \eqref{E:mon1} applies with $v(y,t)$ replaced by $v(-y,-t)$, so that we can also assert that \eqref{E:mon1} holds with $v(y,t)$ replaced by $v(-y,-t)$.  Thus, for all $t_1\in \mathbb{R}$
\begin{equation}
\label{E:mon2}
\int v^2 ( -y,-t_1) ( \phi (y- y_0) -  \phi (-y_0)  ) dy \lesssim y_0^{-6/5}
\end{equation}
Changing variable $-y \mapsto y$, and using that $\phi(-y-y_0)-\phi(-y_0) = \phi(y_0) - \phi(y+y_0)$ (which follows from the formula \eqref{E:monphi} for $\phi$), we have
\begin{equation}
\label{E:mon3}
\int v^2 ( y, -t_1) ( \phi (y_0) -  \phi (y+y_0)  ) dy \lesssim y_0^{-6/5}
\end{equation}
Taking $t_1=-t_0$, and adding \eqref{E:mon1} and \eqref{E:mon3}, we obtain
\begin{equation}
\label{E:mon4}
\int v^2 ( y, t_0) \rho(y,y_0) dy \lesssim y_0^{-6/5}
\end{equation}
where
$$\rho(y,y_0) = \phi(y-y_0)- \phi(-y_0) - \phi(y+y_0) + \phi(y_0)$$
From the formula \eqref{E:monphi} for $\phi$, it follows that $\rho$ is even in $y$ (that is $\rho(-y,y_0) = \rho(y,y_0)$) and $\partial_y \rho(y,y_0) \geq 0$ for $y>0$.  Since $\rho(0,y_0)=0$ and $\rho(y_0,y_0) \geq \frac{\pi}{6}$ whenever $y_0\geq \sqrt{3}A$, it follows that $\rho(y,y_0) \geq 0$ for all $y\in \mathbb{R}$ and $\rho(y,y_0) \geq \frac{\pi}{6}$ when $|y| \geq y_0$ (provided $y_0 \geq \sqrt{3}A$).  Thus from \eqref{E:mon4} 
$$\forall \; y_0 > \sqrt{3}A, \quad \int_{|y| > y_0} v^2 (t_0, y) dy \lesssim y_0^{-6/5}
$$
from which we can integrate in $y_0$ and find that uniformly for all $t\in \mathbb{R}$
\begin{equation}
\label{E:finitevar}
\int_{y\in \mathbb{R}} |y| v^2 (t, y) dy \lesssim 1
\end{equation}
The (nonlocalized) virial identity obtained by computing $\partial_t \int y v(y,t)^2 \, dy$, substituting the equation \eqref{E:D1C} for $v$, applying integration by parts in $y$, and integrating over $t_1\leq t \leq t_2$, is
$$ \int y v^2 (t_2) \, dy - \int y v^2 (t_1) \,dy  =  - \|v\|_{L_y^2}^2 - 2\|D_y^{1/2} v\|_{L_y^2}^2 + \int (Q  - yQ') v^2 \, dy$$
From this, it follows that
$$\| v \|_{L_{[t_1,t_2]}^2 H_y^{1/2}}^2 \lesssim \| |y|^{1/2} v \|_{L_{[t_1,t_2]}^\infty L_y^2}^2 + \| \la y \ra^{-1} v \|_{L_{[t_1,t_2]}^2L_y^2}^2$$
By \eqref{E:linlio1} and \eqref{E:finitevar}, the right side is bounded uniformly for all $t_1<t_2$, so taking $t_1=0$ and $t_2\to +\infty$ implies $\| v \|_{L_{t>0}^2 H_y^{1/2}}<\infty$.  Hence, there exists a time sequence $t_n \rightarrow \infty$ along which  $ \|v(t_n) \|_{H_y^{1/2}} \rightarrow 0$ as $n \rightarrow +\infty$.
Now, from the fact that $\mathcal{L} Q' =0$ we can deduce that the quantity $ \langle \mathcal{L} v(t), v(t) \rangle$ is conserved in time.   Hence for any $t$,
$$ 
\langle \mathcal{L} v(t), v(t)\rangle = \lim_{t_n \rightarrow +\infty} \langle \mathcal{L} v(t_n), v(t_n) \rangle =0
$$
But since $v$ satisfies the orthogonality conditions \eqref{E:D1E}, it follows that $\langle \mathcal{L} v(t), v(t)\rangle  \gtrsim \|v(t) \|^2_{H^{1/2}} $.  Therefore, $v(t)\equiv 0$ for all $t$, as claimed.

\end{document}